 \documentclass{amsart}

\usepackage{amssymb}

\usepackage[cmtip,all]{xy}

\usepackage{xcolor}

\newcommand{\C}{{\mathbb C}}

\newcommand{\Z}{{\mathbb Z}}

\newcommand\Ker{\operatorname{Ker}}

\newcommand\rk{\operatorname{rk}}

\newcommand\Simpl{\operatorname{Spl}}
\newcommand\Iso{\operatorname{Iso}}
\newcommand\Sub{\operatorname{Sub}}
\newcommand\Ind{\operatorname{Ind}}
\newcommand\rank{\operatorname{rank}}
\newcommand\Tr{\operatorname{Tr}}
\newcommand\Ob{\operatorname{Ob}}
\newcommand\Pic{\operatorname{Pic}}
\newcommand\lev{\operatorname{lev}}
\newcommand\Inv{\operatorname{Inv}}
\newcommand\length{\operatorname{length}_{\mathcal R}}
\newcommand\plength{\operatorname{p-length}}
\newcommand\End{\operatorname{End}}
\newcommand\Aut{\operatorname{Aut}}

\newcommand\Hom{\operatorname{Hom}}

\newcommand\Id{\operatorname{Id}}

\newcommand\Ext{\operatorname{Ext}}

\newcommand\rat{\mathrm{rat}}
\newcommand\tffr{\mathrm{tffr}}
\newcommand\Ra{\mathcal{R}}
\newcommand\RC{\mathcal{RC}}

\newcommand{\sld}{{\mathfrak{sl}(2)}}
\newcommand{\Spl}{\operatorname{Spl}}
\newcommand{\PSpl}{\operatorname{PSpl}}

\renewcommand{\sl}{\mathfrak{sl}}

\newcommand\Mod{\text{-}\mathrm{Mod}}

\newtheorem{thm}{Theorem}[section]
\newtheorem*{theo}{Theorem}
\newtheorem{prop}[thm]{Proposition}
\newtheorem{cor}[thm]{Corollary}
\newtheorem{lem}[thm]{Lemma}

\theoremstyle{defn}
\newtheorem{defn}[thm]{Definition}

\theoremstyle{rem}
\newtheorem{rem}[thm]{Remark}
\theoremstyle{exam}
\newtheorem{exam}[thm]{Example}

\numberwithin{equation}{section}

\begin{document}

\title[The Grothendieck and Picard groups of torsion free $\sld$-modules\qquad  ]{The Grothendieck and Picard groups of finite rank torsion free $\sl(2)$-modules}


%    author one information
\author[F. J. Plaza Mart\'{\i}n]{F. J. Plaza Mart\'{\i}n}
\author[C. Tejero Prieto]{C. Tejero Prieto}
\address{Departamento de Matem\'aticas and IUFFyM, Universidad de
Salamanca,  Plaza de la Merced 1-4
        \\
        37008 Salamanca. Spain.
        \\
         Tel: +34 923294460. Fax: +34 923294583}
\curraddr{}
\email{fplaza@usal.es}\email{carlost@usal.es}
\thanks{}

%    author two information
%\author[C. Tejero Prieto]{C. Tejero Prieto}
%
%%\address{Departamento de Matem\'aticas and IUFFyM, Universidad de
%%Salamanca,  Plaza de la Merced 1-4
%%        \\
%%        37008 Salamanca. Spain.
%%        \\
%%         Tel: +34 923294460. Fax: +34 923294583}
%\curraddr{}
%\email{carlost@usal.es}
\thanks{This work is supported by the research contracts 
PGC2018-099599-B-I00 of Ministerio de Ciencia e Innovaci\'on, Spain (first author) and MTM2017-86042-P (second author) of Ministerio de Econom\'ia, Industria y Competividad, Spain. The second author has received support also  from GIR STAMGAD (JCyL)}

%    \subjclass is required.
\date{}

\dedicatory{}

%    Abstract is required.
\begin{abstract} The classification problem for simple $\sld$-modules leads in a natural way to the study of the category of finite rank torsion free $\sld$-modules and its subcategory of rational $\sld$-modules. We prove that the rationalization functor induces an identification between the isomorphism classes of simple modules of these categories. This raises the question of what is the precise relationship between other invariants associated with them. We give a complete solution to this problem for the Grothendieck and Picard groups, obtaining along theway several new results regarding these categories that are interesting in their own right. 
\end{abstract}

\maketitle


\section{Introduction}

Simple $\sld$-modules are quotients of the universal enveloping algebra of $\sld$ and hence  have countable $\C$-dimension. Thus the generalized version of the Schur Lemma, due to Dixmier, applies to them, showing that every simple $\sld$-module is a Casimir module,  (see Theorem  \ref{thm:Schur-Dixmier}). This is the first reduction in the classification problem of simple $\sld$-modules.

The next important step in this program is the following key dichotomy: a simple $\sl(2)$-module is either a weight module or a torsion free module,; we refer the reader to \cite[Thm. 6.3]{Mazor} for further details. 

There is a well known (\cite[Theorem 3.32]{Mazor}) explicit classification of simple weight modules into four families consisting of finite dimensional, Verma, anti-Verma and dense modules. This classification has allowed us in \cite{Plaza-Tejero-Ext} to establish the extension properties  of weight $\sld$-modules to Witt and Virasoro algebras.

On the other hand, Block proved in \cite{Block} that simple torsion free modules can also be classified, although admittedly in a less explicit way, since they are parametrized by the similarity classes of irreducible elements of a certain non commutative euclidean algebra. See Theorem \ref{thm:isom-classes}  and Bavula's paper \cite{Bavula2} for the precise statement. 

In a previous paper \cite{Plaza-Tejero-Con} we proved that every simple torsion free module has finite rank, (Theorem \ref{thm:simple-tf-is-of-fr}). Therefore, in the context of the classification program it is sensible to study in detail the category $\sld\Mod_\tffr$ of torsion free finite rank $\sld$-modules. The finite rank result, together with Bavula's approach, lead us in a natural way to introduce the full subcategory $\Ra$ of $\sld\Mod_\tffr$ formed by rational $\sld$-modules. One advantage of considering rational representations is that they are finite dimensional $\C(z)$-vector spaces, and torsion free finite rank $\sld$-modules are exactly their $\sld$-submodules. However, as we shall prove, their major convenience lies in the existence of a bijection between the isomorphism classes of simple torsion free modules $\widehat\Spl(\sld\Mod_\tffr)$ and the isomorphism classes of rational representations that are $\Ra$- irreducible, (see Theorem \ref{thm:isom-classes-generalized}). This allows us to reduce the determination of  $\widehat\Spl(\sld\Mod_\tffr)$ to a problem on finite dimensional $\C(z)$-vector spaces instead of a problem on $\C[z]$-modules. The situation is reminiscent of that encountered in integral representation theory, see for instance  \cite{Reiner}. A key role for understanding this relationship is played by  the rationalization functor $F_\rat\colon \sld\Mod_\tffr\to\Ra$ that sends a torsion free finite rank $\sld$-module to the rational module obtained by localization, (see Definition \ref{defn:rationalization-functor}). 

One of the central motivations of this paper is precisely to study to what extent the equivalence between the isomorphism classes of simple objects in the categories $\sld\Mod_\tffr$, $\Ra$ can be extended to other types of invariants associated to them. In particular, we will endeavor to analyze in detail the relationship between their Grothendieck groups.  To achieve this goal, we start  with a careful study of $\Ra$.

A key result of this paper, on which many others  depend, is the existence of minimal polynomials over $\C$ for the Casimir operators of rational representations:

\begin{theo}[\textbf{\ref{t:caracCasC}}]
If $(W,\rho)$ is a rational $\sl(2)$-representation, then the minimal polynomial  $M_{C_\rho}(t)$ of its Casimir operator  $C_\rho$, considered as a $\C(z)$-linear endomorphism of $W$, has its coefficients in $\C$; that is, $M_{C_\rho}(t)\in \C[t]$. Therefore, the Casimir operator $C_\rho$ understood as a $\C$-linear endomorphism of $W$ has minimal polynomial  $M_{C_\rho}(t)$.
\end{theo}

From this we obtain that every $\sld$-endomorphism of a rational representation also has  a minimal polynomial over $\C$, (Proposition \ref{p:finite+extension}). Another key result that we obtain from Theorem {\ref{t:caracCasC} is that every rational representation decomposes into a finite direct sum of generalized Casimir rational $\sld$-modules, (Theorem \ref{thm:finite+decomp}). This gives a decomposition of  the abelian category $\Ra$ into the $\Hom$-orthogonal direct sum of the abelian subcategories of generalized rational Casimir modules:
		\[
		{\mathcal R}\,=\, \bigoplus_{\mu\in \C } {\mathcal{RC}}^\bullet_{\mu}\, . 
		\]

We also obtain important consequences regarding the problem of classification of simple torsion free $\sld$-modules. In Theorem \ref{thm:identification-isomorphism-classes} we prove that one has the following identifications $$\widehat{\Simpl}({\sl(2)\mathrm{\Mod}_\mathrm{tf}})=\widehat{\Simpl}({\sl(2)\mathrm{\Mod}_\mathrm{tffr}})\xrightarrow[\sim]{\widehat F_\mathrm{rat}}\widehat{\Simpl}_{\mathcal{RC}}({\mathcal{RC}})=\widehat{\Simpl}_{\mathcal R}({\mathcal R}).$$  Therefore, determining the isomorphism classes of simple torsion free $\sl(2)$-modules is equivalent to the determination of the isomorphism classes of $\Ra$-irreducible rational $\sl(2)$-modules. The latter is an easier  problemsince we are dealing with $\C(z)$-vector spaces instead of $\C[z]$-modules. Taking into account the above mentioned dichotomy, this gives a complete description of the isomorphism classes of simple $\sl(2)$-modules \begin{align*}\widehat{\Simpl}({\sl(2)\mathrm{\Mod}})&= \widehat{\Simpl}({\sl(2)\mathrm{\Mod}}_\mathrm{weight})\coprod \widehat{\Simpl}({\sl(2)\mathrm{-Mod}_\mathrm{tf}})\simeq \\ &\simeq \widehat{\Simpl}({\sl(2)\mathrm{\Mod}}_\mathrm{weight})\coprod \widehat{\Simpl}_{\mathcal{R}}({\mathcal{R}}).\end{align*}

The categorical properties of rational representations are also very interesting. We prove in Theorems \ref{thm:Jordan-Holder-Krull-Schmidt} and \ref{thm:rational-essentially-small}  that $\Ra$ is an essentially small abelian category of finite length, and therefore it is a Jordan-H\"older category and a Krull-Schmidt category. Moreover, we will see in Proposition  \ref{prop:end=C} that the $\sl(2)$-endomorphisms of an $\Ra$-irreducible rational module are just the $\C$ multiples of the identity. This result can not be obtained either by means of Dixmier's generalization of Schur's Lemma, because rational modules have $\C$-dimension equal to the continuum, or by Quillen's Lemma, because rational modules are never $\sld$-simple. From this we also get that $\Ra$ is a $\Hom_\C$-finite category, (Theorem \ref{t:hom-finite}). 

Most of these properties of rational modules have their reflection on the category $\sld\Mod_\tffr$. This is the case for the existence of minimal polynomials for Casimir operators, (Theorem \ref{thm:minimalCasC}), and $\sld$-endomorphisms, (Proposition \ref{p:finite+extension-frtf}). We also obtain in Theorem \ref{thm:finite+decomp-tffr}   a decomposition of the exact category $\sl(2)\Mod_\mathrm{tffr}$  into the $\Hom$-orthogonal direct sum of the exact subcategories of generalized Casimir modules:
		\[
		\sl(2)\Mod_\mathrm{tffr}\,=\, \bigoplus_{\mu\in \C} {\mathcal{C}}^\bullet_{\mu,\mathrm{tffr}}\, . \]
As far as categorical properties are concerned, the behavior of torsion free finite rank modules  is nearly as good as in the  rational case. We prove in Proposition \ref{prop:exact-tffr} and Theorem \ref{thm:tffr-is-essentially-small}  that $\sl(2)\Mod_\mathrm{tffr}$ is an essentially small exact category. Furthermore, in Theorem \ref{t:hom-finite-frtf} we will see that it is $\Hom_\C$-finite and therefore it is a Krull-Schmidt category. We also have a very precise relationship between $\Ra$ and the ambient category $\sl(2)\Mod_\mathrm{tffr}$. Indeed we will prove  that $\Ra$ is a thick subcategory, (Proposition \ref{prop:exact-tffr}), and the rationalization functor $F_\rat$ is a retraction of the natural embedding $i_\tffr\colon\Ra\hookrightarrow \sl(2)\Mod_\mathrm{tffr}$, (Proposition \ref{prop:rationalization-is-a-retraction}). Moreover, in Theorem \ref{thm:rational-is-reflective-localization-of-tffr}  we show that $\Ra$ is a faithful reflective localization of $\sl(2)\Mod_\mathrm{tffr}$  with localization functor $F_\rat$.

The most notable property missing from the rational case is the finite length condition, (see Remark \ref{rem:finite-length-missed}). However,  we will prove  in Proposition \ref{prop:finite-pure-length-category} that $\sl(2)\Mod_\mathrm{tffr}$  satisfies both the ascending and descending chain conditions on pure submodules and thus is a finite pure length category. By Proposition \ref{prop:p-lenght-is-R-length-for-rational-modules}, pure length agrees on rational modules  with $\Ra$-length.  Moreover, introducing the concept of purely irreducible $\sld$-module,  we show in Proposition \ref{prop:existence-of-filtrations}  that every torsion free finite rank $\sld$-module admits pure composition series and we prove in Theorem  \ref{thm:JH-pure-composition-series} the key result that the Jordan-H\"older theorem holds for them.

Armed with the knowledge we have gained about the structural properties of these categories, we proceed to study their Grothendieck groups. The main tools we use in their computation are devissage subcategories and a generalization of Heller's devissage theorem that identifies special instances of them, (Theorem \ref{thm:devissage-generalizado}).  

The additive Grothendieck group $K^\oplus_0(\Ra)$ of the category of rational representations is completely determined by the Krull-Schmidt property of $\Ra$, which implies that indecomposable generalized Casimir rational modules form a devissage subcategory, as shown in Theorem \ref{thm:additive-Grothendieck-group}, and also yields the decomposition ${\mathcal R}\,=\, \bigoplus_{\mu\in \C } {\mathcal{RC}}^\bullet_{\mu}\,$ provided by the minimal polynomial of the Casimir operator, (see Theorem \ref{thm:descomposición-Grothendieck-aditivo} and Corollary \ref{cor:compatibility-aditivo}). 
In turn, the structure of the general Grothendieck group $K_0(\Ra)$ is completely determined by the above decomposition and the  devissage subcategory formed by the $\Ra$-irreducible Casimir modules, since $\Ra$ is a Jordan-H\"older category, (Theorems  \ref{t:Grothendieck-rational} and \ref{thm:decomposition-Grothendieck}). We also analyze the compatibility of the devissage procedure with respect to the canonical filtration of a generalized Casimir rational module $W\in\RC^\bullet$ described in Proposition \ref{prop:filtration}, proving that rational Casimir modules form a devissage subcategory for $\RC^\bullet$, (Theorem \ref{thm:first-devissage} and Corollary \ref{cor:compatibility}). Therefore, the algorithm for disassembling (devissage) the class of a rational module $W$ in the Grothendieck group $[W]\in K_0(\Ra)$  into a sum of classes of irreducible Casimir modules in $K_0(\Spl(\RC))$ proceeds in stages. First, we decompose $W$ into a sum of generalized Casimir representations according to the minimal polynomial of its Casimir operator. Then these classes of generalized Casimir modules are decomposed into a sum of Casimir rational representations by means of the canonical filtration. Finally, each one of these Casimir classes is decomposed  by means of a composition series into a sum of classes of irreducible Casimir modules.

We obtain analogous results for the Grothendieck groups of $\sld\Mod_\tffr$, (Theorems \ref{t:Grothendieck-rational-tffr} and \ref{thm:decomposition-Grothendieck-tffr},) where now the key points are to prove that the category of purely irreducible modules $\PSpl(\sld)$ is a devissage subcategory, and use the decomposition $\sld\Mod_\tffr=\bigoplus_{\mu\in\C}\mathcal C^\bullet_{\mu,\tffr}$ obtained from the minimal polynomial of the Casimir operator. We also have compatibility with the canonical filtration,  (Theorem \ref{thm:first-devissage-tffr}), although the proof is more involved since now we cannot use Theorem \ref{thm:devissage-generalizado}.

The relationship between the general Grothendieck groups of $\Ra$ and $\sld\Mod_\tffr$ can be established by an analogue of Heller's localization theorem for the quotient of an abelian category under a Serre subcategory. There is a  well known  one to one correspondence between Serre subcategories and torsion theories. The localization theorem says that the map induced on Grothendieck groups by the quotient functor is a surjection and has in its kernel the Grothendieck group of  the corresponding torsion theory, \cite[Theorem 5.13, pag.  115]{Swan}. Although the situation here is different because $\sld\Mod_\tffr$ is only an exact category,  we are able to establish our localization result due to the fact that the rationalization functor $F_\rat\colon \sld\Mod_\tffr\to \Ra$, as explained above, is a reflective localization. Indeed, we prove in Theorem \ref{thm:localization-theorem}
that there exists a naturally split short exact sequence
	  	\[
	  0 \to    \mathcal{VT}(\sl(2)\Mod_\mathrm{tffr}) \to   K_0(\sl(2)\Mod_\mathrm{tffr})\xrightarrow{F_{\rat*}}  K_0(\Ra) \to  0
	  \]
where $\mathcal{VT}(\sl(2)\Mod_\mathrm{tffr})$ are the virtual torsion modules of finite rank. This is the best analogue one could hope for, since the kernel of $F_\rat$ is just the zero module.

The categories $\sld\Mod_\tffr$, $\Ra$ are stratified according to rank and $\C(z)$-\-di\-me\-nsion, respectively. The rationalization functor is compatible with these stratifications and therefore induces an identification between the isomorphism classes of simple objects of the corresponding strata $$\widehat{\Simpl}({\sl(2)\mathrm{\Mod} ^m_\mathrm{tffr}})\xrightarrow[\sim]{\widehat F_\mathrm{rat}}\widehat{\Simpl}_{\mathcal R}({\mathcal R^m}).$$ We study in detail the one dimensional stratum since it  plays a key role for several results ranging from the rationalization and extension problems to the Picard group. In the first place, by  Corollary \ref{cor:rational-onedim-are-irred-Casimir} one has the identifications $\Ra^1=\Spl_\Ra(\Ra^1)=\RC^1$ and by Theorem \ref{thm:identification-isomorphism-classes} we get $ \widehat{\Spl}(\sld\Mod^1_\tffr) \simeq \widehat{\Ra^1}$. We describe $\widehat{\Ra^1}$ explicitly as a quotient of the space $\C(z)^\times$ of invertible rational functions, (Theorem \ref{t:equivalentrational}). This makes it possible to answer the question as to whether a one dimensional rational representation arises as the rationalization of a rank one polynomial representation, (Theorem \ref{t:ratio-poly}). We also analyze the extension problem  for rational representations.  It is a classical result of Bavula \cite[Theorem 3]{Bavula2}, (see also \cite[[Theorem 6.40 pag. 210]{Mazor}) that the extension groups of finite length $\sld$-modules have finite $\C$-dimension. We have proved in Theorem \ref{t:hom-finite} that the space of $\sld$-homomorphisms between rational $\sld$-modules  is also finite $\C$-dimensional. Therefore it is natural to analyze whether this is also true for the higher extension groups between rational modules. However, we show that this is not the case in Corollary \ref{cor:non-finite-Exts-rational-modules} where, thanks to the explicit characterization of $\mathcal R^1$, we describe explicit rational modules whose first extension group has infinite dimension over $\C$.

Since  the category of rational representations $\Ra$ is a $\C$-linear abelian subcategory of the category of finite dimensional $\C(z)$-vector spaces, it is natural to study which of its properties remain valid for $\Ra$. It is remarkable that the tensor product $\otimes$ of $\C(z)$-vector spaces restricts to $\Ra$, because this is not the usual construction for $\sld$-modules, where the tensor product  is taken over $\C$. Moreover, we show that this rational tensor product has very nice properties that can be summarized  by saying that $(\Ra,\otimes)$ is a closed symmetric monoidal category, (see Theorem \ref{t:Rmonoidal} where we also describe the internal $\Hom$). This makes it possible, following May \cite{May}, to define the Picard group $\Pic(\Ra)$ of the category of rational representations. In Theorem  \ref{thm:invertible-elements}  we show that it is completely determined by the one dimensional stratum and that $\Pic(\Ra)=\widehat{\Ra^1}$. Moreover, we prove  that the level map gives rise to a split short exact sequence $$0\to \Pic^0(\Ra)\to \Pic(\Ra)\xrightarrow{\lev}\C\to 0$$ and that we have an identification $\Pic^0(\Ra)=\C(z)^\times_0$ with the multiplicative subgroup of $\C(z)^\times$ formed by the invertible rational functions without zeros or poles at the origin, (Theorem \ref{thm:Picard-group}). The symmetric monoidal structure also makes  it possible to introduce ring structures on the Grothendieck groups, (see Section \ref{subsect:Grothendieck-rings}).

The identification $\widehat F_\rat\colon  \widehat{\Spl}(\sld\Mod^1_\tffr) \xrightarrow{\sim} \widehat{\Ra^1}=\Pic(\Ra)$ given in Theorem \ref{thm:identification-isomorphism-classes} allows one to formally define  the Picard group of the category $\sld\Mod_\tffr$ as the group $\Pic(\sld\Mod_\tffr)=\widehat{\Spl}(\sld\Mod^1_\tffr)$.  The existence of this group structure on $\widehat{\Spl}(\sld\Mod^1_\tffr)$ might point to the existence of an appropriate symmetric monoidal structure on the category $\sld\Mod_\tffr$ such that its Picard group is the one introduced above. We plan to investigate this in future work.

%%%%%%%%%%%%%%%%%%%%%%%%%%%%%%%%%%%%%%%%%

We summarize the notation used in the paper. 
\begin{itemize}
	\item  $\sl(2)\text{-}\mathrm{Mod}$ is the category of $\sl(2)$-modules;
	\item $\mathcal C$ is the full subcategory of $\sl(2)\text{-}\mathrm{Mod}$ formed by all Casimir $\sl(2)$-modules;
	\item $\mathcal R$ (resp. $\mathcal RC$) is the full subcategory of $\sl(2)\text{-}\mathrm{Mod}$ formed by all rational $\sl(2)$-modules (resp. rational Casimir  $\sl(2)$-modules).
	\end{itemize}
For any subcategory $\mathcal A$ of $\sl(2)\text{-}\mathrm{Mod}$, we have the following full subcategories of $\mathcal A$:
\begin{itemize}
	\item  $\mathcal A_{\mathrm{tf}}$, of torsion free $\sl(2)$-modules;
	\item  $\mathcal A_{\mathrm{tffr}}$, of torsion free, finite rank $\sl(2)$-modules;
	\item$\mathcal A_{\mathrm{tfft}}$, of torsion free, finite type  $\sl(2)$-modules;
	\item $\Ind(\mathcal A)$,  of indecomposable $\sl(2)$-modules of $\mathcal A$;
	\item $\Ind_{\mathcal A}(\mathcal A)$, of $\mathcal A$-indecomposable objects of $\mathcal A$;
	\item $\Simpl(\mathcal A)$,  of simple $\sl(2)$-modules of $\mathcal A$;
	\item $\Simpl_{\mathcal A}(\mathcal A)$, of $\mathcal A$-simple or $\mathcal A$-irreducible objects of $\mathcal A$.
\end{itemize}
and we denote:	
	\begin{itemize}
	\item $\widehat{\Ind}(\mathcal A)$,  the  isomorphism classes of indecomposable $\sl(2)$-modules of $\mathcal A$;
	\item $\widehat\Ind_{\mathcal A}( \mathcal A)$, the isomorphism classes of $\mathcal A$-indecomposable objects of $\mathcal A$;
	\item $\widehat{\Simpl}(\mathcal A)$, the isomorphism classes of simple $\sl(2)$-modules of $\mathcal A$;
	\item  $\widehat{\Simpl}_{\mathcal A}(\mathcal A)$, the isomorphism classes of $\mathcal A$-simple objects of $\mathcal A$;
	\end{itemize}
	
Finally, given two subcategories ${\mathcal A}_1, {\mathcal A}_2$ of ${\mathcal A}$, we write:
\begin{itemize}
	\item $\perp\hskip-.13cm(\mathcal A_1,\mathcal A_2)$ if $\Hom_{\mathcal A}(M,N)=0$ for every $M\in\Ob(\mathcal A_1)$, $N\in\Ob(\mathcal A_2)$;
	\item $\mathcal A_1\perp\mathcal A_2$ if one has $\perp\hskip-.13cm(\mathcal A_1,\mathcal A_2)$ and also $\perp\hskip-.13cm(\mathcal A_2,\mathcal A_1)$. In this case we say that ${\mathcal A}_1, {\mathcal A}_2$ are $\Hom$-orthogonal.
\end{itemize}

\section{$\sl(2)$-modules}

Let $\sl(2)$ be the Lie algebra of the Lie group $\operatorname{SL}(2,\C)$ and recall that  $\sl(2)$ is a simple Lie algebra. We have set $\C$ as the base field but everything admits a straightforward generalization to an arbitrary algebraically closed field of characteristic $0$. Let $\{e,f,h\}$  be a Chevalley basis of $\sl(2)$ satisfying the commutation relations:
\begin{equation}\label{eq:basis-sl2}
[e,f]=h\; , \qquad
[h,e]=2e \; , \qquad
[h,f]=-2f\; .
\end{equation}
For the purposes of this paper, it will be more convenient to consider the basis \begin{equation}
\{L_{-1}:=f,\quad L_0:=-\frac12h ,\quad  L_1:=-e\}.
\end{equation}

Every $\sl(2)$-module $V$ has a natural $\C[z]$-module structure, where $z$ acts on $V$ by $L_0$. That is, $z\cdot v:= L_0(v)$, for every $v\in V$. Let us now consider the ring automorphism $\nabla\colon \C[z]\to \C[z]$ such that $\nabla(z)=z+1$. For any $k\in\Z$, we denote by $\End^k_{\C[z]}(V)$ the $\C[z]$-module of $\nabla^k$-semilinear endomorphisms of $V$; i.e., 
\[
\End^k_{\C[z]}(V)\,:=\,
\big\{ \varphi \in \End_{\C}(V) \text{ s.t. }
\varphi (z\cdot v)=\nabla^k(z)\cdot \varphi(v)=(z+k)\cdot \varphi(v)\big\} \, . 
\] 

In particular, one has $\End^0_{\C[z]}(V)=\End_{\C[z]}(V)$. The $\C[z]$-module of $\nabla^k$-semilinear automorphisms of $V$, $\Aut^k_{\C[z]}(V)$, is defined similarly. 

Notice that an $\sl(2)$-module is just a $\C[z]$-module $V$  endowed with two $\C[z]$-semilinear endomorphisms  \begin{equation}\label{eq:semilinear}\rho(L_{-1})\in \End^{-1}_{\C[z]}(V),\quad \rho(L_1)\in \End^1_{\C[z]}(V),\end{equation} that satisfy  \begin{equation}\label{eq:commutation-relation}
 [\rho(L_{-1}),\rho(L_1)]=-2z.
\end{equation} In what follows we think of every $\sl(2)$-module in this way.  Given such a $\rho$ defined on a $\C[z]$-module $V$, we denote by $V_\rho$ the corresponding $\sl(2)$-module.

\section{Casimir modules}\label{section:Casimir-modules}

The Casimir element $C$ of the universal enveloping algebra of $\sl(2)$ can be expressed in the following equivalent ways:
    \begin{align}\label{align:casimir} 
	\begin{split}
		    C &=\frac14[(h+1)^2-1]+fe=L_0(L_0-1)-L_{-1}L_1 \,, \\
    C &=\frac14[(h-1)^2-1]+ef=L_0(L_0+1)-L_1 L_{-1}\, .
    \end{split}
    \end{align}

\begin{defn} The Casimir operator of an $\sl(2)$-module $(V,\rho)$,  is the endomorphism $C_\rho:=\rho(C)\in\End_{\sl(2)}(V).$
We say that  $(V,\rho)$ is a Casimir module of level $\mu\in\C$ if $C_\rho=\mu\Id_V$. We denote by $\mathcal C$ (resp. $\mathcal C_\mu$) the full preadditive (resp. additive) subcategory of $\sl(2)\text{-}\mathrm{Mod}$ formed by all Casimir modules (resp. of level $\mu$).
\end{defn}

Having in mind \eqref{align:casimir} and defining the polynomial:
	\[ 
	\pi_\mu(z)=z(z-1)-\mu\in\C[z]
	\, , \]
one easily checks that Casimir modules of level $\mu$ are exactly the $\C[z]$-modules endowed with $\C[z]$-semilinear endomorphisms $\rho(L_{-1})$, $\rho(L_{1})$ as in (\ref{eq:semilinear}) that verify: 
\begin{equation}\label{eq:casimir-module}
\rho(L_{-1})\circ\rho(L_1) =\pi_\mu(z), \quad \rho(L_1)\circ\rho(L_{-1}) =\pi_\mu(z+1).
\end{equation}

The commutation relation (\ref{eq:commutation-relation}) follows automatically from (\ref{eq:casimir-module})

Casimir (simple) $\sl(2)$-modules of level $\mu$ are exactly the (simple) modules over the generalized Weyl algebra:
	{\small \[
	\mathbb A(\mu)\, :=\,  U/\langle C-\mu\rangle \, =\,  U/ \langle L_{-1}L_1 -\pi_{\mu}(z) \rangle
	\, =\,  U/ \langle L_1 L_{-1} -\pi_{\mu}(z+1) \rangle
	\, , 
	\]}where $U:=U(\sl(2))$ is the universal enveloping algebra of $\sl(2)$ and $\langle x \rangle$ denotes the left ideal generated by $x$, see for instance \cite[Chapter 6]{Mazor}. Moreover, $\mathbb A(\mu)$ is a $\Z$-graded algebra:
	\[\mathbb A(\mu)=\bigoplus_{i\in\Z} \mathbb A(\mu)_i,
	\] 
with $\mathbb A(\mu)_0=\C[z]$ and $\mathbb A(\mu)_{-i}=\mathbb A(\mu)_0\cdot (L_{-1})^i$, $\mathbb A(\mu)_{i}=\mathbb A(\mu)_0\cdot (L_{1})^i$, for $i>0$.

The key role played by Casimir modules follows from the generalized version of Schur Lemma  introduced by Dixmier, see \cite[Lemma 4.1.4]{Dixmier}, since simple $\sld$-modules have countable $\C$-dimension. We get the next result,  \cite[Theorem 4.7]{Mazor}.

\begin{thm}\label{thm:Schur-Dixmier} Every simple $\sl(2)$-module is a Casimir module. Hence, one has

 $$\Simpl(\sl(2)\text{-}{\mathrm{Mod}_\mathrm{}})=\Simpl(\mathcal{C}).$$
\end{thm}

\begin{prop}\label{prop:homs-casimir} Let $(V_1,\rho_1)$,  $(V_2,\rho_2)$ be two Casimir $\sl(2)$-modules of levels   $\mu_1$ and $\mu_2$, respectively. Then one has:
\begin{enumerate}
\item If $\mu_1\neq\mu_2$, then $\Hom_{\sl(2)}(V_1,V_2)=0$. Therefore, the categories $\mathcal C_{\mu_1}$, $\mathcal C_{\mu_2}$ are $\Hom$-orthogonal and there is a natural identification $\mathcal C=\coprod_{\mu\in\C}\mathcal C_\mu$.
\item If $\mu_1=\mu_2$ and $V_2$ is torsion free, then $$\Hom_{\sl(2)}(V_1,V_2)=\{\phi\in \Hom_{\C[z]}(V_1,V_2)\colon \rho_2(L_1)\circ \phi =\phi \circ \rho_1(L_1)\}.$$
\end{enumerate}
\end{prop}

\begin{proof} (1) The first claim follows  since $\phi$ intertwines both Casimir operators. 
The second  is a consequence of the first and the definition of the coproduct of preadditive categories. (2) One implication is obvious. On the other hand, given $\phi\in \Hom_{\C[z]}(V_1,V_2)$ such that  $\rho_2(L_1)\circ \phi=\phi\circ \rho_1(L_1)$ we have \begin{align*}
&\pi_\mu(z)[\rho_2(L_{-1})\circ \phi -\phi \circ \rho_1(L_{-1})]=\\&=\rho_2(L_{-1})\circ\rho_2(L_{1})\circ[\rho_2(L_{-1})\circ \phi -\phi \circ \rho_1(L_{-1})]=\\&=\rho_2(L_{-1})\circ[\rho_2(L_{1})\circ \rho_2(L_{-1})\circ \phi -\rho_2(L_{1})\circ \phi \circ \rho_1(L_{-1})]=\\&=\rho_2(L_{-1})\circ[\rho_2(L_{1})\circ \rho_2(L_{-1})\circ \phi -\phi \circ\rho_1(L_{1})\circ  \rho_1(L_{-1})]=\\&=\rho_2(L_{-1})\circ[\pi_{\mu_2}(z) \phi -\phi \circ( \pi_{\mu_1}(z)) ]=\\&=\rho_2(L_{-1})\circ[(\pi_{\mu_2}(z)-\pi_{\mu_1}(z))\phi ].\end{align*} Since $\mu_1=\mu_2$ we get $\pi_\mu(z)[\rho_2(L_{-1})\circ \phi -\phi \circ \rho_1(L_{-1})]=0$ and bearing in mind that $V_2$ is torsion free we have $\rho_2(L_{-1})\circ \phi -\phi \circ \rho_1(L_{-1})=0$.
\end{proof}

\section{Torsion free modules}

\subsection{Polynomial modules}

\begin{defn} 
	Let $V$ be an $\sl(2)$-module. We say that $V$ is a polynomial $\sl(2)$-module or that we have a polynomial representation of $\sl(2)$ if $V$ is a finite type free  $\C[z]$-module.\end{defn}

The study of polynomial Casimir representations was one of the main achievements of \cite{Plaza-Tejero-Con}. Its relevance follows from their close relationship with simple modules as the following result shows.

\begin{thm}[{\cite[Thm. 2.10]{Plaza-Tejero-Con}}]
Every simple, torsion free, finite type  $\sl(2)$-module is a polynomial  Casimir representation of $\sl(2)$. 
Furthermore, a torsion free, finite type  $\sl(2)$-module is simple in  $\sl(2)\text{-}{\mathrm{Mod}_\mathrm{tfft}}$ if and only if it is simple in the subcategory consisting of polynomial Casimir representations.
\end{thm}

\subsection{Rational modules} Now we introduce rational $\sl(2)$-modules and describe the role that Casimir modules play in this case. 

The description given by Bavula \cite{Bavula1, Bavula2} of all simple $\sl(2)$-modules is based on the euclidean algebra $\mathbb B$ of skew Laurent polynomials over the field of rational fractions $\C(z)$ defined by the extension of the automorphism $\nabla$. Following the standard notation, we write   $\mathbb B=\C(z)[X,X^{-1};\nabla]$ whose product is determined  by the condition:
	$$X^i\cdot \xi(z)=\nabla^i(\xi(z))\cdot X^i=\xi(z+i)\cdot X^i,$$
for every $\xi(z)\in\C(z)$ and every $i\in\Z$. Notice that a $\mathbb B$-module is just a pair $(W,\sigma)$ formed by a $\C(z)$-vector space $W$ endowed with a $\nabla$-semilinear automorphism $\sigma$. 

We denote by $\mathbb B\text{-}\mathrm{Mod}$ the category of $\mathbb B$-modules and $\mathbb B\text{-}\mathrm{Mod}_{\mathrm{fd}}$ is the subcategory formed by those $\mathbb B$-modules that are finite dimensional $\C(z)$-vector spaces. For any subcategory $\mathcal B$ of $\mathbb B\text{-}\mathrm{Mod}$, we denote by $\Simpl(\mathcal B)$ the simple $\mathbb B$-modules of $\mathcal B$ and $\widehat{\Simpl}(\mathcal B)$ is the set of isomorphism classes of simple $\mathbb B$-modules.

If one considers the multiplicative subset $S=\C[z]\setminus \{0\}$, then the generalized Weyl algebra $ \mathbb A(\mu)$ embeds naturally in the localization $S^{-1}\mathbb A(\mu)$ and there is a natural identification $\mathbb B\simeq S^{-1}\mathbb A(\mu)$ providing a natural embedding of $\C$-algebras $\Phi_\mu\colon \mathbb A(\mu)\hookrightarrow\mathbb B$ that at the same time is a morphism of left $\C[z]$-modules such that $\Phi_\mu(L_{1})=X$, $\Phi_\mu(L_{-1})={\pi_\mu(z)}X^{-1}$. From now on, we identify  $\mathbb A(\mu)$ with its image under $\Phi_\mu$ inside $\mathbb B$. The natural functor of extension of scalars 
	\begin{equation}\label{eq:functorextension}
	F:=\mathbb B\otimes_{\mathbb A(\mu)}(-)\colon \mathcal C_{\mu}\longrightarrow\mathbb B\text{-}\mathrm{Mod}
	\end{equation}
sends an $\sl(2)$-module $V$ to its localization $S^{-1}V = \mathbb B\otimes_{\mathbb A(\mu)} V$ as a $\mathbb B$-module. Restricting the action of $\mathbb B$ to $\mathbb A(\mu)$ via the embedding  $\Phi_\mu\colon \mathbb A(\mu)\hookrightarrow\mathbb B$, one obtains a functor $i_\mu\colon \mathbb B\text{-}\mathrm{Mod}\to \mathcal C_\mu$ which is the right adjoint functor of $F$ and there is a functorial isomorphism $F\circ i_\mu\cong \Id_{\mathbb B\text{-}\mathrm{Mod}}$. 
\begin{rem}\label{rem:functor-extension-scalars} For any Casimir $\sl(2)$-module $V$ defined by a representation $\rho$, the $\mathbb B$-module structure of $S^{-1}V$ is given by the $\nabla$-semilinear automorphism
	\[
	S^{-1}(\rho(L_1))\colon S^{-1}V\to S^{-1}V
	\]
that is, one has $$X\cdot \left(\frac{v}{p(z)}\right):= S^{-1}(\rho(L_1))\left(\frac{v}{p(z)}\right)=\frac{\rho(L_1)(v)}{\nabla(p(z))}=\frac{\rho(L_1)(v)}{p(z+1)}.$$ 
Therefore, the functor of extension of scalars is  just $F(V,\rho)=(S^{-1}V,S^{-1}(\rho(L_1)))$, showing that $F$ is an exact functor. \end{rem}

A key property of the euclidean algebra $\mathbb B$ is that it is both a left and a right principal ideal domain, see \cite[Corollary 6.11]{Mazor}. This implies the following pleasant consequence, see \cite[Propositions 6.14, 6.15]{Mazor}.

\begin{prop}
The isomorphism classes  of simple $\mathbb B$-modules form a set $$ \widehat{\Simpl}({\mathbb B\Mod})=\left(\mathrm{Irr}(\mathbb B)/\sim\right):=\widehat{\mathrm{Irr}}(\mathbb B)$$ that gets naturally identified with  the similarity classes $\widehat{\mathrm{Irr}}(\mathbb B)$ of irreducible elements of $\mathbb B$.
\end{prop}

One has the following crucial result, see \cite{Bavula1, Bavula2} and  \cite[Thm. 6.3.6, Cor. 6.3.9]{Mazor}. 
\begin{thm}\label{thm:isom-classes} The functor of extension of scalars 
	$F\colon \mathcal C_{\mu}\longrightarrow\mathbb B\text{-}\mathrm{Mod}$
induces a bijection $$\widehat F\colon \widehat{\Simpl}({\mathcal C}_{\mu,\mathrm{tf}})\xrightarrow{\sim} \widehat{\Simpl}({\mathbb B\Mod})$$ between the  isomorphism classes  $\widehat{\Simpl}({\mathcal C}_{\mu,\mathrm{tf}})$ of  simple, torsion free, Casimir $\sl(2)$-modules of level $\mu$ and the set $\widehat{\Simpl}({\mathbb B\Mod})$ of isomorphism classes  of simple $\mathbb B$-modules. The inverse bijection maps a simple $\mathbb B$-module  to  its $\mathbb A(\mu)$-socle.
\end{thm}

Hence we have the following ``reasonable size'' property.

\begin{cor}
The categories ${\Simpl}({\mathcal C}_{\mu,\mathrm{tf}})$, ${\Simpl}({\mathbb B\Mod})$ are essentially small.
\end{cor}

We would like to extend this type of result to other invariants of the category of torsion free $\sl(2)$-modules. In order to do this we need to recall the next property.

\begin{prop}[{\cite[Lemmas 2.11, 2.12]{Plaza-Tejero-Con}}]\label{p:simple-rational}
Any simple $\mathbb B$-module is a finite dimensional vector space over the field of rational functions $\C(z)$. That is, one has $$\Simpl(\mathbb B\text{-}\mathrm{Mod)}=\Simpl(\mathbb B\text{-}\mathrm{Mod}_{\mathrm{fd}}).$$ Therefore, the category $\Simpl(\mathbb B\text{-}\mathrm{Mod}_{\mathrm{fd}})$ is essentially small.
\end{prop}

This motivates the following:

\begin{defn}\label{d:rational-repr}
We say that an $\sl(2)$-module $W$ is rational or that we have a  rational representation of $\sl(2)$ if $W$ endowed with its $\C[z]$-module structure is a finite dimensional $\C(z)$-vector space. We denote by $\mathcal R$ the abelian category of rational $\sl(2)$-modules and $\mathcal{RC}$ (resp. \,$\mathcal{RC}_\mu$) is the full preadditive subcategory formed by rational Casimir representations (resp. of level $\mu$).  
\end{defn}

Thanks to Proposition \ref{prop:homs-casimir} we have a natural identification $\mathcal{RC}=\coprod_{\mu\in\C}\mathcal{RC}_\mu$ of preadditive categories. On the other hand, Proposition~\ref{p:simple-rational} just states that the image of simple $\mathbb B$-modules under $i_\mu\colon \mathbb B\text{-}\mathrm{Mod}\hookrightarrow \mathcal C_\mu$ is contained in $\mathcal{RC}_\mu$. Combining this with Remark \ref{rem:functor-extension-scalars} and Theorem \ref{thm:isom-classes}, we get:

\begin{thm}\label{thm:simple-tf-is-of-fr} Every simple, torsion free $\sl(2)$-module has finite rank and thus $$\Simpl(\sl(2)\text{-}\mathrm{Mod}_\mathrm{tf})= \Simpl(\sl(2)\text{-}\mathrm{Mod}_{\mathrm{tffr}}).$$
\end{thm} 

Considering the extension of scalars functor, we have the following dichotomy:

\begin{cor}\label{cor:dichotomy-simple-torsion-free} A simple, torsion free $\sl(2)$-module is either a polynomial module or a non finite type $\sl(2)$-submodule of a rational module. \end{cor}

\begin{cor}  The simple $\sl(2)$-modules of the categories ${{\mathcal C}_{\mu,\mathrm{tf}}}$ and ${{\mathcal C}_{\mu,\mathrm{tffr}}}$ coincide. That is, one has $$\Simpl({{\mathcal C}_{\mu,\mathrm{tf}}})=\Simpl({{\mathcal C}_{\mu,\mathrm{tffr}}})$$ 
and therefore, 
the category ${\Simpl}({{\mathcal C}_{\mu,\mathrm{tffr}}})$ is essentially small. \end{cor}

\begin{rem}\label{rem:morphisms-rat-reps} The category of rational representations $\mathcal R$ is a full subcategory of $\sl(2)\text{-} \mathrm{Mod}_\mathrm{tffr}$. Given two rational representations $W_1$, $W_2$, one easily checks that $$\Hom_{\mathcal R}(W_1,W_2)=\Hom_{\sl(2)}(W_1,W_2)$$ is a $\C$ vector subspace of $ \Hom_{\C(z)}(W_1,W_2)$. 
\end{rem}

\begin{defn}
If $W$ is a  $\C(z)$-vector space, we denote by $\mathcal{RC}(W)$ (resp. $\mathcal{RC}_{\mu}(W)$) the set of rational Casimir $\sl(2)$-module structures (resp. of level $\mu$) such that $\rho(L_0)=z$.  
\end{defn}

Having in mind \eqref{eq:casimir-module},  one proves straightforwardly the following result.

\begin{prop}\label{prop:only-one-condition-for-rational-Casimir} 
There is a bijective correspondence between $\mathcal{RC}_{\mu}(W)$  and the set of $\C(z)$-semilinear automorphisms $\Aut^1_{\C(z)}(W),$ such that $$\rho(L_{-1}):=\pi_\mu(z) \rho(L_1)^{-1}\in \Aut^{-1}_{\C(z)}(W),\quad  \rho(L_1)\in \Aut^1_{\C(z)}(W)$$ noting that the commutation relation (\ref{eq:commutation-relation}) follows automatically.
\end{prop}

\begin{exam}\label{ex:rhomu}
For any $\mu\in\C$ one has a Casimir representation $\rho^{(\mu)}$ defined on $\C(z)$ by $$\rho^{(\mu)}(L_{-1})=\pi_\mu(z)\nabla^{-1},\quad \rho^{(\mu)}(L_{0})=z,\quad \rho^{(\mu)}(L_{1})=\nabla.$$
\end{exam}

Via the embedding $i_\mu\colon \mathbb B\text{-}\mathrm{Mod}\hookrightarrow \mathcal C_\mu$ one has that $\mathbb B$-modules are exactly the $\sl(2)$-Casimir modules of level $\mu$ that are $\C(z)$-vector spaces. In particular, those $\mathbb B$-modules that are finite dimensional $\C(z)$-vector spaces belong to $\mathcal{RC}_\mu$.  Taking into account 
(\ref{eq:casimir-module}) one has the following:

\begin{thm}\label{thm:isom-rational-N-mod} The functor $i_\mu\colon \mathbb B\text{-}\mathrm{Mod}_{\mathrm{fd}}\to\mathcal{RC}_\mu$ is an equivalence of categories whose inverse is the functor  $F\colon \mathcal{RC}_\mu\to \mathbb B\text{-}\mathrm{Mod}_{\mathrm{fd}}$ of extension of scalars; that is, $F\circ i_\mu\cong \Id_{\mathbb B\Mod}$, $ i_\mu \circ F\cong \Id_{\mathcal{RC}_\mu}$. Moreover, $F$ is naturally isomorphic to the forgetful functor   $j_\mu\colon \mathcal{RC}_\mu\to \mathbb B\text{-}\mathrm{Mod}_{\mathrm{fd}}$ defined by $j_\mu(W,\rho)=(W,\rho(L_1))$.\end{thm}

\begin{cor}\label{cor:B-simples-Racionales-simples} The isomorphism of categories $i_\mu\colon \mathbb B\text{-}\mathrm{Mod}_{\mathrm{fd}}\to\mathcal{RC}_\mu$ induces an identification $$\Simpl(\mathbb B\Mod)\xrightarrow[\sim]{i_\mu}\Simpl_{\mathcal{RC}_\mu}(\mathcal{RC}_\mu)$$ between the categories of simple objects.
\end{cor}

For any two levels   $\mu,\nu\in\C$, we define a functor $$\Phi_{\mu\nu}\colon \mathcal{RC}_\mu\to \mathcal{RC}_\nu$$ by $\Phi_{\mu\nu}=i_\nu\circ j_\mu$. Taking into account Proposition \ref{prop:only-one-condition-for-rational-Casimir} the proof of the following result is straightforward.

\begin{cor}\label{cor:isom-Casimir-mu-nu}  The categories $\mathcal{RC}_\mu$, $\mathcal{RC}_\nu$ are isomorphic through the pair of exact functors $\Phi_{\mu\nu}\colon \mathcal{RC}_\mu\to \mathcal{RC}_\nu$, $\Phi_{\nu\mu}\colon \mathcal{RC}_\nu\to \mathcal{RC}_\mu$. Moreover, given $(W,\rho)\in \mathcal{RC}_\mu$ one has $\Phi_{\mu\nu}(W,\rho)=(W,\Phi_{\mu\nu}(\rho))$ with  $$\Phi_{\mu\nu}(\rho)(L_{-1})=\frac{\pi_\nu(z)}{\pi_\mu(z)}\rho(L_{-1}),\quad \Phi_{\mu\nu}(\rho)(L_1)=\rho(L_1).$$
\end{cor}

By means of the natural embedding $\C[z]\hookrightarrow\C(z)=S^{-1}\C[z]$ we can generalize now the functor of extension of scalars to rational representations, see \eqref{eq:functorextension} and Remark \ref{rem:functor-extension-scalars}.  
\begin{defn}\label{defn:rationalization-functor}
The rationalization functor   $$F_{\mathrm{rat}}\colon \sl(2)\text{-}\mathrm{Mod}_{\mathrm{tffr}}\to \mathcal R$$  sends a finite rank $\sl(2)$-module $(V,\rho)$ to $F_{\mathrm{rat}}(V,\rho)=(S^{-1}V,\rho_\mathrm{rat})$  where the $\C(z)$ vector space $S^{-1}V$ obtained by localization of the $\C[z]$ module $V$ is endowed with the $\sl(2)$-representation $ \rho_\mathrm{rat}=S^{-1}(\rho)$ defined by 
	\begin{align*}
	\rho_\mathrm{rat}(L_{-1})\left(\frac{v}{p(z)}\right)&:=\frac{\rho(L_{-1})(v)}{p(z-1)},\\ 
	\rho_\mathrm{rat}(L_{0})\left(\frac{v}{p(z)}\right)&:=\frac{z\cdot v}{p(z)},\\
	\rho_\mathrm{rat}(L_{1})\left(\frac{v}{p(z)}\right)&:=\frac{\rho(L_{1})(v)}{p(z+1)}.\end{align*} 
\end{defn}

One straightforwardly checks that $F_{\mathrm{rat}}$ is a faithful and exact functor that by restriction induces faithful and exact functors $F_{\mathrm{rat}}\colon \mathcal C_\mathrm{tffr}\to \mathcal{RC}$, $F_{\mathrm{rat}}\colon {\mathcal C}_{\mu,\mathrm{tffr}}\to \mathcal{RC}_\mu$.

These results and Theorem \ref{thm:Schur-Dixmier} show that Theorem \ref{thm:isom-classes} can be restated  as follows:

\begin{thm}\label{thm:isom-classes-generalized} The rationalization  functor $F_{\mathrm{rat}}\colon {\mathcal C}_{\mathrm{tffr}}\longrightarrow \mathcal{RC}$
induces a bijection $$\widehat F_{\mathrm{rat}}\colon \widehat{\Simpl}({{\mathcal C}_{\mathrm{tffr}}})\xrightarrow{\sim} \widehat{\Simpl}_{\mathcal{RC}}({\mathcal{RC}}).$$ Furthermore, given $\mu\in\C$, there is a bijection 
	$$\widehat F_{\mathrm{rat}}\colon \widehat{\Simpl}({{\mathcal C}_{\mu,\mathrm{tffr}}})\xrightarrow{\sim} \widehat{\Simpl}_{\mathcal{RC}_\mu}({\mathcal{RC}_\mu}),$$ hence ${\Simpl}_{\mathcal{RC}_\mu}({\mathcal{RC}_\mu})$ is an essentially small category.
\end{thm}

\begin{rem}\label{rem:uncountable}
Notice that every rational $\sl(2)$-module is  torsion free. However, rational modules are always not simple as  $\sl(2)$-modules  since they have uncountable dimension as $\C$-vector spaces, whereas simple  $\sl(2)$-modules, which can be realized as quotients of the universal enveloping algebra $U(\sl(2))$, have at most countable $\C$-dimension due to the Poincar\'e-Birkhoff-Witt theorem.
\end{rem}

\section{Generalized Casimir modules}

\begin{defn} 
We say that an $\sl(2)$-module $(V,\rho)$ is a generalized Casimir $\sl(2)$-module of level $\mu\in \C$, if its Casimir operator $C_\rho$ fulfills $$(C_\rho-\mu)^n=0,$$ for some $n\in{\mathbb N}$. Moreover, if $(C_\rho-\mu)^n=0$ and $(C_\rho-\mu)^{n-1}\neq 0$, then we say that the generalized Casimir module $V$ has exponent $n$.

The full subcategory of $\sl(2)\Mod$ (resp. ${\mathcal R}$)  consisting of generalized Casimir $\sl(2)$-modules of level $\mu\in \C$ will be denoted by  ${\mathcal C}_\mu^\bullet$ (resp. $\mathcal{RC}_\mu^\bullet$) and ${\mathcal C}_\mu^{(n)}$ (resp. $\mathcal{RC}_\mu^{(n)}$) denotes its full subcategory formed by all modules of exponent $n$. We denote by $\mathcal C^\bullet$ (resp. $\mathcal{RC}^\bullet$) the category formed by all (resp. rational) generalized  Casimir $\sl(2)$-modules.
\end{defn}

One easily checks that generalized Casimir modules of level $\mu$  and exponent $n$ are exactly the $\C[z]$-modules  endowed with $\C[z]$-semilinear endomorphisms $\rho(L_{-1})$, $\rho(L_{1})$ and a $\C[z]$-linear $n$-th nilpotent endomorphism $N$ that verify:
\begin{equation}\label{eq:gen-casimir-module}
\rho(L_{-1})\circ\rho(L_1) =\pi_\mu(z)-N, \quad \rho(L_1)\circ\rho(L_{-1}) =\pi_\mu(z+1)-N.
\end{equation}

The commutation relation, $ [\rho(L_{-1}),\rho(L_1)]=-2z$, follows automatically from these expressions and the nilpotent endomorphism $N$ is related to the Casimir operator $C_\rho$   by  the equality $N=C_\rho-\mu$.
The next result follows straightforwardly.
\begin{prop}\label{prop:filtration}
	Every generalized Casimir $\sl(2)$-module $(V,\rho)$ of level $\mu$ and exponent $n$,  has a canonical strictly increasing filtration of length $n$ by $\sl(2)$-submodules:
	\begin{equation}\label{e:nat-filtration}
	V^0:= (0) \, \subsetneq \, V^1 \, \subsetneq \, \ldots \, \subsetneq \, V^n:=V
	\end{equation}
	where  $V^i=\ker(C_\rho-\mu)^i$, for $i=1,\ldots, n$,  is a generalized Casimir $\sl(2)$-module of level $\mu$ and exponent $i$ and $V^{i}/ V^{i-1}$ is a Casimir $\sl(2)$-module of level $\mu$.
	
	Moreover, if $V$  is a finite rank torsion free   module, then $V^i$ and $V^{i}/ V^{i-1}$ are also torsion free modules and for every $i=1,\ldots,n-1$, one has  $$\rk\left(V^i/V^{i-1}\right)\geq \rk\left(V^{i+1}/V^{i}\right).$$  In particular, if $V$ is a rational module, then $V^i$ and $V^{i}/ V^{i-1}$ are rational modules.  \end{prop}

\begin{prop}\label{prop:orthogonality-generalized-casimirs} Let us consider $\mu_1,\mu_2\in\C$. If  $\mu_1\neq\mu_2$, then for any $n_1,n_2\in \mathbb N$ the categories $\mathcal  C_{\mu_1}^{(n_1)}, \mathcal  C_{\mu_2}^{(n_2)}$ are $\Hom$-orthogonal.
\end{prop}

\begin{proof}
Let us suppose that $\mu_1\neq \mu_2$. For any $V_i\in {\mathcal C}_{\mu_i}^{(n_i)}$ we have to check that $\Hom_{\sl(2)}(V_1,V_2)=(0)$. Let  $V_i^\bullet$ be the filtration of $V_i$ described in Proposition \ref{prop:filtration}. By the first part of Proposition \ref{prop:homs-casimir}  it follows  that $\Hom_{\sl(2)}(V_{1}^1,V_{2}^1)=0$.  Proceeding by induction on the terms of the filtration $V_1^\bullet$ one easily shows by considering its successive quotients  that $\Hom_{\sl(2)}(V_1,V_2^1)=0$.   By means of the natural filtration $V_2^\bullet$  of $V_2$ one shows now in a similar way that $\Hom_{\sl(2)}(V_1,V_2)=0$.
\end{proof}

%%%%%%%%%%%%%%%%%%%%%%%%%%%%%%
\section{Rational $\sl(2)$-modules}
%%%%%%%%%%%%%%%%%%%%%%%%%%%%

\subsection{$\C$-rationality of Casimir operators and endomorphisms} Let us recall that a $\C$-linear endomorphism of a complex vector space of infinite dimension does not have in general a minimal polynomial. Therefore, if we take into account that a rational module has complex dimension equal to $\mathfrak c$ -the cardinality of the continuum- it is remarkable, and also crucial for obtaining further results, that the Casimir operators of rational $\sl(2)$-modules always have a minimal polynomial.  

\begin{thm}\label{t:caracCasC}
If $(W,\rho)$ is a rational $\sl(2)$-representation, then the minimal polynomial  $M_{C_\rho}(t)$ of its Casimir operator  $C_\rho$, considered as a $\C(z)$-linear endomorphism of $W$, has its coefficients in $\C$; that is, $M_{C_\rho}(t)\in \C[t]$. Therefore, the Casimir operator $C_\rho$ understood as a $\C$-linear endomorphism of $W$ has minimal polynomial  $M_{C_\rho}(t)$.
\end{thm}

\begin{proof} Notice that the field automorphisms $\nabla$ and $\nabla^{-1}$ of $\C(z)$ induce in a natural way automorphisms of the polynomial ring $\C(z)[t]$ such that given $P= a_n t^n+\cdots+a_1 t+a_0\in\C(z)[t]$ one has $\nabla(P):=\nabla(a_n) t^n+\cdots+\nabla(a_1) t+\nabla(a_0)$ and a similar formula holds for $\nabla^{-1}(P)$. The minimal polynomial of $C_\rho$ is the monic generator of its annihilating ideal $\mathrm{Ann}(C_\rho)\subset \C(z)[t]$. One has 
\begin{equation}\label{eq:clave}
\rho(L_1)\circ\rho(L_{-1})\circ \nabla(M_{C_\rho})(C_\rho)=\rho(L_1)\circ M_{C_\rho}(C_\rho)\circ\rho(L_{-1})=0.
\end{equation}
 
By (\ref{align:casimir}) we have $\rho(L_1)\circ\rho(L_{-1})=z(z+1)-C_\rho$, plugging this into (\ref{eq:clave}) we get $$\left( z(z+1)-C_\rho \right)\circ \nabla(M_{C_\rho})(C_\rho)=0.$$ This shows that
 the polynomial $\left(t-z(z+1)\right)\cdot \nabla(M_{C_\rho})(t)$ belongs to $\mathrm{Ann}({C_\rho})$. Since $M_{C_\rho}(t)$ is the monic generator of  this ideal and   $\nabla(M_{C_\rho})(t)$ is also monic of the same degree, there exists $\xi\in\C(z)$ such that we have the equality \begin{equation}\label{eq:pol-eq1}
\left(t-z(z+1)\right)\cdot \nabla(M_{C_\rho})(t)=(t-\xi)\cdot M_{C_\rho}(t) .\end{equation}
Proceeding in a similar way with $\rho(L_{-1})\circ\rho(L_1)\circ \nabla(M_{C_\rho})(C_\rho)$ one proves that there exists $\zeta\in\C(z)$ satisfying the identity
\begin{equation*}
\left(t-z(z-1)\right)\cdot \nabla^{-1}(M_{C_\rho})(t)=(t-\zeta)\cdot M_{C_\rho}(t). \end{equation*} Applying to it the automorphism $\nabla$ we get \begin{equation}\label{eq:pol-eq2}
\left(t-z(z+1)\right)\cdot M_{C_\rho}(t)=(t-\nabla(\zeta))\cdot \nabla(M_{C_\rho})(t). \end{equation} Now, equations (\ref{eq:pol-eq1}) and (\ref{eq:pol-eq2}) straightforwardly imply that $$\nabla(M_{C_\rho})(t)=M_{C_\rho}(t)$$ and thus the claim follows.
\end{proof}

\begin{cor}\label{c:L1invertible} If $(W,\rho)$ is a rational representation, then $\rho(L_1)$ and $\rho(L_{-1})$  are $\C(z)$-semilinear automorphisms of $W$. 
\end{cor}
  
\begin{proof}
Bearing in mind the Cayley-Hamilton theorem, it follows that Theorem~\ref{t:caracCasC} implies that all eigenvalues of $C_\rho$ belong to $\C$, hence the $\C(z)$-endomorphisms  $C_\rho-z(z-1)$, $C_\rho-z(z+1)$  are invertible. By the first equality of (\ref{align:casimir}) one has $C_\rho-z(z-1)=-\rho(L_{-1})\circ \rho(L_1)$, thus $\rho(L_1)$ is injective and $\rho(L_{-1})$ is surjective. Similarly, the second equality of   (\ref{align:casimir})  gives $C_\rho-z(z+1)=-\rho(L_1)\circ \rho(L_{-1})$ and thus  $\rho(L_1)$ is surjective and $\rho(L_{-1})$ is injective.
\end{proof}

\begin{prop}\label{p:finite+extension} 
Let $(W,\rho)$ be a rational $\sl(2)$-representation. For every $T\in \End_{\sl(2)}(W)$  the minimal polynomial  $M_{T}(t)$ of $T$ considered as a $\C(z)$-linear endomorphism of $W$ has its coefficients in $\C$; that is, $M_{T}(t)\in \C[t]$. Therefore,  $T$ understood as a $\C$-linear endomorphism of $W$ has minimal polynomial  $M_{T}(t)$.
\end{prop}
	
\begin{proof} One has $\rho(L_1)\circ \nabla^{-1}(M_T)(T)=M_T(T)\circ \rho(L_1)=0.$ Since $\rho(L_1)$ is invertible by Corollary \ref{c:L1invertible}, it follows that $\nabla^{-1}(M_T)(t)$ belongs to the annihilating ideal $\mathrm{Ann}(T)\subset\C(z)[t]$ of $T$ as a $\C(z)$-linear map.  Taking into account that $\mathrm{Ann}(T)$ is generated by the monic polynomial $M_{T}(t)$ and that $\deg(\nabla^{-1}(M_T)(t))=\deg(M_T(t))$, we conclude $\nabla^{-1}(M_T)(t)=M_{T}(t)$ and the proof is finished.
\end{proof}

\subsection{Categorical properties of rational modules}

\begin{prop}\label{p:abelian}
	${\mathcal R}$ is a $\C$-linear abelian subcategory of the category  $\C(z)\text{-}\mathrm{Vect}_\mathrm{fd}$ of finite dimensional $\C(z)$-vector spaces.
\end{prop}

\begin{proof}
	Thanks to  Remark \ref{rem:morphisms-rat-reps}, it follows that ${\mathcal R}$ is a $\C$-linear subcategory of the category of finite dimensional $\C(z)$-vector spaces, which is an abelian category. It is straightforward to check that ${\mathcal R}$ is an abelian category. Moreover, since every exact sequence of rational $\sl(2)$-modules is an exact sequence of $\C(z)$-vector spaces, the claim is proved. \end{proof}
 
\begin{thm}\label{thm:Jordan-Holder-Krull-Schmidt} $\mathcal R$ is a finite-length abelian category and thus every rational module has rational Jordan-H\"older filtrations. Moreover,  $\mathcal R$ is a Krull-Schmidt category and therefore every rational module has a Remak decomposition into a finite direct sum of indecomposable rational modules, unique up to reordering of the direct summands.  
\end{thm}

\begin{proof} Since $\mathcal R$ is an abelian subcategory of $\C(z)\text{-}\mathrm{Vect}_\mathrm{fd}$, it follows that every rational module $(W,\rho)$ is both Artinian and Noetherian. Hence $\mathcal R$ is a finite-length abelian category. By \cite[Lemma 5.1 and Theorem 5.5]{Krause} it follows that $\mathcal R$ is Krull-Schmidt. The existence  of a Remak decomposition, which is essentially the Krull-Schmidt Theorem, follows from \cite[Theorem 1]{Atiyah} taking into account  \cite[Lemma 5.1]{Krause}.
\end{proof}

\begin{defn}
For a rational representation $W$, we define its length, $\length(W)$, as the maximal length of all filtrations of $W$  by rational subrepresentations; that is, by subobjects of $W$ in ${\mathcal R}$. 
\end{defn}

It is clear that for any rational $\sl(2)$-module $W$ it holds that:
	\[ 
	\length(W)\, \leq \, \dim_{\C(z)} W.
	\]
Now we prove that $\mathcal R$ has a ``reasonable size'' to have a Grothendieck group. 
\begin{thm}\label{thm:rational-essentially-small} $\mathcal R$ is an essentially small category. Therefore, $\mathcal{RC}^\bullet$, $\mathcal{RC}$, $\mathcal{RC}_\mu^\bullet$ and $\mathcal{RC}_\mu$, for every $\mu\in \C$, are also essentially small categories.
\end{thm}

\begin{proof} Let $(W,\rho)$ be an $\sl(2)$-rational representation such that $\dim_{\C(z)}W=m$. Choosing an isomorphism of $\C(z)$-vector spaces  $\varphi\colon W\xrightarrow{\sim} \C(z)^m$ one has that $(W,\rho)$ is $\mathcal R$-isomorphic to $(\C(z)^m,\rho')$ where $\rho'=\varphi\circ\rho\circ\varphi^{-1}$.  One has $$\rho'(L_1)=\sigma(z)\circ\nabla,\quad \rho'(L_0)=z,\quad \rho'(L_{-1})=\tau(z)\circ\nabla^{-1}$$ for certain $\sigma(z),\tau(z)\in M(n,\C(z))$ such that (\ref{eq:commutation-relation}) holds. One has an identification $$\mathcal R(\C(z)^m)\simeq\{(\sigma(z),\tau(z))\in  M(n,\C(z))^2\colon \sigma(z)\circ\tau(z+1)-\tau(z)\circ\sigma(z-1)=2z\Id_n\}.$$ By Remark \ref{rem:morphisms-rat-reps} it follows that $T(z)\in GL(n,\C(z))$ acts on $\mathcal R(\C(z)^m)$ by conjugation. In terms of the previous identification this action is written $$T(z)\bullet(\sigma(z),\tau(z))=(T(z)\circ\sigma(z)\circ T(z+1)^{-1}, T(z)\circ\tau(z)\circ T(z-1)^{-1}).$$ Thus, the isomorphism classes of $\sl(2)$-representations on $\C(z)^m$ is given by the set $$\Iso(\mathcal R(\C(z)^m))= \mathcal R(\C(z)^m)/GL(m,\C(z)).$$
Finally, the isomorphism classes of rational representations in given by the set $$\Iso(\mathcal R)=\coprod_{m\geq 0}\Iso(\mathcal R(\C(z)^m).$$Therefore $\mathcal R$ is essentially small. The other claims follow immediately since all the categories considered are full subcategories of $\mathcal R$.
\end{proof}

\subsection{Decomposition of the category of rational modules}

\begin{thm}\label{thm:finite+decomp}
	The abelian category ${\mathcal R}$ of rational $\sl(2)$-modules decomposes into the $\Hom$-orthogonal direct sum of the abelian subcategories of generalized rational Casimir modules:
		\[
		{\mathcal R}\,=\, \bigoplus_{\mu\in \C } {\mathcal{RC}}^\bullet_{\mu}\, . 
		\] This is compatible with the coproduct decomposition $\mathcal{RC}\,=\, \coprod_{\mu\in \C} {\mathcal{RC}}_{\mu}\, .$
\end{thm}

\begin{proof}
	Let $(W,\rho)$ be a rational $\sl(2)$-module. Theorem~\ref{t:caracCasC} implies that the minimal  $M_{C_\rho}(t)$ polynomial of its Casimir operator $C_\rho$ as a $\C(z)$-endomorphism of $W$  space splits over $\C$; that is,
	\[
	M_{C_\rho}(t)\,=\, (t-\mu_1)^{n_1}\cdot \ldots \cdot (t-\mu_r)^{n_r}
	\]
	(where $\mu_i\neq \mu_j$ for $i\neq j$) and, accordingly, $W$ decomposes as a $\C(z)$-vector space as follows:
	\[
	W\,= \, \ker(C_\rho-\mu_1)^{n_1}\oplus \cdots \oplus \ker(C_\rho-\mu_r)^{n_r}
	\, . \]
	
	Having in mind that $C_\rho$ commutes with $\rho(L_i)$ for $i=-1,0,1$, the previous expression is also a decomposition as $\sl(2)$-modules, where  $\ker(C_\rho-\mu_i)^{n_i}$ belongs to $ {\mathcal R}_{\mu_i}^{(n_i)}$. The orthogonality of the decomposition follows from Proposition \ref{prop:orthogonality-generalized-casimirs}. Thus $\mathcal R$ is the direct sum of the abelian subcategories $\mathcal{RC}^\bullet_\mu$, see  \cite[page 3]{Etingof}.
\end{proof}

As a formal consequence we get the:

\begin{cor} For every $\mu\in\C$, the abelian category $\mathcal{RC}_\mu^{\bullet}$ is closed under extensions in $\Ra$.
\end{cor}

\subsection{Indecomposable and irreducible rational modules}

Let us recall that a rational $\sl(2)$-module $(W,\rho)$ is $\sl(2)$-inde\-com\-po\-sable or simply indecomposable, (resp. ${\mathcal R}$-indecomposable) if it can not be written as $W=W_1\oplus W_2$ where $W_1$, $W_2$ are  $\sl(2)$-modules (resp. rational $\sl(2)$-modules).   

\begin{prop}\label{prop:indecomposable-generalized-Casimir} A rational $\sl(2)$-module is indecomposable if and only if it is $\mathcal R$-indecomposable. Moreover, every indecomposable rational module is a generalized Casimir module. Hence, we have $\Ind_{\mathcal R}(\mathcal R)=\Ind(\mathcal R)=\Ind_{\mathcal{RC}^\bullet}(\mathcal{RC}^\bullet)=\Ind(\mathcal{RC}^\bullet)$.
\end{prop}

\begin{proof} For the first statement, one implication is obvious. On the other hand, let $W$ be a rational $\sl(2)$-module that is  $\mathcal R$-indecomposable and suppose that there is a decomposition $W=V_1\oplus V_2$ as $\sl(2)$-modules. It follows that $V_1$, $V_2$ are $\C[z]$-submodules. Given $v_1\in V_1$ and $\xi=\frac{p(z)}{q(z)}\in\C(z)$, since $W$ is a $\C(z)$-vector space one has 
	\[
	\xi\cdot (v_1,0) =(u_1,u_2),
	\]
for certain $u_1\in V_1$, $u_2\in V_2$. On the other hand 
	\[ 
	q(z)\cdot(\xi\cdot v)=(q(z)\cdot u_1,q(z)\cdot u_2)=
	(q(z)\cdot\xi)\cdot (v_1,0)=(p(z)\cdot v_1,0).
	\]
Therefore, $q(z)\cdot u_2=0$ and this implies $u_2=0$. Hence, $\xi\cdot v_1\in V_1$ and so $V_1$ is a $\C(z)$-vector space. In a similar way one shows that $V_2$ is also a $\C(z)$-vector subspace. Hence, $W$ would be $\mathcal R$-decomposable, giving a contradiction. Therefore, $W$ is necessarily $\sl(2)$-indecomposable. This finishes the proof of the first claim. The second one is an immediate consequence of Theorem \ref{thm:finite+decomp}.
\end{proof}

\begin{prop}\label{prop:indec-endomorphism-iso-or-nilpotent} If $(W,\rho)$ is an indecomposable rational $\sl(2)$-module, then every $\phi\in \End_{\sl(2)}(W)$ is either an isomorphism or nilpotent.
\end{prop}

\begin{proof} By Proposition \ref{p:finite+extension} any $\phi\in \End_{\sl(2)}(W)$ has minimal polynomial $M_\phi(t)$ in $\C[t]$. Since $\phi$ commutes with $\rho(L_1)$ and $\rho(L_{-1})$, the ideas used in the proof of Theorem \ref{thm:finite+decomp} show  that $(W,\rho)$ would be decomposable unless $M_\phi(t)$ has a unique root, say $\mu\in\C$. Hence  $M_\phi(t)=(t-\mu)^m$ for certain integer $m\geq 1$. This implies $(\phi-\mu)^m=0$ and thus $\phi=\mu+ N$ with $N$ an $m$-th nilpotent endomorphism, proving the claim.
\end{proof}

In a similar way let us recall that a rational module is $\sl(2)$-simple or just simple,  (resp. ${\mathcal R}$-simple) if it has no non-trivial $\sl(2)$-submodules (resp. rational $\sl(2)$-sub\-mo\-dules). In particular, $W$ is $\mathcal R$-simple if and only if $\length(W)=1$. Notice that  one has the following implications: $$\mathcal{R}\text{-simple}\Longrightarrow \mathcal{R}\text{-indecomposable}\Longleftrightarrow \sl(2)\text{-indecomposable}.$$

\begin{prop}\label{prop:R-irred-rational-is-Casimir} An $\mathcal R$-simple, rational $\sl(2)$-module is a Ca\-si\-mir module. Hence, one has $$\Spl_\Ra(\Ra)=\Spl_{\Ra\mathcal C^\bullet}(\Ra\mathcal C^\bullet)=\Spl_{\Ra\mathcal C}(\Ra\mathcal C),\quad \Spl_{\Ra\mathcal C_\mu^\bullet}({\Ra\mathcal C}_\mu^\bullet)=\Spl_{\Ra\mathcal C_\mu}(\Ra\mathcal C_\mu).$$
\end{prop}

\begin{proof} A rational $\sl(2)$-module that is $\mathcal R$-irreducible is $\mathcal R$-inde\-com\-po\-sable, and therefore is a generalized Casimir module by Proposition \ref{prop:indecomposable-generalized-Casimir}. Furthermore, its exponent has to be $1$ since otherwise by Proposition \ref{prop:filtration} it would have a non trivial filtration and therefore it could not be $\mathcal R$-simple.
\end{proof}

\begin{thm}\label{thm:identification-isomorphism-classes} One has the following identifications $$\widehat{\Simpl}({\sl(2)\mathrm{\Mod}_\mathrm{tf}})=\widehat{\Simpl}({\sl(2)\mathrm{\Mod}_\mathrm{tffr}})\xrightarrow[\sim]{\widehat F_\mathrm{rat}}\widehat{\Simpl}_{\mathcal{RC}}({\mathcal{RC}})=\widehat{\Simpl}_{\mathcal R}({\mathcal R}).$$
\end{thm}

 \begin{proof} The first one is due to Theorems  \ref{thm:Schur-Dixmier} and \ref{thm:simple-tf-is-of-fr}, whereas the second one is given by Theorem \ref{thm:isom-classes-generalized}. The  last equality follows from Proposition \ref{prop:R-irred-rational-is-Casimir}.
\end{proof} 

In particular, it is clear that every rational $\sl(2)$-module that is one dimensional over $\C(z)$ is $\Ra$-simple. Therefore we get the:

\begin{cor}\label{cor:rational-onedim-are-irred-Casimir} Every rational $\sl(2)$-module that is one dimensional over $\C(z)$ is an $\Ra$-simple Casimir  $sl(2)$-module.
\end{cor}

\subsection{Action of automorphisms on rational Casimir representations}

Recalling Proposition~\ref{prop:only-one-condition-for-rational-Casimir} and Corollary~\ref{c:L1invertible}, one proves the following fact. 

\begin{prop}\label{prop:Aut1-sl2}
	Given a $\C(z)$-vector space $W$, there is an identification:
		\[\Aut^{1}_{\C(z)}(W)\longleftrightarrow \mathcal{RC}_\mu(W)
		\, ,
		\]
	defined by sending $\varphi\in \Aut^{1}_{\C(z)}(W)$ to the $\sl(2)$-representation $\rho$  defined on $W$ by $\rho(L_{-1}):=\pi_\mu(z)\, \varphi^{-1}\in \Aut^{-1}_{\C(z)}(W)$, $\rho(L_1):=\varphi\in \Aut^{1}_{\C(z)}(W)$. 
\end{prop}

More generally, given a rational representation $(W,\rho)$ and an automorphism $\varphi\in \Aut_{\C(z)}(W)$ we  define a map  $\varphi\cdot\rho\colon\sl(2)\to\End_\C(W)$ such that  \begin{align}\label{eq:action-automorphism-on-rpresentation}\varphi\cdot\rho(L_{-1})&=\rho(L_{-1})\circ\varphi^{-1}\in\Aut_{\C(z)}^{-1}(W),\nonumber\\
\varphi\cdot\rho(L_0)&=z,\\
 \varphi\cdot\rho(L_{1})&=\varphi\circ \rho(L_{1})\in\Aut_{\C(z)}^1(W).\nonumber
\end{align}
One checks that $\varphi\cdot\rho$ is a rational $\sl(2)$-representation if and only if $$\varphi\circ C_\rho=C_\rho\circ\varphi.$$ Therefore, the group $\Aut_{\C(z)}(W;C_\rho)=\{\varphi\in\Aut_{\C(z)}(W)\colon [\varphi,C_\rho]=0\}$, which naturally contains $\C^\times(z)$,  acts on $\rho\in\mathcal R(W)$. This action preserves the Casimir operator, that is $C_{\varphi\cdot\rho}=C_\rho$. In what follows we denote $\varphi\cdot(W,\rho):=(W,\varphi\cdot\rho)$. 

In particular, one has:

\begin{prop}\label{prop:paramtrization-Casimir-reps} $\Aut_{\C(z)}(W)$ acts freely on $\mathcal{RC}(W)$ and its orbits are $\mathcal{RC}_\mu(W)$. Thus, $\Aut_{\C(z)}(W)$ acts freely and transitively on  $\mathcal{RC}_\mu(W)$, for every $\mu\in\C$.
\end{prop}
%%%%%%%%%%%%%%%%%%%%%%%%%%%%%%%%%%%%%%%%%%

\subsection{Hom-finiteness of the category of rational modules}

Let $\Bbbk$ be a commutative ring. Recall that an additive category $\mathcal A$ is $\Hom_\Bbbk$-finite if it is $\Bbbk$-linear and $\Hom_{\mathcal A}(A,B)$ is a $\Bbbk$-module of finite length for all objects $A,B$.

\begin{prop}\label{prop:end=C}
Let $W$ be a rational $\sl(2)$-representation. If $W$ is $\mathcal R$-simple, then $\End_{\sl(2)}(W)=\C$. 
\end{prop}

\begin{proof}
Let $\phi\in \End_{\sl(2)}(W)$. By Proposition~\ref{p:finite+extension}, the minimal polynomial of $\phi$, $M_{\phi}(t)$, has its coefficients in $\C$. Considering a root $\alpha$ of 
	$M_{\phi}(t)$, the space of eigenvectors $\ker(\phi-\alpha)$ is a non-zero rational subrepresentation of $W$. Hence, the hypothesis implies that  $W=\ker(\phi-\alpha)$; that is, $\phi=\alpha\in \C$. 
\end{proof}

\begin{rem} Let us point out that Proposition \ref{prop:end=C} can not be obtained by more conventional means. On the one hand, Dixmier's generalization, \cite[Lemma 4.1.4]{Dixmier}, of Schur's Lemma can not be applied since rational $\sl(2)$-modules have $\C$-dimension equal to $\mathfrak c$. On the other hand, Quillen's Lemma, \cite{Quillen}, can not be applied since, by Remark \ref{rem:uncountable}, rational $\sl(2)$-modules are never simple as $\sl(2)$-modules.
\end{rem}

\begin{prop}\label{prop:hom-irreds-is-one-dimensional} If  $(W_1,\rho_1)$,  $(W_2,\rho_2)$ are two $\mathcal R$-irreducible rational $\sl(2)$-repre\-sen\-ta\-tions, then every element of $\Hom_\sld(W_1,W_2)$ is either zero or an isomorphism. Moreover, if $W_1$ is not isomorphic to $W_2$, then $\Hom_\sld(W_1,W_2)=0$,  whereas $\dim_\C\Hom_\sld(W_1,W_2)=1$ if $W_1\simeq W_2$.
\end{prop}

\begin{proof} Let $\phi\in \Hom_\sld(W_1,W_2)$ be a non-zero homomorphism. Since $\ker\phi$ is a rational subrepresentation of $W_1$, we conclude that $\phi$ is injective. On the other hand, since $\operatorname{Im}\phi\subset W_2$ is a rational subrepresentation, $\phi$ is surjective. Hence, $\phi$ is an isomorphism.  Let us consider an arbitrary homomorphism $\phi' \in \Hom_{\sl(2)}(W_1, W_2)$.  Then, $\phi'\circ \phi^{-1} \in \End_{\sl(2)}(W_2)$ and, by Proposition~\ref{prop:end=C}, we obtain that $\phi'\circ \phi^{-1}\in \C$ and thus $ \Hom_{\sl(2)}(W_1, W_2)$ is of dimension $1$ over $\C$. 
\end{proof}

\begin{thm}\label{t:hom-finite}
	Let $(W_1,\rho_1)$,  $(W_2,\rho_2)$ be two rational $\sl(2)$-repre\-sen\-ta\-tions. Then, 
	\[
	\dim_{\C} \Hom_{\sl(2)}(W_1, W_2)\, \leq \, \length W_1\cdot \length  W_2
	\, . \] Therefore, the category $\mathcal R$ of rational $\sl(2)$-modules is $\Hom_\C$-finite.
\end{thm}
	
\begin{proof}
We proceed by double induction on $\length  W_1$ and $\length  W_2$. If we have  $\length  W_1 =  \length  W_2=1$, then the claim follows from Proposition \ref{prop:hom-irreds-is-one-dimensional}. 
Suppose now  that $\length  W_1 =1 $ and $ \length  W_2>1$ and choose a non-trivial rational subrepresentation of $W_2$, say $W'_2$. Consider the exact sequence:
	\[
	\Hom_{\sl(2)}(W_1, W'_2) \,\longrightarrow\, 
	\Hom_{\sl(2)}(W_1, W_2)\,\longrightarrow\, 
	\Hom_{\sl(2)}(W_1, W_2/W'_2)\, .
	\]
Due to the induction hypothesis, the statement holds for the terms on the left and on the right. Bearing in mind  how the dimension varies with the above exact sequences, the claim holds for the central term. 

Finally, if  $\length  W_1 >1 $ and $ \length  W_2>1$, we consider  a non-trivial rational subrepresentation of $W_1$, say $W'_1$, and the exact sequence:
	\[
	\Hom_{\sl(2)}(W_1/W'_1, W_2) \,\longrightarrow\, 
	\Hom_{\sl(2)}(W_1, W_2)\,\longrightarrow\, 
	\Hom_{\sl(2)}(W'_1, W_2)
	\, .\]
A similar argument as above yields the result.
\end{proof}

%%%%%%%%%%%%%%%%%%%%%%%%%%%%%%%%%%%%%
\subsection{Extensions of rational $\sl(2)$-representations}

Let $(W',\rho')$ and $(W'',\rho'')$ be rational $\sl(2)$-representations. Any rational $\sl(2)$-extension $(W,\rho)$ of $(W'',\rho'')$ by $(W',\rho')$ gives rise to a short exact sequence of rational representations $$0\to W'\to W\to W''\to 0,$$ that splits as $\C(z)$-vector spaces, thus $W=W'\oplus W''$ and we can write $$\rho(L_{-1})=\begin{pmatrix} \rho'(L_{-1}) & B_{-1}\\ 0 & \rho''(L_{-1})
\end{pmatrix},\quad \rho(L_1)=\begin{pmatrix} \rho'(L_{1}) & B_{1}\\ 0 & \rho''(L_{1})
\end{pmatrix}$$ where $B_{-1}\in \Hom_{\C(z)}^{-1}(W'',W'), B_{1}\in \Hom_{\C(z)}^{1}(W'',W')$. The representation condition for $\rho$ is equivalent to the equality of $\C(z)$-linear maps in $\Hom_{\C(z)}(W'',W')$ \begin{equation}\label{eq:def-T} T:=\rho'(L_{-1})\circ B_1+B_{-1}\circ \rho''(L_1)=\rho'(L_1)\circ B_{-1}+B_1\circ \rho''(L_{-1}).\end{equation} This implies the identity \begin{multline}\label{eq:identidad-T}
\rho'(L_1)\circ T-T\circ\rho''(L_1)=\rho'(L_1)\circ \rho'(L_{-1})\circ B_1-B_1\circ \rho''(L_{-1})\circ\rho''(L_1)=\\=\frac{1}{4}(B_1\circ C_{\rho''}-C_\rho\circ B_1),\end{multline} where $C_{\rho'}$, $C_{\rho''}$ are the Casimir operators of $(W',\rho')$ and  $(W'',\rho'')$, respectively. Conversely, given any $T\in\Hom_{\C(z)}(W'',W')$ that satisfies (\ref{eq:identidad-T}), since $\rho''(L_{-1})$ is invertible by Corollary \ref{c:L1invertible},  we can define $B_{-1}:=[T-\rho'(L_{-1})\circ B_1]\circ\rho''(L_1)^{-1}$, then $B_{-1}, B_1 $ satisfy (\ref{eq:def-T}). Therefore, the space ${\mathcal E}\!{xt}^1_{\mathcal R}(W''_{\rho''},W'_{\rho'})$ of data for rational extensions  of the rational $\sl(2)$-module $W''_{\rho''}$ by the rational $\sl(2)$-module $W'_{\rho'}$ is given by \begin{align*}{\mathcal E}\!{xt}^1_{\mathcal R}(W''_{\rho''},W'_{\rho'})=\{(&B_1,T)\in \Hom^1_{\C(z)}(W'',W')\times\Hom_{\C(z)}(W'',W')\colon\\ &\rho'(L_1)\circ T-T\circ\rho''(L_1)=\frac{1}{4}(B_1\circ C_{\rho''}-C_{\rho'}\circ B_1)\}.\end{align*} We denote the extension data $(B_1,T)$ by $(\rho',\rho''; B_1, T)$ when we want to express explicitly the underlying representations involved. Moreover,  $W(\rho',\rho''; B_1,T)\in \mathcal R(W)$ denotes the representation defined by $(\rho',\rho''; B_1, T)\in {\mathcal E}\!{xt}^1_{\mathcal R}(W''_{\rho''},W'_{\rho'})$. There is a natural injective map $$\Ext\colon {\mathcal E}\!{xt}^1_{\mathcal R}(W''_{\rho''},W'_{\rho'}) \hookrightarrow \mathcal R(W)$$ defined by  $\Ext(\rho',\rho''; B_1, T):=W(\rho',\rho''; B_1,T)$.

In particular, we get:

\begin{prop}\label{prop:extension-data-for-rational-Casimir-same-level} If $W'_{\rho'}, W''_{\rho''}$ belong to $\mathcal{RC}_\mu$, then  one has $${\mathcal E}\!{xt}^1_{\mathcal R}(W''_{\rho''},W'_{\rho'})=\Hom^1_{\C(z)}(W'',W')\times\Hom_{\sl(2)}(W''_{\rho''},W'_{\rho'})$$and ${\mathcal E}\!{xt}^1_{\mathcal R}(W''_{\rho''},W'_{\rho'})={\mathcal E}\!{xt}^1_{\mathcal{RC}^\bullet_\mu}(W''_{\rho''},W'_{\rho'})\subset \mathcal {RC}^{(2)}_\mu(W)$.
\end{prop}

To determine the isomorphism classes of extensions we have to find those morphisms  $\phi\in \Hom_{\sl(2)}( W(\rho',\rho''; B_1^{(1)},T^{(1)}), W(\rho',\rho''; B_1^{(2)},T^{(2)}))$, such that the diagram 
	\[ \\
	\xymatrix{
		0\ar[r] &  W' \ar@{=}[d]  \ar[r] & W(\rho',\rho''; B_1^{(1)},T^{(1)}) \ar[d]^{\phi} \ar[r] &  W'' \ar@{=}[d]
		\ar[r] &  0 
		\\
		0\ar[r] &  W'  \ar[r] & W(\rho',\rho''; B_1^{(2)},T^{(2)}) \ar[r] &  W'' 
				\ar[r] &  0	}
	\]
is commutative. Therefore $$\phi=\begin{pmatrix} \Id_{W'} &\alpha\\ 0 & \Id_{W''}
\end{pmatrix}$$ with $\alpha\in\Hom_{\C(z)}(W'',W')$ and one checks that the condition that $\phi$ is a morphism of $\sl(2)$-modules is equivalent to the equalities \begin{align*}B_{1}^{(2)}-B_1^{(1)}&=\alpha\circ \rho''(L_1)-\rho'(L_1)\circ\alpha,\\ 
T^{(2)}-T^{(1)} &=\frac{1}{4}(C_{\rho'}\circ \alpha-\alpha\circ C_{\rho''}).
\end{align*} Equivalently, $\Hom_{\C(z)}(W'',W')$ acts on ${\mathcal E}\!{xt}_{\mathcal R}(W''_{\rho''},W'_{\rho'})$ and the $\C$-vector space $\Ext^1_{\mathcal R}(W''_{\rho''},W'_{\rho'})$ of extensions  of $W''_{\rho''}$ by $W'_{\rho'}$ is the quotient under this action $$\Ext^1_{\mathcal R}(W''_{\rho''},W'_{\rho'})={\mathcal E}\!{xt}^1_{\mathcal R}(W''_{\rho''},W'_{\rho'})/\Hom_{\C(z)}(W'',W').$$

Taking into account the previous results and Theorem \ref{thm:finite+decomp}, we obtain for rational Casimir representations the following description of the extension groups.  

\begin{prop}\label{prop:extension-group-rational-Casimir} Let  $W'_{\rho'}, W''_{\rho''}$ be two rational Casimir representations of levels   $\mu',\mu''\in\C$, respectively. \begin{enumerate} \item If $\mu'\neq \mu''$, then $\Ext^1_{\mathcal R}(W''_{\rho''},W'_{\rho'})=0.$
\item If  $\mu'= \mu''$, then  one has \begin{align*}
\Ext^1_{\mathcal R}(&W''_{\rho''},W'_{\rho'})=\\&=\left(\Hom^1_{\C(z)}(W'',W')/\Hom_{\C(z)}(W'',W')\right)\times\Hom_{\sl(2)}(W''_{\rho''},W'_{\rho'}),\end{align*}where $\alpha\in \Hom_{\C(z)}(W'',W')$ acts on $B_1\in \Hom^1_{\C(z)}(W'',W')$ by $$\alpha\cdot B_1:= B_1+\alpha\circ \rho''(L_1)-\rho'(L_1)\circ\alpha.$$
\end{enumerate}
\end{prop}

\begin{cor}  If $W'_{\rho'}, W''_{\rho''}$ are two rational Casimir representations of level $\mu$, then $W(\rho',\rho''; B_1,T)$ is a Casimir representation, necessary of level $\mu$, if and only if $T=0$. Therefore, the category $\mathcal{RC}_\mu$ is not closed under extensions in $\mathcal R$. 
\end{cor}

%%%%%%%%%%%%%%%%%%%%%%%%%%%%%%%%%%%%%%%%%%

\section{Finite rank torsion free $\sl(2)$-modules}

Notice that the abelian category of rational representations $\mathcal R$ is a full subcategory of the category $\sl(2)\Mod_\mathrm{tffr}$ of finite rank torsion free $\sl(2)$-modules. In this section we aim at studying what properties of $\mathcal R$ are valid in the larger category  $\sl(2)\Mod_\mathrm{tffr}$.

\subsection{Minimal polynomials for Casimir operators and endomorphisms} Since finite rank torsion free $\sl(2)$-modules always have infinite dimension as $\C$-vector spaces, the following result is not obvious.

\begin{thm}\label{thm:minimalCasC}
If $(V,\rho)$ is a finite rank torsion free $\sl(2)$-module, then its Casimir operator  $C_\rho$ considered as a $\C$-linear endomorphism has a minimal polynomial and it coincides with the minimal polynomial of the Casimir operator $C_{\rho_\mathrm{rat}}$ of its rationalization. 
\end{thm}

\begin{proof} The rationalization of $(V,\rho)$ gives an injection of  $\sl(2)$-modules $(V,\rho)\hookrightarrow (S^{-1}V,\rho_\mathrm{rat})$. Therefore, one has $C_\rho={C_{\rho_\mathrm{rat}}}_{|V}$ and so the result follows from Theorem \ref{t:caracCasC}.
\end{proof}

\begin{cor}\label{cor:L1injective} If $(V,\rho)$ is a finite rank torsion free $\sl(2)$-module, then $\rho(L_1)$ and $\rho(L_{-1})$  are injective $\C[z]$-semilinear endomorphisms of $V$ and their images are essential $\C[z]$-submodules. 
\end{cor}

\begin{proof}
By Corollary~\ref{c:L1invertible} we know that $S^{-1}(\rho(L_1))$, $S^{-1}(\rho(L_{-1}))$ are $\C(z)$-semilinear automorphisms of $S^{-1}V$. This implies the injectivity of $\rho(L_1)$ and $\rho(L_{-1})$  and also that their images are essential $\C[z]$-submodules.  \end{proof}

Taking into account Proposition \ref{p:finite+extension} and using the same ideas as in the proof of Theorem \ref{thm:minimalCasC} we get the following result.

\begin{prop}\label{p:finite+extension-frtf} 
Let $(V,\rho)$ be a finite rank torsion free $\sl(2)$-re\-pre\-sen\-ta\-tion. Every $\sl(2)$-endomorphism $\phi\in \End_{\sl(2)}(V)$ considered as a $\C$-linear endomorphism has a minimal polynomial and it coincides with the minimal polynomial of its rationalization $S^{-1}(\phi)\in \End_{\sl(2)}(F_\mathrm{rat}(V,\rho))$.
\end{prop}

\subsection{Categorical properties of finite rank torsion free modules}

We begin by analyzing the exact structure.

\begin{prop}\label{prop:exact-tffr} The natural inclusion $\sl(2)\Mod_{\mathrm{tffr}}\hookrightarrow \sl(2)\Mod$
makes the category $\sl(2)\Mod_{\mathrm{tffr}}$ into a $\C$-linear exact category. The abelian category $\mathcal R$ is a thick subcategory of the exact category $\sld\Mod_\tffr$.
\end{prop}

\begin{proof} Since  $\sl(2)\Mod_{\mathrm{tffr}}$ is a full additive subcategory of the abelian category $\sl(2)\Mod$, according to \cite[Lemma 10.20]{B} we get the first claim if $\sl(2)\Mod_{\mathrm{tffr}}$ is closed under extensions in $\sl(2)\Mod$. This follows from the exactness of the rationalization functor and the left exactness of the torsion subfunctor. To prove the second claim we must show that if  $0\to V'\to V\to V''\to 0$ is exact in $\sl(2)\Mod_{\mathrm{tffr}}$, then $V\in \mathcal R$ if and only if $V',V''\in\mathcal R$. Since $V\in\mathcal R$ if and only if the underlying $\C[z]$-module structure on $V$ is such that for every $\lambda\in\C$ multiplication by $z-\lambda$  is an isomorphism, the result follows immediately by  the snake Lemma.
\end{proof}

\begin{rem}\label{rem:finite-length-missed}
The main difference with rational representations is that $\sl(2)\Mod_{\mathrm{tffr}}$ is not a finite length category. This is so because although every rational module $W$ has finite $\Ra$-length, it never has finite length as an object of $\sl(2)\Mod_{\mathrm{tffr}}$, because if this were true then $W$ would have countable $\C$-dimension, which  is impossible.
\end{rem}

Following the ideas of Remark \ref{rem:functor-extension-scalars} one gets the following result.

\begin{prop}\label{prop:rationalization-is-a-retraction} Let $F_\rat\colon \sld\Mod_\tffr\to\Ra$ be the rationalization functor. If $i_\tffr\colon \Ra\hookrightarrow \sld\Mod_\tffr$ is the natural embedding, then there is an isomorphism of functors $F_\rat \circ i_\tffr\cong \Id_{\Ra}$ and thus  $F_\rat$ is a retraction of $i_\tffr$. 
\end{prop}

Now we prove a result that later will be crucial for studying the structure of the Grothedieck group of the category $\sl(2)\text{-} \mathrm{Mod}_\mathrm{tffr}$.

\begin{thm}\label{thm:rational-is-reflective-localization-of-tffr} The embedding functor $i_\tffr\colon\mathcal R\hookrightarrow \sl(2)\text{-} \mathrm{Mod}_\mathrm{tffr}$ is a right adjoint of the rationalization functor $F_{\mathrm{rat}}\colon \sl(2)\text{-}\mathrm{Mod}_{\mathrm{tffr}}\to \mathcal R$. Therefore, $\mathcal R$  is a reflective localization of $\sl(2)\text{-} \mathrm{Mod}_\mathrm{tffr}$ with localization functor $F_{\mathrm{rat}}$. Moreover, $F_{\mathrm{rat}}$ is faithful.
\end{thm}

\begin{proof} One checks the first claim straightforwardly. The second  follows from the first one thanks to \cite[Proposition 1.3, pag. 7 ]{GZ} since  $i_\tffr\colon\mathcal R\hookrightarrow \sl(2)\text{-} \mathrm{Mod}_\mathrm{tffr}$ is fully faithful.  Finally, it is well known that $F_{\mathrm{rat}}$ is faithful precisely if the component of the unit of the adjunction $\eta$ over a finite rank torsion free $\sl(2)$-module $V$ is a monomorphism. This is immediate since $\eta_V$ is the natural inclusion $V\hookrightarrow S^{-1} V$.
\end{proof}

We end this section proving that $\sl(2)\text{-} \mathrm{Mod}_\mathrm{tffr}$ has a ``suitable size'' for $K_0$-groups.

\begin{thm}\label{thm:tffr-is-essentially-small}
The category $\sl(2)\text{-} \mathrm{Mod}_\mathrm{tffr}$ is essentially small and thus $\mathcal{C}^\bullet_\tffr$, $\mathcal{C}_\tffr$, $\mathcal{C}_{\mu,\tffr}^\bullet$ and $\mathcal{C}_{\mu,\tffr}$, for every $\mu\in \C$, are also essentially small categories.
\end{thm}

\begin{proof} The rationalization functor  induces a map $F_\rat\colon \Iso(\sld\Mod_\tffr)\to\Iso(\Ra)$.
By Theorem \ref{thm:rational-essentially-small} the category $\mathcal R$ is essentially small. Therefore, to prove the claim it is enough to show that given $\{W\}\in \Iso(\Ra)$,  the isomorphism classes $\{V\}$  of finite rank torsion free $\sl(2)$-modules such that $F_\rat(V)\simeq W$ form a set. Without loss of generality we may assume that $V$ is an $\sl(2)$-submodule of $W$. Now the condition $F_\rat(V)= W$ is equivalent to saying that $V$ has rank  $r=\rk(W)$. It is well known, \cite[10.6.2, pag. 92]{Schubert}, that $\sld\Mod$, as the category of modules over the enveloping algebra $U(\sld)$, is well powered and therefore there is a set $\Sub(W)$ of  $\sl(2)$-submodules of $W$. The collection $\Sub_r(W)$ formed by those submodules of $W$ whose rank  is $r$ is a subset of $\Sub(W)$, finishing the proof.
\end{proof}

\subsection{Decomposition of the category of finite rank torsion free modules}

Taking into account Theorem \ref{thm:minimalCasC}, the following result is proved in the same way as Theorem \ref{thm:finite+decomp}.

\begin{thm}\label{thm:finite+decomp-tffr}
	The exact category $\sl(2)\Mod_\mathrm{tffr}$ decomposes into the $\Hom$-orthogonal direct sum of the exact subcategories of generalized Casimir modules:
		\[
		\sl(2)\Mod_\mathrm{tffr}\,=\, \bigoplus_{\mu\in \C} {\mathcal{C}}^\bullet_{\mu,\mathrm{tffr}}\, . 
		\] This is compatible with the coproduct decomposition $\mathcal C_\mathrm{tffr}\,=\, \coprod_{\mu\in \C } {\mathcal{C}}_{\mu,\mathrm{tffr}}.$
\end{thm}

\subsection{Indecomposable and irreducible finite rank torsion free modules}

\begin{prop}\label{prop:irred-tffr-is-Casimir} 
Let $V$ be an object of $\sl(2)\Mod_\mathrm{tffr}$. If $V$ is
	\begin{enumerate}
		\item indecomposable, then $V\in\mathcal C^\bullet_\tffr$; that is, $\Ind(\sld\Mod_\tffr)=\Ind(\mathcal C^\bullet_\tffr)$. Moreover, every $\phi\in \End_{\sl(2)}(V)$ is either an isomorphism or nilpotent.;
		\item simple, then $V\in\mathcal C_\tffr$; that is, $\Spl(\sld\Mod_\tffr)=\Spl(\mathcal C_\tffr)$.
	\end{enumerate}
\end{prop}

\begin{proof}
$(1)$ The first part follows from Theorem \ref{thm:finite+decomp-tffr}. For the second one, having in mind Proposition \ref{p:finite+extension-frtf}, one may argue  as in Proposition \ref{prop:indec-endomorphism-iso-or-nilpotent}. $(2)$ follows from Theorem \ref{thm:Schur-Dixmier}. Alternatively, a finite rank torsion free module that is simple is also indecomposable. Therefore, recalling the proof of Proposition \ref{prop:R-irred-rational-is-Casimir}, we are done. 
\end{proof}

\subsection{Hom-finiteness and Krull-Schmidt property for finite rank torsion free modules}

Bearing in mind Proposition~\ref{p:finite+extension-frtf},  Proposition \ref{prop:end=C} gives:

\begin{prop}\label{prop:end=C-tffr}
Let $(V,\rho)$ be a finite rank torsion free $\sl(2)$-module. If $(V,\rho)$ is irreducible, then $\End_{\sl(2)}(V)=\C$. 
\end{prop}

This result and the same ideas used in the proof of Proposition \ref{prop:hom-irreds-is-one-dimensional} give the following:

\begin{prop}\label{prop:hom-irreds-is-one-dimensional-tffr} If  $(V_1,\rho_1)$,  $(V_2,\rho_2)$ are two irreducible finite rank torsion free $\sl(2)$-modules, then every element of $\Hom_\sld(V_1,V_2)$ is either zero or an isomorphism. Moreover, if $V_1$ is not isomorphic to $V_2$, then $\Hom_\sld(V_1,V_2)=0$,  whereas $\dim_\C\Hom_\sld(V_1,V_2)=1$ if $V_1\simeq V_2$.
\end{prop}

By Theorem \ref{thm:rational-is-reflective-localization-of-tffr}, $F_\rat$ is faithful,  this together with Theorem \ref{t:hom-finite}, yields:

\begin{thm}\label{t:hom-finite-frtf}
	Let $(V_1,\rho_1)$,  $(V_2,\rho_2)$ be two finite rank torsion free $\sl(2)$-modules. Then, 
	\[
	\dim_{\C} \Hom_{\sl(2)}(V_1, V_2)\, \leq \, \rank(V_1)\cdot \rank(V_2)
	\, . \] Therefore, the category $\sl(2)\Mod_{\mathrm{tffr}}$ of finite rank torsion free $\sl(2)$-modules is $\Hom_\C$-finite.
\end{thm}

This in turn implies: 

\begin{cor}\label{cor:tffr-is-KS} The category $\sl(2)\Mod_{\mathrm{tffr}}$ of finite rank torsion free $\sl(2)$-modules is a Krull-Schmidt category.
\end{cor}

\begin{proof}
If $\Bbbk$ is a field, then a $\Hom_\Bbbk$-finite exact category is a Krull-Schmidt category, this follows from  \cite[Corollary pag. 310, Theorem 1 pag. 313]{Atiyah}. Therefore, the claim follows from Theorem \ref{t:hom-finite-frtf}.
\end{proof}

%%%%%%%%%%%%%%%%%%%%%%%%%%%%%%%%%%
\subsection{Purely irreducible modules}
%%%%%%%%%%%%%%%%%%%%%%%%%%%%

Let us recall the following:

\begin{defn}
Given a finite rank torsion free $\C[z]$-module $V$, one says that a $\C[z]$-submodule $V'$ is: \begin{enumerate}
\item a pure submodule if  $V/V'$ is a torsion free $\C[z]$-module,
\item an essential submodule if $V/V'$ is a torsion $\C[z]$-module.
\end{enumerate}
\end{defn}

The following is well known.

\begin{prop} If $V$ is a finite rank torsion free $\C[z]$-module and $V'\subset V$ is a $\C[z]$-submodule, then $$P(V'):= (S^{-1}V')\cap V$$ is a pure submodule of $V$  and $V'$ is an essential submodule of $P(V')$. Moreover, $P(V')=V'$ if and only if $V'$ is a pure submodule and $P(V')=V$ if and only if $V'$ is an essential submodule.
\end{prop}

\begin{defn} Let $V$ be a finite rank torsion free $\C[z]$-module. For any $\C[z]$-submodule $V'\subset V$, the pure submodule $P(V')$ is called the purification of $V'$ in $V$.
\end{defn}

\begin{prop}\label{prop:rationall-irred-is-purely-irred}
If $(W,\rho)$ is an $\mathcal R$-irreducible rational $\sl(2)$-module, then it has no proper pure $\sl(2)$-submodules. Therefore, every proper $\sl(2)$-submodule of $W$ is an essential submodule.
\end{prop}

\begin{proof} Let $V\hookrightarrow W$ be an $\sl(2)$-submodule. Then there is a commutative diagram of short exact sequences $$\xymatrix{& 0\ar[d] & 0\ar[d] & & \\0\ar[r] & V\ar[r]\ar[d] & W\ar[r]\ar[d] & W/V\ar[r]\ar[d] & 0\\ 0\ar[r] & F_\mathrm{rat}(V)\ar[r] & F_\mathrm{rat}(W)\ar[r] & F_\mathrm{rat}(W/V)\ar[r] & 0}$$
Since, $F_\mathrm{rat}(W)\simeq W$ and $W$ is $\mathcal R$-irreducible, we get either $F_\mathrm{rat}(V)=0$ or $F_\mathrm{rat}(V)=F_\mathrm{rat}(W)$. In the first case we have $V=0$, whereas in the second one  $F_\mathrm{rat}(W/V)=0$ and this is equivalent to say that $W/V$ is either zero or a torsion $\C[z]$-module.
\end{proof} 

\begin{prop}\label{prop:rat-irred-only-essential-submodules} 
Let $(V,\rho)$ be a finite rank torsion free $\sl(2)$-module. $F_\mathrm{rat}(V,\rho)$ is an $\mathcal R$-irreducible rational-module if and only if $V$ has no proper pure $\sl(2)$-submodules, if and only if  every proper $\sl(2)$-submodule of $V$ is an essential submodule. \end{prop}

\begin{proof} Let $V'\hookrightarrow V$ be a proper $\sl(2)$-submodule. Then there is a commutative diagram of short exact sequences $$\xymatrix@R=16pt@C=30pt{& 0\ar[d] & 0\ar[d] & & \\0\ar[r] & V'\ar[r]\ar[d] & V\ar[r]\ar[d] & V/V'\ar[r]\ar[d] & 0\\ 0\ar[r] & F_\mathrm{rat}(V')\ar[r] & F_\mathrm{rat}(V)\ar[r] & F_\mathrm{rat}(V/V')\ar[r] & 0}$$
If $F_\mathrm{rat}(V)$ is $\mathcal R$-irreducible, then we get that either $F_\mathrm{rat}(V')=0$ or $F_\mathrm{rat}(V')=F_\mathrm{rat}(V)$. In the first case we would have $V'=0$, whereas in the second one $F_\mathrm{rat}(V/V')=0$ and this is equivalent to say that $V/V'$ is either $(0)$ or a torsion $\C[z]$-module. Since $V'$ is proper, the unique possibility left is that $V/V'$ is torsion, that is $V'$ is essential. 

In a similar way, let $W\subset F_\mathrm{rat}(V)$ be a proper rational $\sl(2)$-submodule. Then there is a commutative diagram of short exact sequences $$\xymatrix@R=16pt@C=30pt{& 0\ar[d] & 0\ar[d] & & \\0\ar[r] & W\cap V\ar[r]\ar[d] & V\ar[r]\ar[d] & V/W\cap V\ar[r]\ar[d] & 0\\ 0\ar[r] & W\ar[r] & F_\mathrm{rat}(V)\ar[r] & F_\mathrm{rat}(V/W\cap V)\ar[r] & 0}$$ Hence $\rank(W\cap V)\leq\rank W<\rank(V)$
and thus $V/W\cap V$ is not a torsion module. This implies that $W\cap V$ is a pure $\sl(2)$-submodule of $V$ and therefore either $W\cap V=0$ or $W\cap V=V$. In the first case the commutative diagram gives $W=0$ whereas in the second $W=F_\mathrm{rat}(V)$. This finishes the proof.
\end{proof} 

\begin{thm}\label{thm:char-rat-irred} Let $(V,\rho)$ be a finite rank torsion free $\sl(2)$-module.   If $(V,\rho)$ is $\sl(2)$-irreducible then its rationalization $F_\mathrm{rat}(V,\rho)$ is an $\mathcal R$-irreducible rational module. 

Conversely, if $F_\mathrm{rat}(V,\rho)$ is an $\mathcal R$-irreducible rational module then either $(V,\rho)$ is $\sl(2)$-irreducible or it has  a unique $\sl(2)$-irreducible essential $\sl(2)$-submodule $(V',\rho')\hookrightarrow(V,\rho)$. 
\end{thm}

\begin{proof} If $(V,\rho)$ is an irreducible finite rank torsion free $\sl(2)$-module, then it has ho proper $\sl(2)$-submodules and therefore  $F_\mathrm{rat}(V,\rho)$ is $\mathcal R$-irreducible by Proposition \ref{prop:rat-irred-only-essential-submodules}. On the other hand, if $F_\mathrm{rat}(V,\rho)$ is an $\mathcal R$-irreducible  rational module, then it follows from \cite[Lemma 6.26]{Mazor} that it has an $\sl(2)$-irreducible $\sl(2)$-submodule $(V',\rho')\hookrightarrow F_\mathrm{rat}(V,\rho)$. At the same time we have an inclusion of $\sl(2)$-modules $(V,\rho)\hookrightarrow F_\mathrm{rat}(V,\rho)$, hence we get an inclusion of $\sl(2)$-modules $$(V',\rho')\cap(V,\rho)\hookrightarrow (V',\rho')$$ and since $(V',\rho')$ is $\sl(2)$ irreducible we must have either $V'\cap V=0$ or $V'\cap V=V'$. However, the first possibility never occurs because $V'\subset S^{-1}V$ always intersects $V$ non trivially. Therefore, there is an injection of $\sl(2)$-modules $(V',\rho')\hookrightarrow (V,\rho)$ and by Proposition \ref{prop:rat-irred-only-essential-submodules} this is either an equality or an essential submodule. If there were two $\sl(2)$-irreducible essential $\sl(2)$-submodules $(V'_1,\rho'_1)$, $(V'_2,\rho'_2)$ of  $(V,\rho)$, then either one has $V'_1\cap V'_2=0$ or $V_1'=V_2'$. However, the first possibility never occurs since $S^{-1}V_1'=S^{-1}V_2'=S^{-1}V$.\end{proof}

Taking into account the previous results we can give the following:

\begin{defn} We say that an  $\sl(2)$-module $(V,\rho)$ is purely irreducible if it is a finite rank torsion free module and it satisfies the equivalent conditions:
\begin{enumerate}
\item it has no proper pure $\sl(2)$-submodules, 
\item its proper $\sl(2)$-submodules are essential, 
\item $F_\mathrm{rat}(V,\rho)$ is $\mathcal R$-irreducible.
\end{enumerate}
 We denote by $\mathrm{PSpl}(\sld)$ the category formed by all purely irreducible $\sld$-modules.
\end{defn}

\begin{rem} Thanks to Theorem \ref{thm:char-rat-irred}, every simple $\sl(2)$-module is purely irreducible. Therefore, there is an inclusion of categories $\Spl(\sld\Mod)\hookrightarrow\mathrm{PSpl}(\sl(2))$.

\end{rem}

\begin{defn}  If $(V,\rho)$ is a purely irreducible $\sl(2)$-module, let $(V'\rho')$ be its unique $\sl(2)$-irreducible essential $\sl(2)$-submodule and  let $i_{(V,\rho)}\colon(V',\rho)\hookrightarrow (V,\rho)$ be the natural inclusion. We say that the subobject of $(V,\rho)$ defined by the pair $(i_{(V,\rho)},(V',\rho'))$ is the  type of $(V,\rho)$ and we denote it $\tau(V,\rho)$.  The injection $i_{(V,\rho)}$ is called the structural morphism of the purely irreducible module $(V,\rho)$ and when there is no danger of confusion we identify $\tau(V,\rho)$ with the irreducible $\sl(2)$-module $(V' , \rho')$. \end{defn}

\begin{prop} A purely irreducible module is indecomposable.
\end{prop}

\begin{proof} Let $(V,\rho)$ be a purely irreducible $\sl(2)$-module. Suppose there exists two finite rank torsion free $\sl(2)$ modules $(V_1,\rho_1)$, $(V_2,\rho_2)$ such that $(V,\rho)=(V_1,\rho_1)\oplus(V_2,\rho_2)$ then $$F_\rat(V,\rho)=F_\rat(V_1,\rho_1)\oplus F_\rat(V_2,\rho_2).$$ Since $F_\rat(V,\rho)$ is $\mathcal R$-irreducible, we have either $F_\rat(V_1,\rho_1)=0$ or $F_\rat(V_2,\rho_2)=0$ and this implies that either $V_1=0$ or $V_2=0$, proving the claim.
\end{proof}

\begin{prop}\label{prop:rank-1-tf-is-p-irred} Every torsion free $\sl(2)$-module of rank 1 is purely irreducible. \end{prop}

\begin{proof} Every proper submodule of a rank 1 torsion free $\sl(2)$-module has rank 1 and therefore is an essential submodule. 
\end{proof}

We also have an equivalence for rational modules.

\begin{prop}\label{prop:rational-R-irreducible-iff-purely-irreducible} A rational module is $\mathcal R$-irreducible if and only if it is purely irreducible.
\end{prop}

\begin{proof}
The direct implication follows from  Proposition \ref{prop:rationall-irred-is-purely-irred}. On the other hand if $W\in\Ra$ is purely irreducible and it had a proper rational submodule $W'\subset W$, then $W'$ would be essential and thus $\dim_{\C(z)}W'=\dim_{\C(z)}W$. This would imply $W'=W$ which is impossible.
\end{proof}

\begin{prop} Let $(V',\rho')$ be an irreducible finite rank torsion free $\sl(2)$-module. An $\sl(2)$-module $(V,\rho)$ is a purely irreducible module of type $V'$ if and only if $V$ is an $\sl(2)$-submodule of $F_\mathrm{rat}(V',\rho')$ that contains $V'$; that is, if and only if one has a chain of $\sl(2)$-modules $(V',\rho')\subseteq (V,\rho)\subseteq F_\mathrm{rat}(V',\rho')$.
\end{prop}

\begin{proof} The direct implication follows from Theorem \ref{thm:char-rat-irred}. On the other hand, if $(V,\rho)$ is an $\sl(2)$-module and we have a chain of $\sl(2)$-modules $(V',\rho')\subseteq (V,\rho)$$\subseteq F_\mathrm{rat}(V',\rho')$, it is clear that $(V,\rho)$ is a finite rank torsion free $\sl(2)$-module.  Moreover, one has  $$F_\mathrm{rat}(V',\rho')\subseteq F_\mathrm{rat}(V,\rho)\subseteq F_\mathrm{rat}(V',\rho'),$$ hence $F_\mathrm{rat}(V',\rho')=F_\mathrm{rat}(V,\rho)$ and therefore $(V,\rho)$ is purely irreducible.
\end{proof}

\begin{prop}\label{prop:purely-irreducible-is-Casimir} If $(V,\rho)$ is a purely irreducible $\sl(2)$-module, then it is a Casimir module and  $\End_\sld((V,\rho))=\C$. 
\end{prop}

\begin{proof} Let $(V',\rho')$ be the type of $(V,\rho)$. The chain of $\sl(2)$-modules $(V',\rho')\subseteq (V,\rho)\subseteq F_\mathrm{rat}(V',\rho')$ gives that the Casimir operators satisfy $C_\rho={C_{\rho'_\rat}}_{|V}$ and so the first claim follows from Proposition \ref{prop:irred-tffr-is-Casimir}. 

On the other hand, since $F_\rat$ is faithful by Theorem \ref{thm:rational-is-reflective-localization-of-tffr}, the equality $F_\mathrm{rat}(V',\rho')=F_\mathrm{rat}(V,\rho)$ obtained by applying the rationalization functor yields an injection $$F_\rat\colon \End_\sld((V,\rho))\hookrightarrow \End_\sld(F_\mathrm{rat}(V,\rho))=\End_\sld(F_\mathrm{rat}(V',\rho'))$$ and since $\End_\sld(F_\mathrm{rat}((V',\rho'))=\C$, due to Proposition \ref{prop:end=C-tffr}, we conclude that $\End_\sld((V,\rho))=\C$.
\end{proof}

\begin{rem}\label{rem:decopmposition-purely-irred-with-irred-Casimir-tffr} 
As a consequence of Theorem \ref{thm:finite+decomp-tffr} and Proposition \ref{prop:purely-irreducible-is-Casimir}, one has that the category $\mathrm{PSpl}(\sld)$ decomposes as follows  $$\mathrm{PSpl}(\sld)=\, \coprod_{\mu\in \C }\mathrm{PSpl}(\sld)_\mu,$$ where $\mathrm{PSpl}(\sld)_\mu$ is the full subcategory of $\mathrm{PSpl}(\sld)$ formed by those purely irreducible $\sld$-modules that are Casimir modules in $\mathcal C_{\mu,\tffr}$.
This decomposition is compatible with the one given in Theorem \ref{thm:finite+decomp-tffr} and there is a commutative diagram $$\xymatrix{\mathrm{PSpl}(\sld)\ar@{=}[r] &\coprod_{\mu\in \C}\mathrm{PSpl}(\sld)_\mu\\ \mathrm{Spl}({\mathcal C}_\tffr) \ar@{=}[r]\ar@{_{(}->}[u] & \coprod_{\mu\in \C} \mathrm{Spl}(\mathcal C_{\mu,\tffr})\ar@{_{(}->}[u] }$$
\end{rem}

\begin{prop}\label{prop:hom-purely-irreds-is-one-dimensional} If  $(V_1,\rho_1)$,  $(V_2,\rho_2)$ are two purely irreducible  $\sl(2)$-modules, then every element of $\Hom_\sld(V_1,V_2)$ is either zero or injective. Moreover, if $\tau(V_1,\rho_1)\not\simeq \tau(V_2,\rho_2)$, then $\Hom_\sld(V_1,V_2)=0$,  whereas if $\tau(V_1,\rho_1)\simeq \tau(V_2,\rho_2)$ and $\Hom_\sld(V_1,V_2)\neq 0$,  the rationalization morphism $$ F_\rat\colon \Hom_{\sl(2)}(V_1,V_2)\xrightarrow{\sim}\Hom_{\sl(2)}(F_\mathrm{rat}(V_1,\rho_1),F_\mathrm{rat}(V_2,\rho_2))$$ is an isomorphism of one dimensional vector spaces.
\end{prop}

\begin{proof} By Theorem \ref{thm:rational-is-reflective-localization-of-tffr} the functor $F_\rat$ is faithful, hence  the rationalization morphism is injective.  Since the rationalizations of $V_1$ and $V_2$ are 
$\mathcal R$-irreducible, from Proposition \ref{prop:hom-irreds-is-one-dimensional}  it follows that for any  $\phi\in \Hom_\sld(V_1,V_2)$ its rationalization $F_\rat(\phi)\in \Hom_{\sl(2)}(F_\mathrm{rat}(V_1),F_\mathrm{rat}(V_2))\simeq\C$ is either zero or an isomorphism. In the first case $\phi=0$, whereas in the second $\phi$ is injective. If $\tau(V_1,\rho_1)\not\simeq \tau(V_2,\rho_2)$, then $F_\mathrm{rat}(V_1)\not\simeq F_\mathrm{rat}(V_2)$ and  bearing in mind Proposition \ref{prop:hom-irreds-is-one-dimensional} we get $ \Hom_{\sl(2)}(F_\mathrm{rat}(V_1,\rho_1),F_\mathrm{rat}(V_2,\rho_2))=0$, thus $\Hom_\sld((V_1,\rho_1),(V_2,\rho_2))=0$. On the other hand, if   $\tau(V_1,\rho_1)\simeq \tau(V_2,\rho_2)$, then $F_\mathrm{rat}(V_1,\rho_1)\simeq F_\mathrm{rat}(V_2,\rho_2)$ and again by Proposition \ref{prop:hom-irreds-is-one-dimensional} we get 
	\[
	\Hom_{\sl(2)}(F_\mathrm{rat}(V_1,\rho_1),F_\mathrm{rat}(V_2,\rho_2))\simeq \C\, . 
	\]
Hence, if $\Hom_\sld((V_1,\rho_1),(V_2,\rho_2))\neq 0$, the rationalization morphism is also an isomorphism. \end{proof}

Therefore, we give the following: 

\begin{defn} Given two purely irreducible  $\sl(2)$-modules  $(V_1,\rho_1)$,  $(V_2,\rho_2)$ we write $$(V_1,\rho_1)\leq(V_2,\rho_2)$$ if $\tau(V_1,\rho_1)= \tau(V_2,\rho_2)$ and there is an injective morphism $V_1\hookrightarrow V_2$ of $\sl(2)$-modules.
\end{defn}

It is straightforward to check that this defines a partial order relation among the objects of the category $\mathrm{PSpl}(\sl(2))_\mu$ of purely irreducible modules of level $\mu$. For any $(V,\rho)\in\mathrm{PSpl}(\sl(2))_\mu$  we always have $$\tau(V,\rho)\leq (V,\rho)\leq F_\mathrm{rat}(V,\rho).$$ 

It also induces an order relation on the full subcategory  $\mathrm{PSpl}(\sl(2);\tau)_\mu$ (resp. $\mathrm{PSpl}(\sl(2);\hat\tau)_\mu$) of $\mathrm{PSpl}(\sl(2))_\mu$ formed by those purely irreducible modules of level $\mu$ with fixed type an irreducible Casimir module $\tau\in\mathrm{Spl}(\mathcal C_{\mu,\mathrm{tffr}})$ (resp. fixed isomorphic type defined by a class $\hat\tau\in\widehat{\mathrm{Spl}}(\mathcal C_{\mu,\mathrm{tffr}})$. 

\begin{prop}\label{prop:decomposition-purely-irred-types} One has the following decomposition  compatible with the order relation $$\mathrm{PSpl}(\sl(2))_\mu=\coprod_{\hat\tau\in\widehat{\mathrm{Spl}}(\mathcal C_{\mu,\mathrm{tffr}})} \mathrm{PSpl}(\sl(2);\hat\tau)_\mu.$$ This in turn gives rise to a decomposition of isomorphism classes $$ \widehat{\mathrm{PSpl}}(\sl(2))_\mu=\coprod_{\hat\tau\in\widehat{\mathrm{Spl}}(\mathcal C_{\mu,\mathrm{tffr}})} \widehat{\mathrm{PSpl}}(\sl(2);\hat\tau)_\mu.$$
\end{prop}

Notice that any isomorphism of $\sl(2)$-modules $\phi\colon\tau_1\simeq\tau_2$, induces an isomorphism of categories $\Phi\colon\mathrm{PSpl}(\sl(2);\tau_1)_\mu\to \mathrm{PSpl}(\sl(2);\tau_2)_\mu$ that maps $(V_1,\rho_1)\in \mathrm{PSpl}(\sl(2);\tau_1)_\mu$ to the same $\sld$-module $(V_1,\rho_1)$ but now endowed with the structural map $i_{(V_1,\rho_1)}\circ\phi^{-1}$. The inverse functor $\Phi^{-1}\colon \mathrm{PSpl}(\sl(2);\tau_2)_\mu\to \mathrm{PSpl}(\sl(2);\tau_1)_\mu$ maps $(V_2,\rho_2)\in \mathrm{PSpl}(\sl(2);\tau_2)$ to $(V_1,\rho_1)$ with structural map $i_{(V_2,\rho_2)}\circ\phi$. It is obvious that the functors $\Phi$ and $\Phi^{-1}$ are monotonic, that is they are compatible with the order relations.

\begin{prop} Let $(V_1,\rho_1)$,  $(V_2,\rho_2)$ be two purely irreducible  $\sl(2)$-modules. One has $(V_1,\rho_1)\leq(V_2,\rho_2)$ and $(V_2,\rho_2)\leq(V_1,\rho_1)$ if and only if  $(V_1,\rho_1)\simeq(V_2,\rho_2)$.
\end{prop}

\begin{proof} If $(V_1,\rho_1)\leq(V_2,\rho_2)$ and $(V_2,\rho_2)\leq(V_1,\rho_1)$, then there are injective morphisms of $\sl(2)$-modules $i\colon V_1\hookrightarrow V_2$, $j\colon V_2\hookrightarrow V_1$. Then $j\circ i\in \End_\sld(V_1)$ is injective and therefore by Proposition \ref{prop:purely-irreducible-is-Casimir} there exists $z_1\in\C^\times$ such that $j\circ i=z_1\Id_{V_1}$. In a similar way, there exists $z_2\in\C^\times$ such that $i\circ j=z_2\Id_{V_2}$. Hence $i$ and $j$ are isomorphisms. The other implication is obvious. 
\end{proof}

\begin{prop}\label{prop:short-exact-sequence-ir-trivial} If $0\to V_1\to V_2\to V_3\to 0$ is a non vanishing short exact sequence of purely irreducible modules, then either $V_1=0$ and $V_2\simeq V_3$ or $V_3=0$ and $V_1\simeq V_2$.
\end{prop}

\begin{proof} Applying the rationalization functor we get a non vanishing short exact sequence $0\to F_\mathrm{rat}(V_1)\to F_\mathrm{rat}(V_2)\to F_\mathrm{rat}(V_3)\to 0$ of $\mathcal R$-irreducible rational $\sl(2)$-modules. Therefore, either $F_\mathrm{rat}(V_1)=0$ or $F_\mathrm{rat}(V_3)=0$, this implies the claim. 
\end{proof}

\begin{cor}\label{cor:type-is-exact-functor} The type functor $\tau\colon \mathrm{PSpl}(\sld)\to \mathrm{Spl}(\sld\Mod_\mathrm{tffr})$ is exact.
\end{cor}

Moreover, it is easy to see that $\tau\colon \mathrm{PSpl}(\sld)\to \mathrm{Spl}(\sld\Mod_\mathrm{tffr})$ is compatible with the decompositions given in Remark \ref{rem:decopmposition-purely-irred-with-irred-Casimir-tffr}  and Proposition \ref{prop:decomposition-purely-irred-types}.

\begin{prop}\label{prop:finite-pure-length-category} The category $\sld\Mod_\tffr$ satisfies both the ascending and descending chain conditions on pure submodules. 
\end{prop}

\begin{proof} This follows since the length of a chain of pure submodules cannot be greater than the rank of the module.
\end{proof}

\begin{defn}
For a finite rank torsion free $\sl(2)$-module $V$, we define its pure-length, $\plength(V)$, as the maximal length of all filtrations of $V$  by pure submodules. 
\end{defn}

\begin{prop}\label{prop:p-lenght-is-R-length-for-rational-modules} For any $W\in\Ra$ (resp. $V\in \sld\Mod_\tffr$) one has
$$\plength(W)=\length(W),\quad (\text{resp.} \plength(V)\, \leq \, \rank(V)).$$
\end{prop} 

\begin{proof}
The first claim follows from Proposition \ref{prop:rational-R-irreducible-iff-purely-irreducible}. The second is clear since $\plength(V)\leq \length(F_\rat(V))\leq\rank(V)$.
\end{proof}

Moreover, one has:

\begin{prop}\label{prop:existence-of-filtrations} 
Every finite rank torsion free $\sl(2)$-module $(V,\rho)$ has a finite filtration \begin{equation}\label{eq:purely-irreducible-filtration} (0)=V_0\subset V_1\subset \cdots\subset V_m=V\end{equation} such that every $V_i$ is a pure $\sl(2)$-submodule and the successive quotients $V_i/V_{i-1}$ are purely irreducible $\sl(2)$-modules.
\end{prop}

\begin{proof} Notice first that the length of a chain of pure submodules cannot be greater than the rank of the entire module.We proceed by induction on the rank of $V$. If $\rank(V)=1$ then $V$ is purely irreducible by Proposition \ref{prop:rank-1-tf-is-p-irred}, hence $(0)\subset V$ is the required filtration.  Therefore, let us assume that the claim is true for   modules of rank less than or equal to $n$ and let us prove it for modules of rank $n+1$. So let us consider a torsion free $\sl(2)$-module $V$ of rank $n+1$. If $V$ is purely irreducible then $(0)\subset V$ has the required properties. If $V$ is not purely irreducible, then it has at least one proper pure $\sl(2)$-module $V'\subset V$. Since $\rank(V')<\rank(V)$, by the induction hypothesis $V'$ has a filtration $V'_\bullet$ verifying the required conditions. Then $V_1:=V'_1$ is a purely irreducible submodule  and $V/V_1$ is a torsion free $\sl(2)$-module of rank at most $n$. Then by the induction hypothesis $U=V/V_1$ has a filtration  $$(0)=U_0\subset U_1\subset \cdots\subset U_{m-1}=U$$ such that every $U_i$ is a pure $\sl(2)$-submodule and the successive quotients $U_i/U_{i-1}$ are purely irreducible $\sl(2)$-modules. If $\pi\colon V\to U=V/V_1$ is the natural projection, we define now $V_i=\pi^{-1}(U_{i-1})$ for $i\geq 2$. Then  $$(0)=V_0\subset V_1\subset \cdots\subset V_m=V$$ is the desired filtration.
\end{proof}

\begin{defn} Given a finite rank torsion free $sl(2)$-module $(V,\rho)$, a filtration having the same properties as (\ref{eq:purely-irreducible-filtration}) is called a pure composition series or a pure Jordan-H\"older filtration of $(V,\rho)$.
\end{defn}

Now we have the key result:

\begin{thm}[Jordan-H\"older]\label{thm:JH-pure-composition-series} Let $(V,\rho)$ be a finite rank torsion free $\sld$-module and let \begin{align*} (0)=V_0\subset V_1\subset \cdots\subset V_m=V,\\ (0)=V'_0\subset V'_1\subset \cdots\subset V'_n=V,
\end{align*}be two pure composition series of $(V,\rho)$ with purely irreducible successive quotients $\{E_i:=V_i/V_{i-1}\}_{i=1}^m$, $\{E'_j:=V'_j/V'_{j-1}\}_{j=1}^n$, respectively. Then $n=m$, and there exists a permutation $\sigma$ of $1,\ldots,n$ such that $E_i$ is isomorphic to $E'_{\sigma(i)}$.
\end{thm}

\begin{proof} Applying the rationalization functor we get two composition series of the rational $\sl(2)$-module $F_\rat(V,\rho)$. Since the Jordan-H\"older theorem holds on $\Ra$ by Theorem \ref{thm:Jordan-Holder-Krull-Schmidt}, we conclude that $m=n$.

Now, we make an induction on the rank of $V$. If $\rank(V)=1$, then $V$ is purely irreducible by Proposition  \ref{prop:rank-1-tf-is-p-irred} and thus $(0)\subset V$ is its unique pure composition series. Therefore, let us assume that the claim is true for modules or rank at most $r$ and let us prove it for modules of rank $r+1$. Suppose that $\rank(V)=r+1$. 

If $V_1=V'_1$, then $V/V_1$ has rank less than $r$. Taking the quotient under $V_1$ of the original filtrations we get two filtrations of $V/V_1$ \begin{align*} (0)\subset V_2/V_1\subset \cdots\subset V_{n-1}/V_1\subset V/V_1,\\ (0)\subset V'_2/V_1\subset \cdots\subset V'_{n-1}/V_1\subset V/V_1,
\end{align*} whose successive quotients are $\{E_i\}_{i=2}^n$, $\{E'_i\}_{i=2}^n$. By the induction assumption there exists a permutation $\sigma$ of $2,\ldots,n$ such that $E_i$ is isomorphic to $E'_{\sigma(i)}$. Since the successive quotients of the original filtrations are $\{V_1\}\cup\{E_i\}_{i=2}^n$, $\{V_1\}\cup\{E'_i\}_{i=2}^n$, we are done.

So let us assume $V_1\neq V_1'$. Since $V_1$ and $V_1'$ are purely irreducible modules, and  $V_1\cap V_1'$ is a submodule of both of them, we conclude that either $V_1\cap V_1'$ is not a proper submodule or is an essential submodule of $V_1$ and $V_1'$. Therefore, there are four possibilities \begin{enumerate}
\item $V_1\cap V_1'\hookrightarrow V_1$, $V_1\cap V_1'\hookrightarrow V_1'$,
\item $V_1'\hookrightarrow V_1$,
\item $V_1\hookrightarrow V_1'$,
\item $V_1\cap V_1'=0$.
\end{enumerate}
Let us analyze these cases separately. 

(1) We have a commutative diagram $$\xymatrix@R=16pt@C=30pt{ & & & 0\ar[d] &\\ & 0\ar[d] & 0 \ar[d] & V_1/V_1\cap V_1'\ar[d] &\\ 0 \ar[r] & V_1\cap V_1'\ar[r]\ar[d] & V_2\ar[r]\ar[d] & V_2/V_1\cap V_1'\ar[r]\ar[d] & 0\\ 0 \ar[r] & V_1\ar[r]\ar[d] & V_2\ar[r]\ar[d] & V_2/V_1\ar[r]\ar[d] & 0\\ & V_1/V_1\cap V_1'\ar[d] & 0 & 0 &\\ & 0 & & &}$$ Since $V_1\cap V_1'$ is essential in $V_1$ we have $\rank(V_1\cap V_1')=\rank(V_1)$, hence  $\rank(V_1/V_1\cap V_1')=0$ and $\rank(V_2/V_1\cap V_1')=\rank(V_2/V_1)>0$ since $V_1$ is a pure submodule of $V_2$. Thus  $V_2/V_1\cap V_1'$ is a torsion free module and its zero rank submodule $V_1/V_1\cap V_1'$  is zero. Therefore $V_1\subset V_1'$. A similar argument shows that $V_1'\subset V_1$, hence $V_1=V_1'$ but this is not possible since we have assumed $V_1\neq V_1'$.

(2) Using the same commutative diagram as above we would also get $V_1=V_1'$ and this is again not possible.

(3) In this case the roles of $V_1$ and $V_1'$ are interchanged, and thus the same reasoning would show that this case is not possible. 

(4) Since $V_1\cap V_1'=0$ we have an natural inclusion of $\sl(2)$-modules $i\colon V_1\oplus V_1'\hookrightarrow V$.  We consider the quotient module $U=V//(V_1\oplus V_1')$ and the quotient map $\pi\colon V\to V/(V_1\oplus V_1')$. By  Proposition \ref{prop:existence-of-filtrations}, $U$ has a pure composition series $(0)=U_0\subset U_1\subset\cdots \subset  U_k=U$ and we denote its successive quotients by $\bar E_i:=U_i/U_{i-1}$. We have a commutative diagram 
$$\xymatrix@R=16pt@C=30pt{ 
& & 0\ar[d] & 0\ar[d] & 0\ar[d] &\\  
& 0\ar[r] & V_1\ar[d]\ar[r] &V_1\oplus V_1' \ar[d]^i\ar[r] & V_1'\ar[r]\ar[d] & 0\\ 
& 0\ar[r] & V\ar@{=}[r]\ar[d] & V\ar[d]^\pi\ar[r] & 0\ar[d]\ar[r]& 0\\ 
0\ar[r]& V_1' \ar[r] & V/V_1\ar[r]^{q}\ar[d] & U\ar[r]\ar[d] & 0\ar[r]\ar[d] & 0\\ 
& & 0 & 0 & 0 &}$$
and a similar one interchanging the roles of $V_1$ and $V_1'$ providing an exact sequence 
	\[
	\xymatrix{0\ar[r]& V_1 \ar[r] & V/V_1'\ar[r]^{q'} & U\ar[r] & 0 }\,.
	\]
Therefore we can define pure composition series $\{q^{-1}(U_i)\}_{i=0}^k$,  $\{q'^{-1}(U_i)\}_{i=0}^k$ of $V/V_1$ and $V/V_1'$, respectively, with successive quotients $E_1'\cup\{\bar E_i\}_{i=1}^k$, $E_1\cup\{\bar E_i\}_{i=1}^k$.

On the other hand taking the quotient of the original pure composition series by $V_1$ and $V_1'$, respectively, we get filtrations for $V/V_1$ and $V/V_1'$ given by \begin{align*} (0)\subset V_2/V_1\subset \cdots\subset V_{n-1}/V_1\subset V/V_1,\\ (0)\subset V'_2/V_1'\subset \cdots\subset V'_{n-1}/V_1'\subset V/V_1',
\end{align*} whose successive quotients are $\{E_i\}_{i=2}^n$, $\{E'_i\}_{i=2}^n$. Since both $V/V_1$ and $V/V_1'$ have rank less than $r$, by the induction hypothesis we have the following equivalences, up to permutation and isomorphism, between collections of purely irreducible $\sl(2)$-modules $$\{E_2,\ldots, E_n\}\sim\{E_1',\bar E_1,\ldots \bar E_k\},\quad \{E_2',\ldots, E_n'\}\sim\{E_1,\bar E_1,\ldots \bar E_k\}.$$ Therefore, $k=n-2$ and we also have equivalences $$ \{E_1,E_2,\ldots, E_n\}\sim\{E_1,E_1',\bar E_1,\ldots \bar E_{n-2}\}\sim\{E_1,E_2',\ldots, E_n'\},$$finishing the proof.
\end{proof} 

\section{Grothendieck groups}

If $\mathcal A$ is an abelian category and $\mathcal B$ is a full subcategory that is essentially small, then  one can  define its Grothendieck group $K_0(\mathcal B)$, see \cite[Definition, pag. 72]{Swan}, as the quotient of the free abelian group $\Z[\Iso(\mathcal B)]$  on the set of isomorphism classes  of objects of $\mathcal B$ by the subgroup of relations $R(\mathcal B)$ generated by the short exact sequences of $\mathcal A$ whose terms all belong to $\mathcal B$. Moreover, if $\mathcal B$ is additive then there is also the additive Grothendieck group $K_0^\oplus(\mathcal B)$ defined as the quotient of  $\Z[\Iso(\mathcal B)]$ by the  subgroup of relations $R^\oplus(\mathcal B)\subset R(\mathcal B)$ generated by split short exact sequences. Therefore, there is a natural  surjective group morphism $\pi^\oplus\colon K_0^\oplus(\mathcal B)\to K_0(\mathcal B)\to 0$. Besides the compatibility of $K_0$ with direct sums of exact categories, see \cite[7.1.6, pag. 142]{Weibel},  there are  two basic techniques,  introduced by  A. Heller, for computing $K_0$-groups: devissage and localization, see \cite[Theorems 3.1, 5.13, pags. 92, 115]{Swan}.

However in order to achieve our goals we need to generalize Heller's devissage theorem. Let $K_0^\bullet$ denote either $K_0^\oplus$ or $K_0$. We give the following:

\begin{defn}Let $\mathcal A$ be an abelian category,  $\mathcal B$, $\mathcal C$ be full subcategories. We say that $\mathcal C$ is a $K_0^\bullet$-devissage subcategory for $\mathcal B$ if $\mathcal C$ is a subcategory of $\mathcal B$ such that the embedding functor $i\colon \mathcal C\hookrightarrow \mathcal B$  induces a group isomorphism $K_0^\bullet(i)\colon K^\bullet_0(\mathcal C)\xrightarrow{\sim} K^\bullet_0(\mathcal B)$. Then $K_0(i)^{-1}\colon K_0(\mathcal B)\xrightarrow{\sim} K^\bullet_0(\mathcal C)$ is called the devissage isomorphism defined by  $\mathcal C$.
\end{defn}

Heller's devissage theorem says that if $\mathcal C$ is closed in $\mathcal A$ under subobjects and quotients and every object of $\mathcal B$ has a  $\mathcal C$-filtration (i.e an increasing finite filtration with all successive quotients in $\mathcal C$), then $\mathcal C$ is a $K_0$-devissage subcategory for $\mathcal B$. We need the following generalization.

\begin{thm}\label{thm:devissage-generalizado} Let $\mathcal A$ be an abelian category,  $\mathcal B$, $\mathcal C$ be full subcategories. If $\mathcal C$ is a subcategory of $\mathcal B$, such that
\begin{enumerate}
\item Every object of $\mathcal B$ has a  $\mathcal C$-filtration.
\item If $B_1\subseteq B_2\subseteq B_3$ are objects of $\mathcal B$ with $B_3/B_1\in\mathcal C$, then $B_2/B_1, B_3/B_2\in \mathcal C$.

\end{enumerate} then $\mathcal C$ is a $K_0$-devissage subcategory for $\mathcal B$. Moreover, the devissage isomorphism $$\chi:=K_0(i)^{-1}\colon K_0(\mathcal B)\to K_0(\mathcal C)$$ is given by $$\chi([B])=\sum_{i=1}^m[B_i/B_{i-1}]$$ where $B_i/B_{i-1}$ are the quotients of a $\mathcal C$-filtration $B_\bullet=(B_i)_{i=1}^m$ of $B\in\mathcal B$.
\end{thm}

\begin{proof} The embedding functor $i\colon \mathcal C\hookrightarrow\mathcal B$  gives raise to a morphism of groups $K_0(i)\colon K_0(\mathcal C)\to K_0(\mathcal B)$ and we write $i_*=K_0(i)$. Let us define a group morphism $\varphi\colon K_0(\mathcal B)\to K_0(\mathcal C)$ that is the inverse of $i_*$. By hypothesis every  object $B$ of $\mathcal B$ has a finite $\mathcal C$-filtration  $0=B_0\subseteq B_1\subseteq\cdots\subseteq B_m=B.$ Let $$\varphi(B):=\sum_{i=1}^m\left[B_i/B_{i-1}\right]\in K_0(\mathcal C).$$ In order to see that $\varphi$ is correctly defined we have to prove that it does not depend on the chosen $\mathcal C$-filtration of $B$. By  Schreier's Theorem for abelian categories, \cite[6.2, pag. 137]{Weibel}, any to $\mathcal C$-filtrations admit equivalent refinements; that is, their successive quotients are isomorphic up to a permutation. By induction we only need to check it for one insertion. Suppose that $B_{i-1}\subseteq B_i $ is changed to $B_{i-1}\subseteq B'\subseteq B_i $, then we have a short exact sequence $$0\to B'/B_{i-1}\to B_i/B_{i-1}\to B_i/B'\to 0$$ and since $B_i/B_{i-1}\in\mathcal C$, by condition (1) this is a short exact sequence in $\mathcal C$ and thus in $K_0(\mathcal C)$ we have the equality $[B_i/B_{i-1}]=[B'/B_{i-1}]+[B_i/B']$. This shows that $\varphi$ is well defined. 

By the universal property of $K_0(\mathcal B)$, to prove that $\varphi$ induces a group morphism $\chi\colon K_0(\mathcal B)\to K_0(\mathcal C)$ we must show that for any short exact sequence in $\mathcal B$ $$0\to B'\to B\xrightarrow{\pi} B''\to 0$$ one has $\varphi(B)=\varphi(B')+\varphi(B'')$. This is easy, because if $$B'_\bullet \equiv 0=B'_0\subseteq B'_1\subseteq\cdots\subseteq B'_{m'}=B',\quad B''_\bullet\equiv0=B''_0\subseteq B''_1\subseteq\cdots\subseteq B''_{m''}=B''$$ are $\mathcal C$-filtrations, then $$0=B'_0\subseteq B'_1\subseteq\cdots\subseteq B'_{m'}\subseteq \pi^{-1}(B''_1)\subseteq\cdots\subseteq \pi^{-1}(B''_{m''})=\pi^{-1}(B'')=B$$ is a $\mathcal C$-filtration whose successive quotients are those of $B'_\bullet$ together with those of $B''_\bullet$, proving thus the additivity of $\varphi$ on short exact sequences.  For every $C\in\mathcal C$ one has $\chi(i_*([C]))=[C]$ because $0\subseteq C$ is a $\mathcal C$-filtration.  On the other hand, if $B_\bullet$ is a $\mathcal C$-filtration of $B\in\mathcal B$, writing down the short exact sequences defined by the successive quotients, $0\to B_{i-1}\to B_i\to B_i/B_{i-1}\to 0$, it follows that $i_*(\chi([B]))=[B]$. Thus $i_*\colon K_0(\mathcal C)\to K_0(\mathcal B)$ is a group isomorphism whose inverse  is $\chi$. \end{proof}

\subsection{Grothendieck groups of the category of rational mo\-du\-les}\label{sec:Grothendieck-rational-modules}

We have proved in Theorem \ref{thm:rational-essentially-small} that the abelian category $\mathcal R$ of rational $\sl(2)$-modules is essentially small, therefore we can define its Grothendieck groups $K_0^\oplus(\mathcal R)$ and $K_0(\mathcal R)$. For the first one, if we denote by $[W]^\oplus\in K_0^\oplus(\mathcal R)$ the class defined by a rational module $W$, we have the following result. 

\begin{thm}[global Krull-Schmidt-Remak devissage]\label{thm:additive-Grothendieck-group} The category $\Ind_{\mathcal{RC}^\bullet}({\mathcal{RC}^\bullet})$ is a $K^\oplus_0$-devissage subcategory for the abelian category $\mathcal R$. Therefore, $K^\oplus_0(\mathcal R)$   is isomorphic to the free abelian group on the isomorphism classes of $\mathcal R$-indecomposable generalized Casimir rational $\sl(2)$-modu\-les. Moreover, the Krull-Schmidt-Remak devissage isomorphism  $$\chi_{_{\text{\tiny KSR}}}\colon K_0^\oplus(\mathcal R)\xrightarrow{\sim}\Z[\widehat\Ind_{\mathcal{RC}^\bullet}(\mathcal{RC}^\bullet)]$$ is given  by $$\chi_{_{KSR}}([W]^\oplus)=\sum_{i=1}^m[W_i]^\oplus$$ with $W_i$ the factors of a Remak decomposition $W=\oplus_{i=1}^m W_i$ of a module $W\in\Ra$. 
\end{thm}	

\begin{proof}
By Proposition \ref{prop:indecomposable-generalized-Casimir} one has $\Ind_{\mathcal R}(\mathcal R)=\Ind(\mathcal R)=\Ind_{\mathcal{RC}^\bullet}(\mathcal{RC}^\bullet)=\Ind(\mathcal{RC}^\bullet)$. The $K_0^\bullet$-devissage property follows from \cite[Theorem 2.1, pag. 76]{Swan} since we know, thanks to Theorem \ref{thm:Jordan-Holder-Krull-Schmidt}, that $\mathcal R$ is a Krull-Schmidt category. The result follows since on $\Ind({\mathcal R})$ there are no nontrivial split exact sequences.
\end{proof}

Now we refine this result by means of the particular structure of the category $\mathcal R$.

\begin{thm}[decomposition]\label{thm:descomposición-Grothendieck-aditivo}
There is a decomposition isomorphism $$\Phi^\oplus:K_0^\oplus(\mathcal R)\xrightarrow{\sim}\bigoplus_{\mu\in\C}K_0^\oplus(\mathcal {RC}^\bullet_\mu)$$ given by $$\Phi^\oplus([W]^\oplus)=\sum_{i=1}^p[W_{\mu_i}]^\oplus$$ where $W_{\mu_i}\in\mathcal {RC}_{\mu_i}^\bullet$ are the factors of the decomposition of the rational $\sld$-module $W$ provided by the minimal polynomial of its Casimir operator.  There is also a decomposition isomorphism $\Psi^\oplus\colon \Z[\widehat\Ind_{\mathcal{RC}^\bullet}(\mathcal{RC}^\bullet)]\xrightarrow{\sim}\bigoplus_{\mu\in\C} \Z[\widehat\Ind_{\mathcal{RC}^\bullet_\mu}(\mathcal{RC}^\bullet_\mu)]$.
\end{thm}

\begin{proof}
The first decomposition is obtained by taking into account  that $K_0^\oplus$ is compatible with direct sums of abelian categories, see \cite[pag. 125]{Weibel}, and the identification of categories $\phi\colon \Ra\xrightarrow{\sim}\bigoplus_{\mu\in\C}\RC^\bullet_\mu$ proved in Theorem \ref{thm:finite+decomp}. We define $\Phi^\oplus:=K_0(\phi)$. For the second decomposition, notice that $\phi$ induces also an identification $\psi\colon \Ind_{\mathcal{RC}^\bullet}(\mathcal{RC}^\bullet)\xrightarrow{\sim} \coprod_{\mu\in\C} \Ind_{\mathcal{RC}^\bullet_\mu}(\mathcal{RC}^\bullet_\mu)$. We put $\Psi^\oplus:=K_0(\psi)$ and consider the identification used in the proof of  Theorem \ref{thm:additive-Grothendieck-group}. 
\end{proof}

Following the proof of Theorem \ref{thm:additive-Grothendieck-group} we get now:
\begin{thm}[Krull-Schmidt-Remak devissage of level $\mu$]\label{thm:KRS-devissage-of-level-mu}
For every $\mu\in\C$ the category $\Ind_{\mathcal{RC}^\bullet_\mu}({\mathcal{RC}^\bullet_\mu})$ is a $K^\oplus_0$-devissage subcategory for the abelian category $\mathcal{RC}^\bullet_\mu$. Moreover, the Krull-Schmidt-Remak devissage isomorphism of level $\mu$   $$\chi_{_{KSR},\mu}\colon K_0^\oplus(\mathcal{RC}^\bullet_\mu)\xrightarrow{\sim}\Z[\widehat\Ind_{\mathcal{RC}^\bullet_\mu}(\mathcal{RC}^\bullet_\mu)]$$ is given  by $$\chi_{_{KSR},\mu}([W]^\oplus)=\sum_{i=1}^m[W_i]^\oplus$$ with $W_i$ the factors of a Remak decomposition $W=\oplus_{i=1}^m W_i$ of a module $W\in\RC^\bullet_\mu$.
\end{thm}
As a consequence of Theorems \ref{thm:additive-Grothendieck-group},  \ref{thm:descomposición-Grothendieck-aditivo}  and \ref{thm:KRS-devissage-of-level-mu} we get the next result.
\begin{cor}[compatibility]\label{cor:compatibility-aditivo}
There is an identification $$\chi_{_{KRS}}=\bigoplus_{\mu\in\C}\chi_{_{KRS},\mu}\colon \bigoplus_{\mu\in\C} K_0^\oplus(\mathcal {RC}^\bullet_\mu)\xrightarrow{\sim}\bigoplus_{\mu\in\C} \Z[\widehat\Ind_{\mathcal{RC}^\bullet_\mu}(\mathcal{RC}^\bullet_\mu)] .$$ \end{cor}

From now on in this section devissage will mean $K_0$-devissage. For the general Grothedieck group, we have:

\begin{thm}[global Jordan-H\"older devissage]\label{t:Grothendieck-rational} The category $\Spl_\RC(\RC)$ is a devissage subcategory for the abelian category $\mathcal R$. Thus $K_0(\mathcal R)$ is isomorphic to the free abelian group on the isomorphism classes of $\mathcal R$-irreducible rational Casimir $\sld$-modules. Moreover, the Jordan-H\"older devissage isomorphism $$\chi_{_{JH}}\colon K_0(\mathcal R)\xrightarrow{\sim}\Z[\widehat{\Spl}_\RC(\RC)]$$ is given  by $$\chi_{_{JH}}([W])=\sum_{i=1}^m[W_i/W_{i-1}]$$ with $W_i/W_{i-1}$ the quotients of a composition series $W_\bullet=(W_i)_{i=0}^m$ of $W\in\mathcal R$. 
\end{thm}

\begin{proof}
We have proved in Theorem \ref{thm:Jordan-Holder-Krull-Schmidt} that $\mathcal R$ is a finite length category. In this case, Heller's devissage theorem \cite[Theorem 3.1,  pag. 92]{Swan} does apply to $\Spl_\RC(\RC)$ and $\Ra$, see \cite[Corollary 3.2, pag. 93]{Swan}.  The result follows since on $\Spl_\Ra(\mathcal R)$ there are no nontrivial short exact sequences.
\end{proof}

Now we reflect the structure of $\mathcal R$ on the Grothendieck group. We begin with a natural decomposition that, as Theorem  \ref{thm:descomposición-Grothendieck-aditivo}, is a consequence of Theorem  \ref{thm:finite+decomp}.   \begin{thm}[decomposition]\label{thm:decomposition-Grothendieck}
There is a decomposition isomorphism$$\Phi\colon K_0(\mathcal R)\xrightarrow{\sim}\bigoplus_{\mu\in\C}K_0(\mathcal {RC}^\bullet_\mu)$$   given by $$\Phi([W])=\sum_{i=1}^p[W_{\mu_i}]$$ where $W_{\mu_i}\in\mathcal {RC}_{\mu_i}^\bullet$ are the factors of the decomposition of the rational $\sld$-module $W$ provided by the minimal polynomial of its Casimir operator. 
There is also a decomposition isomorphism $\Psi\colon \Z[\widehat\Spl_{\mathcal{RC}}(\mathcal{RC})]\xrightarrow{\sim}\bigoplus_{\mu\in\C} \Z[\widehat\Spl_{\mathcal{RC}_\mu}(\mathcal{RC}_\mu)]$.
\end{thm} 

Moreover, we have:

\begin{thm}[canonical filtration devissage of level $\mu$]\label{thm:first-devissage} The category $\mathcal{RC}_\mu$ is a devissage subcategory for the abelian category $\mathcal{RC}^\bullet_\mu$ and the canonical filtration devissage isomorphism of level $\mu$ $$\chi_{_{CF},\mu}\colon K_0(\RC^\bullet_\mu)\xrightarrow{\sim}K_0(\RC_\mu)$$ is given by $$\chi_{_{CF},\mu}([W])=\sum_{i=1}^l [W_i/W_{i-1}]$$ where $W_i/W_{i-1}\in\RC_\mu$ are the quotients of the canonical filtration $W_\bullet=(W_i)_{i=0}^l$ of a generalized Casimir rational $\sl(2)$-module $W$ of level $\mu$.
\end{thm}

\begin{proof} 

Let us show that $\RC_\mu$ and $\RC_\mu^\bullet$ satisfy the conditions of Theorem \ref{thm:devissage-generalizado} with $\mathcal C=\RC_\mu$ and $\mathcal B=\mathcal A=\RC_\mu^\bullet$. This choice is possible since we know that $\RC_\mu^\bullet$ is an abelian category. The existence of $\RC_\mu$-filtrations for every object of $\RC_\mu^\bullet$ is guaranteed by the existence of the canonical filtration given  in Proposition \ref{prop:filtration}. Let us assume now that we have on $\RC^\bullet_\mu$  a chain $V_1\subseteq V_2\subseteq V_3$ such that $V_3/V_1\in \RC_\mu$. From the short exact sequence in $\RC_\mu^\bullet$ $$0\to V_2/V_1\to V_3/V_1\to V_3/V_2\to 0$$ it follows that all of its terms belong to $\RC_\mu$.
\end{proof}
Proceeding as in Theorem \ref{t:Grothendieck-rational} we get: 
\begin{thm}[Jordan-H\"older devissage of level $\mu$]\label{thm:second-devissage} The category $\Spl_{\RC_\mu}(\RC_\mu)$ is a devissage subcategory for $\RC_\mu$. Hence $K_0(\RC_\mu)$ is isomorphic to the free abelian group on the isomorphism classes of $\mathcal R$-irreducible  Casimir  rational $\sld$-modules of level $\mu$. Moreover, the Jordan-H\"older devissage isomorphism  of level $\mu$ $$\chi_{_{JH},\mu}\colon K_0(\RC_\mu)\xrightarrow{\sim}\Z[\widehat{\Spl}_{\RC_\mu}(\RC_\mu)]$$ is given  by $$\chi_{_{JH},\mu}([W])=\sum_{i=1}^m[W_i/W_{i-1}]$$ with $W_i/W_{i-1}$ the quotients of a composition series $W_\bullet=(W_i)_{i=0}^m$ of $W\in\RC_\mu$. 
\end{thm}

As a consequence of Theorems \ref{t:Grothendieck-rational}, \ref{thm:decomposition-Grothendieck}, \ref{thm:first-devissage} and \ref{thm:second-devissage} we get the following result.

\begin{cor}[compatibility]\label{cor:compatibility}
There is an identification $$\chi_{_{JH}}=\bigoplus_{\mu\in\C}\chi_{_{JH},\mu}\circ\chi_{_{CF},\mu}\colon \bigoplus_{\mu\in\C} K_0(\mathcal {RC}^\bullet_\mu)\xrightarrow{\sim}\bigoplus_{\mu\in\C} \Z[\widehat\Spl_{\mathcal{RC}_\mu}(\mathcal{RC}_\mu)] .$$ \end{cor}

\begin{rem} 
Given two levels $\mu, \nu\in\C$, by Corollary \ref{cor:isom-Casimir-mu-nu} there is an exact functor $\Phi_{\mu\nu}\colon \RC_\mu\to\RC_\nu$ 
inducing an equivalence of categories. Therefore, we have an isomorphism $\Phi_{\mu\nu}\colon K_0(\RC_\mu)\xrightarrow{\sim}K_0(\RC_\nu)$. Hence, fixing $\mu_0\in\C$  we get an identification $K_0(\Ra)\xrightarrow{\sim}\oplus_{\mu\in\C} K_0(\RC_{\mu_0})\simeq \oplus_{\mu\in\C}\Z[\widehat\Spl(\RC_{\mu_0})]$. %can not be determined just by knowing  one of the Grothendieck groups $K_0({\mathcal{C}}^\bullet_{\mu_0,\tffr})\simeq \Z[\widehat{\mathrm{PSpl}}(\sl(2))_{\mu_0}]$ for a particular level $\mu_0\in\C$.
% isomorphism of abelian groups $\Phi_{\mu\nu}\colon \Z[\widehat{\Spl}_{\mathcal{RC}_\mu}(\mathcal{RC}_\mu)]\xrightarrow{\sim} \Z[\widehat{\Spl}_{\mathcal{RC}_\nu}(\mathcal{RC}_\nu)]$.  Thus, $ \chi_{_{CF},\mu}^{-1}\circ \chi_{_{JH},\nu}^{-1}\circ\Phi_{\mu\nu}\circ\chi_{_{JH},\mu}\circ\chi_{_{CF},\mu}$ yields an isomorphism $ K_0(\mathcal{RC}_\mu^{\bullet})\xrightarrow{\sim} K_0(\mathcal{RC}_\nu^{\bullet})$.
\end{rem}

\subsection{Grothendieck groups of the category of finite rank torsion free modules}
We have proved in Theorem \ref{thm:tffr-is-essentially-small} that the category of finite rank torsion free modules $\sld\Mod_\tffr$ is  essentially small. Therefore we can define its Grothendieck groups $K_0^\oplus(\sld\Mod_\tffr)$ and $K_0(\sld\Mod_\tffr)$.
The determination of these groups in the present case runs in parallel to the computations carried out for rational modules. For the sake of brevity, we just indicate how the theorems proved in Section \ref{sec:Grothendieck-rational-modules} can be obtained for finite rank torsion free modules, we use the notation introduced there. 

Theorems \ref{thm:additive-Grothendieck-group}, \ref{thm:descomposición-Grothendieck-aditivo}, \ref{thm:KRS-devissage-of-level-mu} and \ref{cor:compatibility-aditivo} regarding the additive Grothendieck group $K_0^\oplus$  remain valid replacing $\Ra$ by $\sld\Mod_\tffr$, $\RC^\bullet$ by $\mathcal C_\tffr^\bullet$ and $\RC^\bullet_\mu$ by $\mathcal C_{\mu,\tffr}^\bullet$. The proofs are the same, save that now in Theorem \ref{thm:additive-Grothendieck-group} we use Proposition \ref{prop:irred-tffr-is-Casimir} that gives $\Ind(\sld\Mod_\tffr)=\Ind(\mathcal C^\bullet_\tffr)$ and in Theorems \ref{thm:descomposición-Grothendieck-aditivo} we use the identification of categories $ \phi\colon \sl(2)\Mod_\mathrm{tffr}\,\xrightarrow{\sim}\, \bigoplus_{\mu\in \C} {\mathcal{C}}^\bullet_{\mu,\mathrm{tffr}}$ proved in Theorem \ref{thm:finite+decomp-tffr}.

For the general Grothendieck group $K_0$, Theorems \ref{t:Grothendieck-rational}, \ref{thm:decomposition-Grothendieck}, \ref{thm:first-devissage}, \ref{thm:second-devissage} and \ref{cor:compatibility} remain valid but the situation is more involved. Therefore we present now  the precise statements and indicate the changes needed in their proofs.

\begin{thm}[global Pure Jordan-H\"older devissage]\label{t:Grothendieck-rational-tffr} The category $\PSpl(\sld)$ is a devissage subcategory for the exact category $\sld\Mod_\tffr$. Thus $K_0(\sld\Mod_\tffr)$ is isomorphic to the free abelian group on the isomorphism classes of purely irreducible $\sld$-modules. Moreover, the Pure Jordan-H\"older devissage isomorphism $$\chi_{_{PJH}}\colon K_0(\sld\Mod_\tffr)\xrightarrow{\sim}\Z[\widehat{\PSpl}(\sld)]$$ is given  by $$\chi_{_{PJH}}([V])=\sum_{i=1}^m[V_i/V_{i-1}]$$ with $V_i/V_{i-1}$ the quotients of a pure composition series $V_\bullet$ of $V\in \sld\Mod_\tffr$. 
\end{thm}

\begin{proof}
We have proved that $\sld\Mod_\tffr$ is exact, Proposition \ref{prop:exact-tffr}, and of finite pure-length, Lemma \ref{prop:finite-pure-length-category}. Therefore any $V\in \sld\Mod_\tffr$ has a pure composition series $V_\bullet=(V_i)_{i=0}^m$. Let us put $$\varphi(V):=\sum_{i=1}^m[V_i/V_{i-1}]\in K_0(\PSpl(\sld).$$  This is well defined since by Theorem \ref{thm:JH-pure-composition-series}  the Jordan-H\"older theorem is valid for pure composition series. Now one proceeds as in the proof of Theorem \ref{thm:devissage-generalizado}, showing that $\varphi$ is additive on short exact sequences because if we join two composition series, as we did there, one checks easily that we get a pure composition series.  The rest of the proof is completely analogous. The claim follows since according to Proposition \ref{prop:short-exact-sequence-ir-trivial} there are no nontrivial short exact sequences  on $\PSpl(\sl2)$.
\end{proof}

The following result is a consequence of Theorem  \ref{thm:finite+decomp-tffr}.   \begin{thm}[decomposition]\label{thm:decomposition-Grothendieck-tffr}
There is a decomposition isomorphism$$\Phi_\tffr\colon K_0(\sld\Mod_\tffr)\xrightarrow{\sim}\bigoplus_{\mu\in\C}K_0(\mathcal C^\bullet_{\mu,\tffr})$$   given by $$\Phi([V])=\sum_{i=1}^p[V_{\mu_i}]$$ where $V_{\mu_i}\in\mathcal C^\bullet_{\mu,\tffr}$ are the factors of the decomposition of  $V\in \mathcal C^\bullet_{\mu,\tffr}$ provided by the minimal polynomial of its Casimir operator. There is also a decomposition isomorphism $\Psi_\tffr\colon \Z[\widehat{\PSpl}(\sld)]\xrightarrow{\sim}\bigoplus_{\mu\in\C} \Z[\widehat{\PSpl}(\sld)_\mu)]$.
\end{thm} 

With respect to this decomposition we have the following commutative diagram of embeddings of categories $$\xymatrix{\PSpl(\sld)_\mu\ar@{^{(}->}[rr]^{a^\bullet_\mu}\ar@{^{(}->}[dr]_{b_\mu}& & \mathcal C_{\mu,\tffr}^\bullet\\& \mathcal C_{\mu,\tffr}\ar@{^{(}->}[ur]_{c_\mu^\bullet}&
}$$

We use it to prove the following result.

\begin{thm}[canonical filtration devissage of level $\mu$]\label{thm:first-devissage-tffr} The category $\mathcal{C}_{\mu,\tffr}$ is a devissage subcategory for the exact category $\mathcal{C}_{\mu,\tffr}^\bullet$ and the canonical filtration devissage isomorphism of level $\mu$ $$\chi_{_{CF},\mu}\colon K_0(\mathcal{C}_{\mu,\tffr}^\bullet)\xrightarrow{\sim}K_0(\mathcal{C}_{\mu,\tffr})$$ is given by $$\chi_{_{CF},\mu}([V])=\sum_{i=1}^l [V_i/V_{i-1}]$$ where $V_i/V_{i-1}\in\mathcal{C}_{\mu,\tffr}$ are the quotients of the canonical filtration $V_\bullet=(V_i)_{i=0}^l$ of a generalized torsion free Casimir $\sl(2)$-module $V$ of level $\mu$.
\end{thm}

\begin{proof}
The commutative diagram of category embeddings gives the equality $$K_0(a^\bullet_\mu)=K_0(c_\mu^\bullet)\circ K_0(b_\mu).$$ Proceeding as in Theorem \ref{t:Grothendieck-rational-tffr} one has that both $K_0(a^\bullet_\mu)$ and $K_0(b_\mu)$ are isomorphisms and therefore $K_0(c_\mu)$ is also an isomorphism. This proves that $\mathcal{C}_{\mu,\tffr}$ is a devissage subcategory for the exact category $\mathcal{C}_{\mu,\tffr}^\bullet$. We denote $(c_\mu^\bullet)_*:=K_0(c_\mu^\bullet)$.

Given $V\in\mathcal{C}_{\mu,\tffr}^\bullet$ we define $$\varphi_{_{CF},\mu}(V):= \sum_{i=1}^m[V_i/V_{i-1}]\in K_0(\mathcal{C}_{\mu,\tffr})$$ where $V_i/V_{i-1}\in\mathcal{C}_{\mu,\tffr}$ are the quotients of the canonical filtration $V_\bullet=(V_i)_{i=0}^l$ of  $V$ described in Proposition \ref{prop:filtration}. Let us prove now that $\varphi_{_{CF},\mu}$ is additive on short exact sequences. In first place, by considering the short exact sequences of the canonical filtration $V_\bullet$, one checks that $(c_\mu^\bullet)_*(\varphi_{_{CF},\mu}(V))=[V]\in K_0(\mathcal{C}_{\mu,\tffr})$. Given a short exact sequence $0\to V'\to V\xrightarrow{\pi} V''\to 0$ in $\mathcal{C}_{\mu,\tffr}$ let $V'_\bullet=(V'_i)_{i=0}^{m'}$, $V_\bullet=(V_i)_{i=0}^{m}$, $V''_\bullet=(V''_i)_{i=0}^{m''}$ be the canonical filtrations of $V'$, $V$, $V''$, respectively. Proceeding as in the proof of Theorem \ref{thm:devissage-generalizado}, joining $V'_\bullet$ to $\pi^{-1}(V''_\bullet)$ we get another filtration $\widetilde V_\bullet$ of $V$ whose successive quotients are those of $V'_\bullet$ together with those of $V''_\bullet$. Therefore, the short exact sequences of the filtration $\widetilde V_\bullet$ of $V$ give us 
\begin{align*}[V]=(c_\mu^\bullet)_*\left(\sum_{i=1}^{m'}[V'_i/V'_{i-1}]+\sum_{i=0}^{m''}[V''_i/V''_{i-1}]\right)=(c_\mu^\bullet)_*(\varphi_{_{CF},\mu}(V')+\varphi_{_{CF},\mu}(V'')).\end{align*} Since we also have $[V]=(c_\mu^\bullet)_*(\varphi_{_{CF},\mu}(V))$ and we have proved that $(c_\mu^\bullet)_*$ is injective, we conclude that $\varphi_{_{CF},\mu}(V))=\varphi_{_{CF},\mu}(V'))+\varphi_{_{CF},\mu}(V'')$ as claimed. Therefore, by the universal property of the Grothendieck group,  $\varphi_{_{CF},\mu}$ induces a group morphism $\chi_{_{CF},\mu}\colon K_0(\mathcal{C}_{\mu,\tffr}^\bullet)\to K_0(\mathcal{C}_{\mu,\tffr})$ such that $(c_\mu^\bullet)_*\circ \chi_{_{CF},\mu}=\Id_{K_0(\mathcal{C}_{\mu,\tffr}^\bullet)}$. If $C\in\mathcal{C}_{\mu,\tffr}$, then  its canonical filtration is $0\subsetneq C$, thus $\chi_{_{CF},\mu}((c_\mu^\bullet)_*([C]))=[C]$. This shows that $\chi_{_{CF},\mu}\circ (c_\mu^\bullet)_*=\Id_{K_0(\mathcal{C}_{\mu,\tffr})}$, finishing the proof.
\end{proof}

Proceeding as in Theorem \ref{t:Grothendieck-rational-tffr} we get: 

\begin{thm}[Jordan-H\"older devissage of level $\mu$]\label{thm:second-devissage-tffr} The category $\PSpl_{}(\sld)_\mu$ is a devissage subcategory for $\mathcal C_{\mu,\tffr}$. Hence $K_0(\mathcal C_{\mu,\tffr})$ is isomorphic to the free abelian group on the isomorphism classes of purely irreducible  Casimir torsion free $\sld$-modules of finite rank and level $\mu$. Moreover, the Jordan-H\"older devissage isomorphism  of level $\mu$ $$\chi_{_{PJH},\mu}\colon K_0(\mathcal C_{\mu,\tffr}))\xrightarrow{\sim}\Z[\widehat\PSpl_{}(\sld)_\mu]$$ is given  by $$\chi_{_{PJH},\mu}([V])=\sum_{i=1}^m[V_i/V_{i-1}]$$ with $V_i/V_{i-1}$ the quotients of a composition series $V_\bullet=(V_i)_{i=0}^m$ of $V\in\mathcal C_{\mu,\tffr}$. 
\end{thm}

As a consequence of Theorems \ref{t:Grothendieck-rational-tffr}, \ref{thm:decomposition-Grothendieck-tffr}, \ref{thm:first-devissage-tffr} and \ref{thm:second-devissage-tffr} we get the following result.

\begin{cor}[compatibility]\label{cor:compatibility-tffr}
There is an identification $$\chi_{_{PJH}}=\bigoplus_{\mu\in\C}\chi_{_{PJH},\mu}\circ\chi_{_{CF},\mu}\colon \bigoplus_{\mu\in\C} K_0(\mathcal {C}^\bullet_\mu)\xrightarrow{\sim}\bigoplus_{\mu\in\C} \Z[\widehat{\PSpl}(\sld)_\mu].$$ Moreover, by Proposition \ref{prop:decomposition-purely-irred-types} there is a further decomposition  $$ \Z[\widehat{\mathrm{PSpl}}(\sl(2))_\mu]=\bigoplus_{\hat\tau\in\widehat{\mathrm{Spl}}(\mathcal C_{\mu,\mathrm{tffr}})}\Z[\widehat{\mathrm{PSpl}}(\sl(2);\hat\tau)_\mu].$$ \end{cor}

\begin{rem} The algorithm for disassembling (devissage) the class of a finite rank torsion free module $V$ in the Grothendieck group $[V]\in K_0(\sld\Mod_\tffr)$  into a sum of classes of purely irreducible modules in $K_0(\PSpl(\sld))$ proceeds in stages:
\begin{enumerate}
\item First, we perform the decomposition $V=V_{\mu_1}\oplus\cdots\oplus V_{\mu_p}$ into generalized Casimir representations, where $P_{\mathrm{min}}^{C_\rho}(t)=(t-\mu_1)^{n_1}\cdots(t-\mu_p)^{n_p}$ is the minimal polynomial of the Casimir operator of $V$. Hence, one has $$[V]=[V_{\mu_1}]+\cdots+[V_{\mu_p}].$$
\item The class   $[V_{\mu_i}]\in K_0(\mathcal C_{\mu_i,\tffr}^{\bullet})$ is decomposed by means of the canonical filtration $V_{\mu_i,\bullet}=(V_{\mu_i,j})_{j=1}^{n_i}$ of $V_{\mu_i}$ to give $$[V_{\mu_i}]=\sum_{j=1}^{n_i}[V_{\mu_i,j}/V_{\mu_i,j-1}]\in K_0(\mathcal C_{\mu_i,\tffr}).$$
\item The class $[V_{\mu_i,j}/V_{\mu_i,j-1}]\in K_0(\mathcal C_{\mu_i,\tffr})$ is finally decomposed by means of a pure composition series $P_{\mu_i,j,\bullet}=(P_{\mu_i,j,k})_{k=0}^{r_{ij}}$ of $V_{\mu_i,j}/V_{\mu_i,j-1}$, yielding $$[V_{\mu_i,j}/V_{\mu_i,j-1}]=\sum_{k=1}^{r_{ij}} [P_{\mu_i,j,k}/P_{\mu_i,j,k-1}]\in K_0(\PSpl(\sld)_{\mu_i}).$$
\item Summing all together we get $$[V]=\sum_{i=1,j=1,k=1}^{p,n_i,r_{ij}}[P_{\mu_i,j,k}/P_{\mu_i,j,k-1}].$$
\end{enumerate}
\end{rem}

\begin{rem}
By Proposition \ref{prop:purely-irreducible-is-Casimir} we know that purely irreducible modules are Casimir modules. For every pair of levels   $\mu,\nu\in\C$ there is an exact functor $$\Psi_{\mu\nu}\colon {\mathrm{PSpl}}(\sl(2))_\mu\to {\mathrm{PSpl}}(\sl(2))_\nu$$ that maps $(V,\rho)\in {\mathrm{PSpl}}(\sl(2))_\mu$ to $\Psi_{\mu\nu}((V,\rho))=(S_{\pi_\mu(z)}^{-1}V,\Psi_{\mu\nu}(\rho))$ where $S_{\pi_\mu(z)}$ is the smallest multiplicative system of $\C[z]$ that is left invariant under $\pi_\mu(z)\nabla^{-1}$ and $\nabla$ (the precise description of  $S_{\pi_\mu(z)}$ can be determined easily)  and $$\Psi_{\mu\nu}(\rho)(L_{-1})\left(\frac{v}{s}\right)=\pi_\nu(z)\frac{\rho(L_{-1})(v)}{\pi_\mu(z)\nabla^{-1}(s)},\quad \Psi_{\mu\nu}(\rho)(L_{1})\left(\frac{v}{s}\right)=\frac{\rho(L_{1})(v)}{\nabla(s)}.$$ This functor is compatible with the isomorphisms established in Corollary \ref{cor:isom-Casimir-mu-nu}  for the category of rational Casimir $\sld$-modules. That is, there is a commutative diagram of functors:

$$\xymatrix{ {\mathrm{PSpl}}(\sl(2))_\mu\ar[r]^(.62){F_\rat}\ar[d]_{\Psi_{\mu\nu}} & \Ra\mathcal C_\mu\ar[d]^{\Phi_{\mu\nu}}_\wr\\
{\mathrm{PSpl}}(\sl(2))_\nu\ar[r]^(.62){F_\rat} & \Ra\mathcal C_\mu}$$

Although there is a morphism of functors $\Id_{{\mathrm{PSpl}}(\sl(2))_\mu}\to \Psi_{\nu\mu}\circ \Psi_{\mu\nu}$, one checks that  $\Psi_{\nu\mu}\circ \Psi_{\mu\nu}$ is not the identity functor. Thus, in contrast to what happens for rational $\sl(2)$-modules, the Grothendieck group $K_0(\sl(2)\Mod_\mathrm{tffr})$ can not be determined by knowing $K_0({\mathcal{C}}_{\mu_0,\tffr})\simeq \Z[\widehat{\mathrm{PSpl}}(\sl(2))_{\mu_0}]$ for a particular level $\mu_0\in\C$.
\end{rem}

%%%%%%%%%%%%%%%%
\subsection{Relationship between the Grothedieck groups of rational and torsion free  finite rank modules. The localization theorem}

The rationalization functor $F_\rat\colon \sl(2)\Mod_\mathrm{tffr}\to \Ra$ is exact and therefore it induces a group morphism between the Grothendieck groups of these categories: $$F_{\rat*}:=K_0(F_\rat)\colon K_0(\sl(2)\Mod_\mathrm{tffr})\to K_0(\Ra).$$ Moreover, we have seen in Theorem \ref{thm:rational-is-reflective-localization-of-tffr} that $\mathcal R$ is a reflective localization of  the category $\sl(2)\Mod_\mathrm{tffr}$ with quotient functor $F_\rat$. We prove now the  analogue of Heller's localization theorem, \cite[Theorem 5.13, pag. 115]{Swan}. 

\begin{thm}\label{thm:localization-theorem}
There is a short exact sequence
	  	\[
	  0 \to    \Ker (F_{\rat*}) \to   K_0(\sl(2)\Mod_\mathrm{tffr})\xrightarrow{F_{\rat*}}  K_0(\Ra) \to  0
	  \]
that splits naturally by the section $i_{\tffr*}\colon K_0(\mathcal R)\to K_0(\sl(2)\Mod_\mathrm{tffr})$ defined by the embedding $i_\tffr\colon\mathcal R\hookrightarrow \sl(2)\Mod_\mathrm{tffr}$. Therefore, the corresponding retraction $R\colon K_0(\sl(2)\Mod_\mathrm{tffr})\to \Ker(F_{\rat*})$ maps the class $[V]\in  K_0(\sl(2)\Mod_\mathrm{tffr})$ of a module  to $R([V])=[V]-i_{\tffr*}[F_\rat(V)]\in\Ker (F_{\rat*})$. It follows that $ \Ker (F_{\rat*})$ gets identified with the group of virtual torsion modules of finite rank $$\mathcal{VT}(\sl(2)\Mod_\mathrm{tffr})=\{[V_1]-[V_2]\colon V_1\subseteq V_2\in\sl(2)\Mod_\mathrm{tffr}, V_1/V_2\ \text{is torsion}\}$$  and  $$K_0(\sl(2)\Mod_\mathrm{tffr})\xrightarrow{\sim} \mathcal{VT}(\sl(2)\Mod_\mathrm{tffr})\oplus K_0(\Ra).$$ Moreover, given $W\in\RC_\mu$ we have $F^{-1}_{\rat*}([W])=\widehat\PSpl(\sld,\tau([W]))_\mu$.
\end{thm}

\begin{proof} We have seen in Proposition \ref{prop:rationalization-is-a-retraction} that $F_\rat$ is a retraction of $i_\tffr$. This implies the first claim. The form of the associated retraction is obvious. It is clear that $\mathcal{VT}(\sl(2)\Mod_\mathrm{tffr})\subseteq \Ker (F_{\rat*})$. The other inclusion follows since $R$ is surjective. The final claim follows from Proposition \ref{prop:decomposition-purely-irred-types}.
\end{proof}

To conclude this section let us point out that, thanks to the universal property of the Grothendieck group, the additive functions $\rank$ and $\dim$ induce surjective group morphisms that fit into the following commutative diagram of groups
	\[
	\xymatrix@C=16pt{K_0(\sl(2)\Mod_\mathrm{tffr})\ar[rr]^(.6){K_0(F_\rat)}\ar[dr]_{\rank} & & K_0(\Ra)\ar[dl]^{{\dim}}{}\\ & \Z &}
	\]
Finally, there is an analogous diagram replacing $\rank$ by $\plength$ and $\dim$ by $\length$. It follows that virtual torsion modules have zero $\rank$  and zero $\plength$.

%%%%%%%%%%%%%%%%%%%%%%%%%%%%%%%%%%%%%%%%%%
\section{Rational representations of dimension $1$}\label{s:dim1}

The dimension of rational representations  as $\C(z)$-vector spaces stratifies the category $\Ra$ into strata $\Ra^m$ of representations of dimension $m$. In this section we focus on the stratum $\Ra^1$ of dimension $1$ and we establish when these representations do arise as rationalization of a polynomial representation of rank $1$.  We also give the structure of all $\sl(2)$-submodules of a one dimensional rational representation.

\subsection{The stratum of one dimensional rational representations}\label{section:stratum-one-dimensional-rational}\label{subsec:one-dimensional-stratum}

Let $W$ be a one dimensional $\C(z)$-vector space. We are interested in $\sl(2)$-representations $\rho$ defined on $W$ such that $\rho(L_0)=z$. By Corollary \ref{cor:rational-onedim-are-irred-Casimir} it follows that $(W,\rho)$ is an $\Ra$-irreducible rational Casimir module, of level $\mu$, for a certain $\mu\in\C$.   Hence we have $\Ra^1=\Spl_\Ra(\Ra^1)=\RC^1$ and  from Theorem \ref{thm:identification-isomorphism-classes} we get $ \widehat{\Spl}(\sld\Mod^1_\tffr) \simeq \widehat{\Ra^1}$.

Let us consider $\nabla\in \Aut^1_{\C(z)}(W)$, then by Proposition \ref{prop:Aut1-sl2} it determines a representation $\rho^{(\mu)}\in\mathcal{RC}_\mu(W)$. Bearing in mind Proposition \ref{prop:paramtrization-Casimir-reps} and the fact that $\dim_{\C(z)}W=1$, we know that there is a bijection
	\begin{equation}\label{eq:RatCasMu}
	\C(z)^\times \longleftrightarrow \mathcal{RC}_\mu(W)
	\end{equation}
that maps $r(z)\in \C(z)^\times$ to the representation $\rho_r^{(\mu)}=r(z)\cdot\rho^{(\mu)}$ according to (\ref{eq:action-automorphism-on-rpresentation}). That is:
	\begin{equation}\label{eq:L-1r(z)}
	\rho_r^{(\mu)}(L_{-1}):=\frac{\pi_\mu(z)}{r(z-1)}\circ \nabla^{-1}, \quad 
	\rho_r^{(\mu)}(L_0):= z,\quad 
	\rho_r^{(\mu)}(L_1):={r(z)}\circ \nabla
	\, .\end{equation}

However, there can be different rational functions giving rise to isomorphic representations. 

\begin{thm}\label{t:equivalentrational}
Two rational functions $r_1(z), r_2(z)\in \C(z)^\times$ yield isomorphic representations $\rho_{r_1}^{(\mu)}$, $\rho_{r_2}^{(\mu)}$ on $W$, and we write $r_1(z)\sim r_2(z)$, if and only if there exist $\alpha_1,\ldots,\alpha_n\in \C$ and $a_1,\ldots, a_n\in {\mathbb Z}$ such that:
	\[
	\frac{r_2(z)}{r_1(z)}
	\,=\, \prod_{i=1}^n \frac{z-\alpha_i}{z+a_i-\alpha_i}
	\, . \]
\end{thm}

\begin{proof}
By the previous discussion, the rational representations defined by $r_1(z)$, $r_2(z)\in\C(z)$ are isomorphic if and only if there exists an isomorphism of $\sl(2)$-modules, $T:(W,\rho_{r_1})\overset{\sim}\to (W,\rho_{r_2})$, where $T$ is defined by $t(z)\in\C(z)^\times$. By Proposition \ref{prop:homs-casimir}, $T$ is a morphism of $\sl(2)$-modules if and only if $\rho_{r_2}^{(\mu)}(L_1)\circ T=T\circ\rho_{r_1}^{(\mu)}(L_1)$:
	\[
	r_2(z) \circ \nabla\circ t(z) \,=\, t(z) \circ r_1(z)\circ \nabla
	\]
or, what amounts to the same
	\[
	\frac{r_2(z)}{r_1(z)} \, =\, \frac{t(z)}{t(z+1)}
	\, . \]
The claim follows from the following proposition.
\end{proof}

\begin{prop}\label{prop:solution-functional-equation}
Let $f(z)\in\C(z)^\times$ be given. Then, the functional equation for $t(z)\in\C(z)$ defined by:
	\[
	\frac{t(z)}{t(z+1)}\,=\, f(z)
	\]
has a solution if and only if there are $\alpha_1,\ldots,\alpha_n\in \C$ and $a_1,\ldots, a_n\in {\mathbb Z}$ such that:
	\[
	f(z)\,=\, \prod_{i=1}^n \frac{z-\alpha_i}{z+a_i-\alpha_i}
	\]
and, if this is the case, then the solution is:
	\[
	t(z)\,=\, c \cdot \prod_{i=1}^n P(z-\alpha_i, a_i)
	\]
where $c\in\C^\times$ is arbitrary and $P(z,m)\in\C(z)$ is the $m$-th Pochhammer rational function.  
\end{prop}

Recall that the Pochhammer function is defined by:
		\begin{equation}\label{e:Poch}
		P(z,n) \,:=\, 
		\begin{cases} z (z+1)\cdot\ldots \cdot (z+n-1)  & \text{ for } n> 0 \\
		1  & \text{ for } n=0 \\
		\frac1{(z+n)(z+n+1)\dots (z-1)} & \text{ for } n\leq -1
		\end{cases}
		\end{equation}

\begin{proof}
If $f(z)$ and $t(z)$ are as in the statement, it is clear that the functional equation holds by the properties of the Pochhammer rational functions. 

Assume that the functional equation has a solution, $t(z)$, and write it as $t(z)=\frac{p(z)}{q(z)}$ with $p(z),q(z)\in \C[z]$, we get
	\[
	f(z)=\frac{p(z)}{p(z+1)}\cdot\frac{q(z+1)}{q(z)}
	\, .\]

If $p(z)=\beta\cdot(z-\beta_1)\cdots(z-\beta_m)$, $q(z)=\gamma\cdot(z-\gamma_1)\cdots(z-\gamma_n)$ then 
  $$f(z)=\frac{(z-\beta_1)}{(z-\beta_1+1)}\cdots\frac{(z-\beta_m)}{(z-\beta_m+1)}\cdot\frac{(z-\gamma_1+1)}{(z-\gamma_1)}\cdots\frac{(z-\gamma_n+1)}{(z-\gamma_n)}$$and this has the required form.
\end{proof}

\subsection{Rationalization of polynomial representations}  
  
In this section we study when a $1$-dimensional rational Casimir representation is the rationalization of a polynomial Casimir representation of rank $1$. That is, given $(W,\rho_r)$ with $W\simeq \C(z)$ and $\rho_r$ defined by $r(z)\in\C(z)$ as in \eqref{eq:L-1r(z)}, we wonder whether there exists $(V,\rho)$ with $V\simeq \C[z]$ and an isomorphism of $\sl(2)$-modules $(W,\rho_r)\simeq F_\rat(V,\rho)$ . 

\begin{thm}\label{t:ratio-poly}
Let $(W,\rho_r)$ be a $1$-dimensional rational Casimir representation defined by $r(z)\in\C(z)^\times$. There exists a polynomial Casimir representation of rank $1$, $(V,\rho')$, such that $(W,\rho_r)\simeq F_\rat(V,\rho')$ if and only if there are $\alpha_1,\ldots,\alpha_n\in \C$ and $a_1,\ldots, a_n\in {\mathbb Z}$ such that $r(z)$ is of the following four possible types:
\begin{enumerate}
\item[I)] $r(z)= \gamma\cdot\pi_\mu(z+1) \prod_{i=1}^n \frac{z+a_i-\alpha_i}{z-\alpha_i}$;
\item[II)] $r(z)= {\gamma}\cdot\alpha_\mu(z+1) \prod_{i=1}^n \frac{z+a_i-\alpha_i}{z-\alpha_i}$;
\item[III)] $r(z)= {\gamma}\cdot\beta_\mu(z+1)  \prod_{i=1}^n \frac{z+a_i-\alpha_i}{z-\alpha_i}$;
\item[IV)] $r(z)= {\gamma} \prod_{i=1}^n \frac{z+a_i-\alpha_i}{z-\alpha_i}$.
\end{enumerate}with $\gamma\in\C^\times$ and $\pi_\mu(z)=\alpha_\mu(z)\cdot\beta_\mu(z)$ is the product of degree 1 monic polynomials.
\end{thm}

\begin{proof}
First, recall the following classification of polynomial Casimir representations of rank $1$ given in \cite[Theorem~5.3]{Plaza-Tejero-Con}: the Casimir representations of level $\mu$ on a rank one free $\C[z]$-module $V$,  are:
\begin{enumerate}
\item[I)] $\rho(L_{-1})=\frac1{\gamma}\cdot\nabla^{-1},\quad\quad\quad\ \rho(L_0)=z,\quad  \rho(L_{1})={\gamma}\cdot\pi_\mu(z+1)\,\nabla.$

\item[II)] $\rho(L_{-1})=\frac1{\gamma}\cdot\beta_\mu(z)\,\nabla^{-1},\quad \rho(L_0)=z,\quad  \rho(L_{1})={\gamma}\cdot\alpha_\mu(z+1)\,\nabla.$

\item[III)] $\rho(L_{-1})=\frac1{\gamma}\cdot\alpha_\mu(z)\,\nabla^{-1},\quad \rho(L_0)=z,\quad  \rho(L_{1})={\gamma}\cdot\beta_\mu(z+1)\,\nabla.$
\item[IV)] $\rho(L_{-1})=\frac1{\gamma}\cdot\pi_\mu(z)\,\nabla^{-1},\quad \rho(L_0)=z,\quad  \rho(L_{1})={\gamma}\cdot\nabla.$
\end{enumerate} where $\gamma\in\C^\times$ is arbitrary and $\alpha_{\mu}(z)$, $\beta_{\mu}(z)$ are monic polynomials of degree 1 such that $\pi_\mu(z)=\alpha_\mu(z)\cdot\beta_\mu(z)$.

Now, let  $(W,\rho_r)$ be defined by the relations \eqref{eq:L-1r(z)} with $r(z)\in\C(z)$. Suppose that there exists  a polynomial Casimir representations of rank $1$, $(V,\rho')$, of type I as above such that $F_\rat(V,\rho')\simeq (W,\rho)$. By Theorem~\ref{t:equivalentrational}, it follows that $(W,\rho_r)$ is equivalent to $F_\rat(V,\rho')$ if and only if there exist $\alpha_1,\ldots,\alpha_n\in \C$ and $a_1,\ldots, a_n\in {\mathbb Z}$ such that:
	\[
	\frac{\gamma\cdot\pi_\mu(z+1)}{r(z)}
	\,=\, \prod_{i=1}^n \frac{z-\alpha_i}{z+a_i-\alpha_i}
	\, , \]
proving  the claim for type I. The proof for types II, III and IV is similar.
\end{proof}

\begin{rem}
The description of all $\sl(2)$-submodules of an arbitrary $1$-dimensional  rational representation is rather complicated. Let us illustrate this claim with the following example.  We begin by constructing non-trivial rank $1$ torsion free $\sl(2)$-sub\-mo\-dules of any $1$-dimensional rational Casimir representation. Let $r(z)\in\C(z)^\times$ and consider its associated rational $\sl(2)$-module $(W,\rho_r)$, with $\rho_r$ given by (\ref{eq:L-1r(z)}).  Every $0\neq w\in W$ generates a non trivial $\sl(2)$-submodule  $V_w$ of $(W,\rho_r)$ that is also a Casimir module, and therefore we have $V_w=\mathbb A(\mu)\cdot w$. Taking into account the graded decomposition of the algebra $\mathbb A(\mu)$ given in Section \ref{section:Casimir-modules}, it follows that $V_w$ is the $\C[z]$-submodule of $W$ generated by $\{w_m\}_{m\in\Z}$ with

$$w_m:=\begin{cases}
\rho_r(L_{-1})^{-m}(w)  & \text{ for } m< 0,\\ 
w  & \text{ for } m= 0, \\ 
\rho_r(L_1)^m(w) & \text{ for } m< 0.
\end{cases}$$

A straightforward computation shows that 
	\[
	w_m\,:=\, \begin{cases}
	P\left(\frac{r(z)}{\pi_\mu(z+1)},m\right)\cdot \nabla^m(w)& \text{ for } m< 0, \\
	P(r(z),m)\cdot \nabla^m(w)  & \text{ for } m \geq 0.
	\end{cases}
	\]
Here $P(\xi(z),m)\in\C(z)$ denotes the $m$-th Pochhammer expression on the rational fraction $\xi(z)\in\C(z)$. That is $$P(\xi(z),m):=\, \begin{cases}
	\frac{1}{\xi(z+m)\xi(z+m+1)\cdots\xi(z-1)} & \text{ for } m< 0, \\
	1  & \text{ for } m= 0, \\ 
	\xi(z)\xi(z+1)\cdots \xi(z+m-1)  & \text{ for } m > 0.
	\end{cases}$$

It is worth pointing out that $V_w$ is not a finite type $\C[z]$-module. However, since $V_w\simeq \mathbb A(\mu)/I_w$, where $I_w$ is the left ideal of $\mathbb A(\mu)$ that annihilates $w$, it follows from \cite[Theorem 4.26 pag. 128]{Mazor} that $V_w$ is a finite length $\sl(2)$-module. On the other hand, since $F_\rat(V_m,{\rho_r}_{|V_w})=(W,\rho)$  is $\Ra$-irreducible, we conclude that $(V_w,{\rho_r}_{|V_w})$ is a purely irreducible $\sl(2)$-module. Moreover, every $\sl(2)$-submodule $V$ of $(W,\rho_r)$ is a Casimir module, and therefore it is of the form $$V=\sum_{i\in I} V_{w_i},$$ for a certain family $\{w_i\}_{i\in I}$ of elements of $W$.
\end{rem}

%%%%%%%%%%%%%%%%%%%%%%%%%%%%%%%%%

\subsection{The $\Ext^1_{\mathcal R}$ groups for rank 1 rational representations.}

Given $r_i(z)\in \C(z)$ for $i=1,2$, let us denote by $W_i$ the rational Casimir $\sl(2)$-module of level $\mu$ defined on a one dimensional  $\C(z)$-vector space $W$ by the representation $\rho_i(L_1):= r_i(z) \nabla$  built as in \eqref{eq:L-1r(z)}.

By Proposition \ref{prop:extension-group-rational-Casimir} one has $$
\Ext^1_{\mathcal R}(W_2,W_1)=\left(\End^1_{\C(z)}(W)/\End_{\C(z)}(W)\right)\times\Hom_{\sl(2)}(W_2,W_1),$$ where $\alpha\in \End_{\C(z)}(W)$ acts on $B\in \Hom^1_{\C(z)}(W)$ by \begin{equation}\label{eq:equivalence-condition}
\alpha\cdot B:= B+\alpha\circ \rho_2(L_1)-\rho_1(L_1)\circ\alpha.
\end{equation}

By Proposition \ref{prop:hom-irreds-is-one-dimensional}, if $W_1$ is not isomorphic to $W_2$, then  $\Hom_{\sl(2)}(W_2,W_1)=0$,  whereas $\Hom_{\sl(2)}(W_2,W_1)$ is a one dimensional complex vector space if $W_1$ is isomorphic to $W_2$.

On the other hand, $\C(z)=\End_{\C(z)}(W)$ and via $\nabla$ we have an identification of $\C(z)$-vector spaces $\C(z)\xrightarrow{\sim} \End^1_{\C(z)}(W)$ that maps $b(z)\in\C(z)$ to $b(z)\nabla$. Therefore, given $\alpha(z)\in\C(z)$, the equation (\ref{eq:equivalence-condition}) is equivalent to $$\alpha\cdot b=b+r_2(z)\alpha(z) -r_1(z)\alpha(z+1).$$ Whence, two elements $b_1(z),b_2(z)\in\C(z)$ define the same class in $\Ext^1_{\mathcal R}(W_2,W_1)$ if and only if they satisfy the equality $$r_1(z)\alpha(z+1)-r_2(z)\alpha(z)=b_1(z)-b_2(z).$$ This equation is a first order linear difference equation of the form $$\alpha(z+1)-r(z)\alpha(z)=\xi(z),$$ where $r(z)=\frac{r_2(z)}{r_1(z)},\xi(z)=\frac{b_1(z)-b_2(z)}{r_1(z)}\in\C(z)$. By Theorem \ref{t:equivalentrational}, $W_1$ is isomorphic to $W_2$ if and only if there exist $t(z)\in\C(z)$ such that $\frac{r_2(z)}{r_1(z)} \, =\, \frac{t(z)}{t(z+1)}$. Substituting above we get  the first order linear difference equation $\phi(z+1) -\phi(z)=s(z)$ where $\phi(z)=t(z)\alpha(z)$ and $s(z)=t(z+1)\xi(z)$. Now we have the:

\begin{lem}
Given $s(z)\in \C(z)$, there exists a rational function $\phi(z)\in \C(z)$ solving the difference equation:		\[
		 \phi(z+1) -\phi(z)
		\,=\,
		s(z) 
		\]
if and only if there are $ \alpha_1,\ldots,\alpha_n\in \C$ and $a_1,\ldots, a_n\in \Z$ such that:
	\[
	\operatorname{Res}_{z=\alpha_i} s(z)(z-\alpha_i)^j \operatorname{d}z
	\,=\, 
	-\operatorname{Res}_{z=\alpha_i-a_i} s(z)(z-(\alpha_i-a_i))^j\operatorname{d}z
	\qquad \forall j\geq 0.
	\]
\end{lem}

\begin{proof}
	Recall that, by the partial fraction decomposition, a rational function can be expressed as a sum of a polynomial plus fractions of type $\frac{\beta}{(z-\alpha)^k}$. The polynomial will be called the polynomial part of the rational function. Note that it suffices to prove the statement for these two cases; namely, for polynomials and for rational functions with zero polynomial part. 
	
	Let us show that there is a polynomial $\phi(z)$ solving the equation for  an arbitrary polynomial $s(z)\in \C[z]$. Having in mind that the set of Pochhammer's functions $\{P(z,n) \vert n\geq 0\}$ (see~\eqref{e:Poch}) is a basis of $\C[z]$ and that:
		\[
		P(z+1,n) - P(z,n)\,=\, P(z+1,n-1)
		\]
	the statement follows. 
	
	Now, let us deal with the case of rational functions with zero polynomial part. 
	Let $\phi(z)$ be a solution. Let us write:
		\[
		\phi(z)\,:=\, \sum_{j=1}^n \sum_{i=r_j}^{-1}\frac{\beta_{ji}}{(z-\alpha_j)^i}
		\] 
	It is straightforward that $s(z)$ satisfies the condition of the statement. 
	
	Conversely, let $s(z)$ be rational function with zero polynomial part fulfilling the condition; that is,
		\[
		s(z)\,:=\, \sum_{i=1}^n 
		\sum_{j=r_i}^{-1} \beta_{ij}\left(\frac{1}{(z-(\alpha_i-a_i) )^j}-\frac{1}{(z-\alpha_i)^j}\right)	
		\]
	for certain complex numbers $\beta_{ij}\in\C$.	It is easy to check that it can be assumed that $n=1$ and $a_1=1$. 	Then $\phi(z):= \sum_{j=r_1}^{-1}\frac{\beta_{1j}}{(z-\alpha_i)^j}$ is the solution.	
\end{proof}

As a consequence of the previous results we get the:

\begin{cor}\label{cor:non-finite-Exts-rational-modules} If $W_1, W_2$ are $1$-dimensional rational Casimir $\sl(2)$-modules such that $W_1\simeq W_2$, then  $\frac{1}{(z-\alpha)^i}, \frac{1}{(z-\beta)^j}\in\C(z)$ define the same class in $\Ext^1_\Ra(W_2,W_1)$ if and only if $\alpha-\beta\in \Z$ and $i=j$. 

In particular, for a $1$-dimensional rational Casimir $\sl(2)$-module $W$,  one has:
	\[
	\dim_{\C} \operatorname{Ext}^1_{\mathcal R}(W,W) \,=\, \infty\, .
	\]
\end{cor}

\begin{rem}
This is in sharp contrast with what happens for finite rank torsion free $\sl(2)$-modules of finite length,  for which Bavula proved in \cite{Bavula2} that all the $\Ext$'s groups are finite dimensional $\C$-vector spaces.
\end{rem}

%%%%%%%%%%%%%%%%%%%%%%%%%%%%%%%%%%%%%
\section{Picard Group and Grothedieck rings of rational representations}

We recall from Proposition~\ref{p:abelian} that the category of rational $\sld$-modules, ${\mathcal R}$, is a $\C$-linear abelian subcategory of the category  $\C(z)\text{-}\mathrm{Vect}_\mathrm{fd}$ of finite dimensional $\C(z)$-vector spaces. It is therefore natural to wonder what properties or constructions of $\C(z)\text{-}\mathrm{Vect}_\mathrm{fd}$ can be restricted to ${\mathcal R}$. 

 It is remarkable that, thanks to Theorem~\ref{thm:Jordan-Holder-Krull-Schmidt},  the tensor product  $\otimes$ of $\C(z)$-vector spaces does restrict to ${\mathcal R}$ and that it satisfies very nice properties that can be summarized by saying that $(\Ra,\otimes)$ is a closed symmetric monoidal category, see Theorem \ref{t:Rmonoidal}.  This makes possible to define the Picard group $\Pic(\Ra)$ of the category of rational representations, see \cite{May}. 
\subsection{The monoidal structure} We begin by constructing the tensor product of rational Casimir  $\sld$-modules.

\begin{defn}\label{defi:tensor-product-rational-casimir} Let $(W_1,\rho_1)$, $(W_2,\rho_2)$ be two rational Casimir $\sld$-modules of levels   $\mu_1,\mu_2\in\C$, respectively. We define on the $\C(z)$-vector space $W_1\otimes_{\C(z)} W_2$ the semilinear endomorphisms \begin{align*}
&(\rho_1\otimes\rho_2)(L_{-1}):=\frac{\pi_{\mu_1+\mu_2}(z)}{\pi_{\mu_1}(z)\pi_{\mu_2}(z)}\rho_1(L_{-1})\otimes\rho_2(L_{-1})\in\End^{-1}_{\C(z)}(W_1\otimes_{\C(z)} W_2),\\ &(\rho_1\otimes\rho_2)(L_{1}):=\rho_1(L_{1})\otimes\rho_2(L_{1})\in\End^{1}_{\C(z)}(W_1\otimes_{\C(z)} W_2).
\end{align*}
\end{defn}

An easy application of Proposition \ref{prop:only-one-condition-for-rational-Casimir} proves  that  $(W_1\otimes_{\C(z)} W_2,\rho_1\otimes\rho_2)$ is a rational Casimir  $\sld$-module  of level $\mu_1+\mu_2$ that we call  the rational Casimir tensor product of $(W_1,\rho_1)$ with $(W_2,\rho_2)$ and we denote it $(W_1,\rho_1)\otimes_{\C(z)}(W_2,\rho_2)$. Given $(W,\rho)\in\RC$ we denote by $\lev((W,\rho))\in\C$ its level, then for every $(W_1,\rho_1), (W_2,\rho_2)\in\RC$ we have $$\lev((W_1,\rho_1)\otimes_{\C(z)} (W_2,\rho_2))=\lev((W_1,\rho_1))+\lev((W_2,\rho_2)).$$

\begin{defn}\label{defi:hom-rational-casimir} Let $(W_1,\rho_1)$, $(W_2,\rho_2)$ be two rational Casimir $\sld$-modules of levels   $\mu_1,\mu_2\in\C$, respectively. We define on the $\C(z)$-vector space $\Hom_{\C(z)}(W_1,W_2)$ the semilinear endomorphisms such that for any $\varphi\in\Hom_{\C(z)}(W_1,W_2)$ one has  \begin{align*}
&\Hom_{\C(z)}(\rho_1,\rho_2)(L_{-1})(\varphi):=\frac{\pi_{\mu_2-\mu_1}(z)}{\pi_{\mu_2}(z)}\rho_2(L_{-1})\circ\varphi\circ \rho_1(L_{1}),\\ &\Hom_{\C(z)}(\rho_1,\rho_2)(L_{1})(\varphi):=\frac{1}{\pi_{\mu_1}(z+1)}\rho _2(L_{1})\circ\varphi \circ\rho_1(L_{-1}).
\end{align*}
\end{defn}

By mean of Proposition  \ref{prop:only-one-condition-for-rational-Casimir}  one shows that $(\Hom_{\C(z)}(W_1,W_2),\Hom_{\C(z)}(\rho_1,\rho_2))$ is a rational Casimir representation of level $\mu_2-\mu_1$ and we denote it  by$$ \Hom_{\C(z)}((W_1,\rho_1),(W_2,\rho_2)).$$ A lengthy but straightforward check of the axioms, see \cite[\S 6.1]{Borceux}, proves the following key result, where $(\C(z),\rho^{(0)})$ has been defined in Example~\ref{ex:rhomu}.

\begin{thm}
The category $\mathcal {RC}$  endowed with the tensor product  $\otimes$ of rational Casimir $\sld$-modules is a closed symmetric monoidal category with unit $(\C(z),\rho^{(0)})$. Moreover, given any two rational Casimir representations $(W_1,\rho_1)$, $(W_2,\rho_2)$, their internal Hom is given by  $ \Hom_{\C(z)}((W_1,\rho_1),(W_2,\rho_2))$; that is, one has
	\begin{align*}
	&\Hom_{\sld}( (W_0,\rho_0)\otimes_{\C(z)}(W_1,\rho_1)  ,(W_2,\rho_2))
	 \,=\,\\
	 =& \Hom_{\sld}((W_0,\rho_0),    \Hom_{\C(z)}((W_1,\rho_1),(W_2,\rho_2))).
\end{align*}
\end{thm}

It is straightforward to show that  Definitions~\ref{defi:tensor-product-rational-casimir} and~\ref{defi:hom-rational-casimir} can be used also to endow the tensor product of generalized rational Casimir $\sl(2)$-modules with the structure of an $\sld$-module.  Furthermore, one has the following results. 

\begin{lem}\label{l:tensorGenCas} If $(W_1,\rho_1)\in\mathcal{RC}_{\mu_1}^{(n_1)}$, $(W_2,\rho_2)\in\mathcal{RC}_{\mu_2}^{(n_2)}$, then one has:
	\begin{enumerate}
		\item the tensor product $(W_1\otimes_{\C(z)}W_2,\rho_1\otimes\rho_2)$ is a generalized rational Casimir $\sl(2)$-module and $$(W_1\otimes_{\C(z)}W_2,\rho_1\otimes\rho_2)\in \mathcal{RC}_{\mu_1+\mu_2}^{(n)}$$ where $n\leq n_1+n_2+\min\{n_1,n_2\}-2$.
		\item  the $\sl(2)$-module $ \Hom_{\C(z)}((W_1,\rho_1),(W_2,\rho_2))$ is a generalized rational Casimir $\sl(2)$-module that is their internal $\Hom$ and $$ \Hom_{\C(z)}((W_1,\rho_1),(W_2,\rho_2))\in \mathcal{RC}_{\mu_2-\mu_1}^{(n)}$$ where $n\leq n_1+n_2+\min\{n_1,n_2\}-2$.
		\end{enumerate}
\end{lem}

\begin{thm}\label{t:Rmonoidal}
The category $\mathcal {R}$  endowed with the tensor product  $\otimes$ of  rational $\sld$-modules is a closed symmetric monoidal category with unit $(\C(z),\rho^{(0)})$. Moreover, given any two rational representations $(W_1,\rho_1)$, $(W_2,\rho_2)$, their internal Hom is given by  $ \Hom_{\C(z)}((W_1,\rho_1),(W_2,\rho_2)).$
\end{thm}

\begin{proof}
If $(W_1,\rho_1)$, $(W_2,\rho_2)$ are two rational $\sl(2)$-modules and $$W_1=\bigoplus_{i=1}^k W_{1,\mu_i}^{(m_i)},\quad W_2=\bigoplus_{j=1}^l W_{2,\nu_j}^{(n_j)}$$ are their decompositions into a direct sum of generalized rational Casimir modules described in Theorem \ref{thm:finite+decomp}, then we introduce an $\sl(2)$-module structure on their tensor product $$W_1\otimes_{\C(z)} W_2=\bigoplus_{i=1, j=1}^{k,l}W_{1,\mu_i}^{(m_i)}\otimes_{\C(z)}  W_{2,\nu_j}^{(n_j)}$$by declaring this to be a direct sum of $\sl(2)$-modules and endowing each component with the $\sl(2)$-module structure defined by the tensor product of generalized rational Casimir modules, see Lemma~\ref{l:tensorGenCas}. In a similar way one has $$\Hom_{\C(z)}(W_1,W_2)=\bigoplus_{i=1, j=1}^{k,l}\Hom_{\C(z)}(W_{1,\mu_i}^{(m_i)}, W_{2,\nu_j}^{(n_j)})$$ and we introduce an $\sl(2)$-module structure on it  by declaring again this to be a direct sum of $\sl(2)$-modules and endowing each component with the $\sl(2)$-module structure defined by the internal Hom of generalized rational Casimir modules, see Lemma~\ref{l:tensorGenCas}. We denote this $\sl(2)$-module structure by $ \Hom_{\C(z)}((W_1,\rho_1),(W_2,\rho_2)$. The claims follow straightforwardly.
\end{proof}

\begin{rem}
Recalling Theorem~\ref{t:ratio-poly}, we observe that if we tensor out two rational representations that come from polynomial representations through the rationalization functor, then we obtain a rational representation which may not arise from a polynomial representation. Indeed, it is straightforward to see that the tensor product of  two rational representations of type I can not be obtained as the rationalization of a polynomial representation. 
\end{rem}

\subsection{The Picard group} An essentially small closed symmetric monoidal category $(\mathcal C,\otimes)$ has a Picard group $\Pic(\mathcal C)$ that was introduced by May in \cite{May}. In order to describe the Picard group $\Pic(\Ra)$ of $(\Ra,\otimes)$, one defines the dual of a rational representation $(W,\rho)$ as $$(W,\rho)^*:=\Hom_{\C(z)}((W,\rho),(\C(z),\rho^{(0)}))$$ where as before  $\Hom_{\C(z)}(-,-)$ denotes the internal $\Hom$ of $(\Ra,\otimes)$. There is a canonical map $$\nu\colon (W,\rho)^*\otimes (W,\rho)\to \Hom_{\C(z)}( (W,\rho), (W,\rho))$$ and one says that $(W,\rho)$ is dualizable if $\nu$ is an isomorphism. One checks straightforwardly that this is always the case. 

\begin{prop}\label{prop:dualizable} 
Every object of the symmetric monoidal category $(\Ra,\otimes)$ is dualizable.
\end{prop}

One says that $(W,\rho)\in\Ra$ is invertible if there exists $(W',\rho')\in\Ra$ such that $$(W,\rho)\otimes_{\C(z)}(W',\rho')\simeq (\C(z),\rho^{(0)}).$$ We denote by $\Inv(\Ra)$ the subcategory of $\Ra$ formed by the invertible elements. Taking $\C(z)$-dimensions in the above identity it follows that we must necessarily have $\dim_{\C(z)}(W,\rho)= \dim_{\C(z)}(W',\rho')=1$ and as we have seen in Section \ref{subsec:one-dimensional-stratum}, this implies that both representations are Casimir.  Moreover, one proves, see \cite[Theorem 2.6, Lemma 2.9]{May}, that the invertibility of $(W,\rho)$ forces the equality $(W',\rho')=(W,\rho)^*$. 

\begin{thm}\label{thm:invertible-elements}
The subcategory $\Inv(\Ra)$ of invertible elements of the closed symmetric monoidal category $(\Ra,\otimes)$ is  the stratum $\mathcal R^1$ of one dimensional rational representations, or equivalently the one dimensional rational Casimir representations $\RC^1$. 
\end{thm}

\begin{proof}
We have already seen that an invertible representation necessarily belongs to $\Ra^1=\RC^1$. On the other hand, given $(W,\rho)\in \RC^1_\mu$ since we have seen in Proposition \ref{prop:dualizable} that every object is dualizable, we have a canonical isomorphism $$\nu\colon (W,\rho)^*\otimes_{\C(z)} (W,\rho)\xrightarrow{\sim} \Hom_{\C(z)}( (W,\rho), (W,\rho)).$$ Composing this with the trace isomorphism $$\Tr\colon \Hom_{\C(z)}( (W,\rho), (W,\rho))\xrightarrow{\sim} (\C(z),\rho^{(0)})$$ we get the desired isomorphism $(W,\rho)^*\otimes_{\C(z)} (W,\rho)\xrightarrow{\sim}(\C(z),\rho^{(0)})$.
\end{proof}

We recall now the definition of the Picard group, see \cite[Definition 2.10]{May}.
\begin{defn} The Picard group $\Pic(\Ra)$ of the closed symmetric monoidal category $(\Ra,\otimes)$ is the set of isomorphism classes $\widehat\Inv(\Ra)=\widehat{\Ra}^1=\widehat{\RC}^1$ of invertible objects endowed with the product and inverses given by $$[(W,\rho)]\cdot[(W',\rho')]=[(W,\rho)\otimes_{\C(z)}(W',\rho')],\quad [(W,\rho)]^{-1}=[(W,\rho)^*],$$ for every $ [(W,\rho)], [(W',\rho')]\in\widehat\Inv(\Ra)$. 
\end{defn}

Now we state the main result of this section.
{thm:Picard-group} {thm:invertible-elements}
\begin{thm}\label{thm:Picard-group} The level morphism gives rise to  a short exact sequence $$0\to \Pic^0(\Ra)\to \Pic(\Ra)\xrightarrow{\lev}\C\to 0.$$
The Picard group $\Pic(\Ra)$ gets identified with the group $\widehat{\mathcal{RC}}(\C(z))=\C\times (\C(z)^\times/\!\sim)$. This induces a group morphism section  $\sigma\colon \C\to \Pic(\Ra)$ given by $\sigma(\mu)=(\C(z),\rho^{(\mu)})$ that splits the short exact sequence.  Furthermore, there are isomorphism of groups
	\begin{equation}\label{e:PicC(q)}
	\Pic^0(\Ra)\xrightarrow{\sim}\C(q)^\times_0,\quad \operatorname{Pic}({\mathcal{R}})
	\, \overset{\sim}\longrightarrow \, \C\times
	\C(q)^\times_0
	\end{equation} compatible with the above exact sequence, where $\C(q)^\times_0$ are the rational functions without zeros or poles at $0$. 
\end{thm}

\begin{proof}
We have the identifications $\widehat\Inv(\Ra)=\widehat{\RC}^1=\widehat{\RC}(\C(z))$. Moreover, we have seen that there is a bijection between ${\RC}_\mu(\C(z))$ and $ \C(z)^\times$, see~\eqref{eq:RatCasMu} and~\eqref{eq:L-1r(z)}. Hence, we have $\widehat{\RC}(\C(z))=\C\times (\C(z)^\times/\!\!\sim)$, where the $\C$ component  is  the level of the representation and $\sim$ is the equivalence relation described in Theorem \ref{t:equivalentrational}.

Let us show the second claim. Let $r_1(z),r_2(z)\in\C(z)$ be such that they yield isomorphic $\sl(2)$-representations $\rho^{(\mu)}_{r_1}$, $\rho^{(\mu)}_{r_2}$ on $W$; hence, by Theorem~\ref{t:equivalentrational}, it means that there  are $\alpha_1,\ldots, \alpha_n\in \C$ and $a_1,\ldots, a_n\in \Z$ such that 
	\begin{equation}\label{e:equiv}
	\frac{r_1(z)}{r_2(z)} = \prod_{i=1}^n \frac{z-\alpha_i}{z+a_i-\alpha_i} \, .
	\end{equation}
This equation holds if and only if the following three conditions are fulfilled:  1) the zeroes of $r_1(z)$ and those of $r_2(z)$ are equal mod $\Z$; 2) the poles of $r_1(z)$ and those of $r_2(z)$ are equal mod $\Z$; 3) the quotient $r_1^\infty$ of the leading coefficient of the numerator of $r_1(z)$ by the leading coefficient of its denominator  is equal to the analogous quotient $r_2^\infty$ for $r_2(z)$. For a non-zero rational function $r(z)\in \C(z)^\times$, define 
	\[
	\tilde{r}(q) \,:= 
	r_\infty  \frac{\prod_i\big(q-  \exp( -2\pi{\bf{i}} \alpha_i)\big)} {\prod_j\big(q-  \exp( -2\pi{\bf{i}} \beta_j)\big)}
	\,\in \, \C(q)^\times_0
	\]
where  $\alpha_1, \ldots , \alpha_n$ is the set of zeroes of $r(z)$, $\beta_1,\ldots,\beta_m$ its set of poles, $r_\infty$ the quotient of the leading coefficient of the numerator of $r(z)$ by the leading coefficient of its denominator and $\bf{i}:=\sqrt{-1}$.  This gives a surjection $\tilde{(\,)}\colon \C(z)^\times\to \C(q)^\times_0$.

Moreover, the representations  associated to $r_1(z)$ and $r_2(z)$ are isomorphic if and only if $ \tilde{r_1}(q)= \tilde{r_2}(q)$. Hence, there is a bijection of sets $\widehat{\mathcal{RC}}(\C(z)) \overset{\sim}\to\C\times  \C(q)^\times_0$. 

It remains to show that it is a homomorphism of groups. But this is straightforward since the trivial representation, which is given by  $\rho^{(0)}(L_1)=\nabla$; i.e. $r(z)=1$, corresponds to $\tilde{r}(q)=1$ and the tensor product of representations corresponds  to the product of the associated functions.  
\end{proof}

\begin{rem}
Pursuing the above identifications, one sees that~\eqref{e:PicC(q)} establishes a one to one correspondence between the rational Casimir representations of level $0$ arising from polynomial ones and the following rational functions, where $\gamma\in\C^\times$:
	\[
	\begin{aligned}
	\{ \gamma, \gamma (q +  \exp(\pi{\bf{i}} \sqrt{1+4\mu})), & 
	\gamma (q+  \exp( - \pi{\bf{i}} \sqrt{1+4\mu})) ,
	\\  &
	\gamma (q +  \exp(\pi{\bf{i}} \sqrt{1+4\mu})) (q+  \exp( - \pi{\bf{i}} \sqrt{1+4\mu}))
	\}.
	\end{aligned}\]
\end{rem}

The identification $\widehat F_\rat\colon  \widehat{\Spl}(\sld\Mod^1_\tffr) \xrightarrow{\sim} \widehat{\Ra^1}=\Pic(\Ra)$ given in Theorem \ref{thm:identification-isomorphism-classes} makes possible to give the following formal:

\begin{defn} The Picard group of the category $\sld\Mod_\tffr$ is the group $$\Pic(\sld\Mod_\tffr)=\widehat{\Spl}(\sld\Mod^1_\tffr)$$ obtained by declaring $\widehat F_\rat\colon  \widehat{\Spl}(\sld\Mod^1_\tffr) \xrightarrow{\sim} \widehat {\Ra^1}$ to be a group isomorphism.
\end{defn}

The existence of this group structure on $\widehat{\Spl}(\sld\Mod^1_\tffr)$ might point to the existence of an appropriate symmetric monoidal structure on the category $\sld\Mod_\tffr$ such that its Picard group is the one introduced above. We will analyze this in future work.

\subsection{The Grothendieck rings}\label{subsect:Grothendieck-rings} The monoidal structure $(\Ra,\otimes)$ of the category  of rational representations allow us to introduce a ring structure on $K^\oplus_0(\Ra)$ with $\oplus$ as addition and $\otimes$ as multiplication. One says that $(K^\oplus_0(\Ra),\oplus,\otimes)$ is the additive Grothendieck ring of $\Ra$. Taking into account that every rational representation is dualizable, we have a natural map of semi-rings $\alpha\colon \Iso(\Ra)\to K(\Ra)$. For more details see \cite[Section 3]{May}. 

One says  that two rational $\sl(2)$-modules $W_1$, $W_2$ are stably isomorphic if there exist a rational module  $W$ such that $W_1\oplus W\simeq W_2\oplus W$. We have proved in Theorem \ref{thm:Jordan-Holder-Krull-Schmidt}  that $\Ra$ is a Krull-Schmidt category and thus $\Ra$ satisfies the cancellation property; that is, stably isomorphic rational representations are isomorphic. Therefore, from \cite[Propositions 3.4, 3.6]{May} it follows that the natural map $\alpha\colon \Iso(\Ra)\to K(\Ra)$ is injective and induces an injective  group morphism $\beta\colon \Pic(\Ra)\to K_0^\oplus(\Ra)^{\times}$ into the group of invertible elements giving rise  to a commutative diagram $$\xymatrix{\Pic(\Ra)\ar@{^{(}->}[r]\ar@{^{(}->}[d]_\beta & \Iso(\Ra)\ar@{^{(}->}[d]^\alpha\\ K_0^\oplus(\Ra)^{\times}\ar@{^{(}->}[r] & K_0^\oplus(\Ra).}$$ This is a pullback diagram because, thanks to the cancellation property of $\Ra$, given $W\in\Ra$ one has that $\alpha([W]) \in K_0^\oplus(\Ra)^{\times} $ if and only if $[W]\in\Pic(\Ra)$. The monoidal structure of $\Ra$ is compatible with short exact sequences since every $\C(z)$-vector space is flat. This implies that the kernel of the natural surjective map $\pi^\oplus\colon K_0^\oplus(\Ra)\to K_0(\Ra)\to 0$ is an ideal and thus $K_0(\Ra)$ is a quotient ring of $K_0^\oplus(\Ra)$.
\bibliographystyle{amsplain}

\end{document}